\documentclass[12pt]{amsart}
\title[Green's function of the Parabolic Anderson Model]{The Green's function of the parabolic Anderson model and the continuum directed polymer}

\input{SHE-macros}

\usepackage{datetime}

\author[T.~Alberts]{Tom Alberts}
\address{Tom Alberts\\ University of Utah\\  Mathematics Department\\ 155 S 1400 E\\   Salt Lake City, UT 84112\\ USA.}
\email{alberts@math.utah.edu }
\urladdr{http://www.math.utah.edu/~alberts}
\thanks{T.~Alberts was partially supported by National Science Foundation grant DMS-1811087}

\author[C.~Janjigian]{Christopher Janjigian}
\address{Christopher Janjigian\\ Purdue University\\  Department of Mathematics \\ 150 N University St\\ West Lafayette, IN 47901\\ USA.}
\email{cjanjigi@purdue.edu}
\urladdr{http://www.math.purdue.edu/~cjanjigi}
\thanks{C~Janjigian was partially supported by National Science Foundation grant DMS-2125961 and Simons Foundation grant MPS-TSM-00012155.}

\author[F.~Rassoul-Agha]{Firas Rassoul-Agha}
\address{Firas Rassoul-Agha\\ University of Utah\\  Mathematics Department\\ 155S 1400E\\   Salt Lake City, UT 84112\\ USA.}
\email{firas@math.utah.edu}
\urladdr{http://www.math.utah.edu/~firas}
\thanks{F.\ Rassoul-Agha was partially supported by National Science Foundation grants DMS-1811090, DMS-2054630, and DMS-2450951.}

\author[T.~Sepp\"al\"ainen]{Timo Sepp\"al\"ainen}
\address{Timo Sepp\"al\"ainen\\ University of Wisconsin--Madison\\  Department of Mathematics\\ 480 Lincoln Drive \\  Madison, WI 53706\\ USA.}
\email{seppalai@math.wisc.edu}
\urladdr{http://www.math.wisc.edu/~seppalai}
\thanks{T.~Sepp\"al\"ainen was partially supported by National Science Foundation grants  DMS-1854619, DMS-2152362 and DMS-2448375,  by the Simons Foundation grant 1019133, and by the Wisconsin Alumni Research Foundation.}


\begin{document}
\begin{abstract} 
We study the behavior of the coupling of solutions to the stochastic heat equation/parabolic Anderson model (PAM) obtained through convolution with the Green's function. In this coupling, we show that the PAM with a (sub-)exponentially growing initial condition has conserved quantities given by the slopes at $\infty$ and $-\infty$. These are then connected to Hopf-Cole solutions to the KPZ equation and the existence, regularity, and continuity of the quenched continuum polymer measures.  We also show that the Green's function is strictly totally positive using a novel coupling argument based on the continuum directed polymer.
\end{abstract}

\maketitle

\setcounter{tocdepth}{2}
\tableofcontents

\section{Introduction}
This paper is concerned with regularity properties of three interconnected models which have attracted substantial interest over recent decades: the parabolic Anderson model (PAM), also known as the multiplicative stochastic heat equation, the Kardar-Parisi-Zhang (KPZ) equation, and the continuum directed polymer (CDRP). We refer the reader to the surveys \cite{Cor-12,Cor-16,Hal-Tak-15,Qua-Spo-15,Qua-12} for more information about the broader research area of the KPZ universality class.

The main object we study is the Green's function of the PAM, which solves the stochastic partial differential equation
\be\begin{aligned}\label{eq:SHE}
\partial_t \Sheb(t,x\viiva s,y) &= \frac{1}{2} \partial_{xx} \Sheb(t,x\viiva s,y) + \beta \Sheb(t,x\viiva s,y) W(t,x)  \\
\Sheb(s,x\viiva s,y) &= \delta(x-y),   
\end{aligned}\ee
$(s,y,t,x)\in \varset = \{(s,y,t,x) \in \bbR^4 : s \leq t\}$ and $\beta \in \R$. In this expression, $W$ is space-time white noise and $\delta$ is the Dirac delta measure at $0$.  

We use the behavior of the Green's function to control the behavior of the PAM started from a more general initial condition
\be\label{eq:SHEf}\begin{aligned}
\partial_t \Sheb(t,x\viiva s;\mu) &= \frac{1}{2} \partial_{xx} \Sheb(t,x\viiva s;\mu) + \beta \Sheb(t,x\viiva s;\mu) W(t,x)  \\
\Sheb(s,\cdot|s;\mu) &= \mu(\cdot),    
\end{aligned}\ee
where $\mu$ is a positive Borel measure through the superposition identity
\begin{align}
\Sheb(t,x\viiva s;\mu) &= \int_{-\infty}^\infty \Sheb(t,x\viiva s,y)\mu(dy). \label{eq:superpos}
\end{align}
When $\mu(dx)= f(x)dx$ for a sufficiently regular Borel measurable function $f$, which corresponds to function-valued initial conditions, we will instead write $\Sheb(t,x\viiva s;f)$ as shorthand. 

The PAM is connected to the KPZ equation, introduced in \cite{Kar-Par-Zha-86}, through formal computations which defines the physically relevant ``Hopf-Cole" solution. Ignoring the distributional structure of $W$, formal computation suggests that if $\Sheb(t,x|s;e^f)$ solves \eqref{eq:SHEf} for a Borel measurable $f$, then 
\begin{align}
\KPZb(t,x\viiva s;f)= \log \Sheb(t,x\viiva s; e^f)\label{eq:KPZdef}
\end{align}
solves the KPZ equation
\be\label{eq:KPZ}\begin{aligned}
\partial_t \KPZb(t,x\viiva s;f) &= \frac{1}{2} \partial_{xx}\KPZb(t,x\viiva s;f) + \frac{1}{2}(\partial_x \KPZb(t,x\viiva s;f))^2 + \beta\tspb W(t,x)  \\
\KPZb(s,x \viiva s;f) &= f(x). 
\end{aligned}\ee
The Hopf-Cole notion of solution defines the solution $\KPZb$ of \eqref{eq:KPZ} through \eqref{eq:KPZdef}. Hopf-Cole solutions arise as a limit of lattice and continuum models which lie in the KPZ class, see e.g., \cite{Ber-Gia-97,Alb-Kha-Qua-14-aop,Hai-Qua-18}, and is the standard notion of solution to \eqref{eq:KPZ}.

Formally, one expects that the solution to \eqref{eq:SHE} has a Feynman-Kac interpretation as
\be\label{eq:FK}\begin{aligned}
\Sheb(t,x\viiva s,y) &= \heat(t-s,x-y) \rnSheb(t,x\viiva s,y),  \ \text{ where }  
\\
\rnSheb(t,x\viiva s,y) &= \oE_{(s,y),(t,x)}^{BB}\bigg[:\!\exp\!:\bigg\{\beta \int_s^t W(u,X_u)du\bigg\}\bigg], 
\end{aligned}\ee 
and $\heat(t,x)$ is the heat kernel,
\begin{align*}
\heat(t,x) &= \frac{1}{\sqrt{2\pi t}} e^{- \frac{x^2}{2t}} \one_{(0,\infty)}(t).
\end{align*}
In the expression in \eqref{eq:FK}, $:\!\exp\!:$ denotes the Wick-ordered exponential, and the expectation is over Brownian bridge paths $X$ from $(s,y)$ to $(t,x)$. The precise meaning of ``Wick-ordered exponential" is recorded below as the chaos expansion \eqref{eq:rnchaos}. In this interpretation, $\rnSheb(t,x\viiva s,y)$ can be viewed as the partition function of a point-to-point directed polymer measure which is intuitively (but not truly \cite[Theorem 4.5]{Alb-Kha-Qua-14-jsp}) a Gibbsian perturbation of the Brownian bridge measure. The interpretation of the solution to \eqref{eq:SHE} as \eqref{eq:FK} was made rigorous in \cite{Alb-Kha-Qua-14-aop, Alb-Kha-Qua-14-jsp} when the initial and terminal conditions are fixed. The KPZ equation can be interpreted as governing the evolution of the free energy of this polymer.

To make the above discussion more precise, for $-\infty < s < t < \infty$ and $x,y \in \bbR$, the (\textit{quenched}) \textit{point-to-point polymer distribution} $\Polyb_{(s,y),(t,x)}$ is the probability measure on the space $\sC([s,t],\bbR)$ of continuous functions determined by the following finite-dimensional distributions,  for $s = t_0  < t_1 < \dots < t_k < t_{k+1} = t$ and with $x_0=x$, $x_{k+1}=y$: 
\be\label{eq:p2pdef}\begin{aligned}
&\Polyb_{(s,y),(t,x)}(X_{t_1} \in dx_1, \dots, X_{t_k} \in dx_k) =  \frac{\prod_{i=0}^k\Sheb(t_{i+1},x_{i+1} | t_i , x_i) }{\Sheb(t,x\viiva s,y)}dx_{1:k} \\
&\qquad= \frac{\prod_{i=0}^k\rnSheb(t_{i+1},x_{i+1} | t_i , x_i) }{\rnSheb(t,x\viiva s,y)} \cdot \frac{ \prod_{i=0}^k \heat(t_{i+1}-t_i, x_{i+1}-x_i)}{\heat(t-s,x-y)}\tspb dx_{1:k},  
\end{aligned}\ee
where we use the notation $d x_{1:k} \;=\; d x_1 \cdots d x_k$.

\subsection{Summary of main results}
With the above discussion of the models we consider in mind, the main contributions of this paper can be summarized as follows.
\begin{enumerate}[label={\rm(\roman*)}, ref={\rm\roman*}] 
\item\label{main:i} We show that the Polish space (see Appendix \ref{app:top})
\begin{align}
\ICM &= \bigg\{\text{non-zero positive Borel measures } \mu \text{ on } \bbR: \forall a >0, \int_{\bbR} e^{-a x^2}\mu(dx) <\infty \bigg\}, \label{eq:ICM}
\end{align}
previously studied in the context of the PAM by Chen and Dalang in \cite{Che-Dal-14,Che-Dal-15}, is the sharp class of non-explosive initial conditions for the PAM. By sharp, we mean that for all initial and terminal times $s<t$ and all $\beta$ simultaneously, \eqref{eq:superpos} is finite for all $x$ if $\mu \in \ICM$ and there is finite time blowup if $\mu \notin \ICM$. We show in Theorem \ref{prop:IC} that this blow-up time is the same as for the $\beta = 0$ case of the unforced heat equation.  
\item We show that the Polish space (see Appendix \ref{app:top}) of strictly positive continuous functions which represent measures in $\ICM$,
\begin{align}
\CICM &= \biggl\{f \in \sC(\R,(0,\infty)) : \forall a>0, \int_{\bbR} e^{-ax^2}f(x)dx <\infty \biggr\}, \label{eq:CICM}
\end{align}
is a natural domain for the PAM. This is natural in two distinct ways: (1) if $\mu \in \ICM$ then for all $s<t$, $\Sheb(t,\aabullet\viiva s;\mu) \in \CICM$ and (2), if $\mu$ is represented by a function in $\CICM$, $(\beta,\mu,s,t) \mapsto \Sheb(t,\aabullet\viiva s;\mu)$ is continuous. 
See Theorem \ref{thm:jcont}. 

As a consequence,the Polish space of logs of such functions,
\begin{align}
\CKPZ &= \biggl\{f \in \sC(\R, \R) : \forall a>0, \int_{\bbR} e^{f(x)-ax^2}dx <\infty \biggr\}. \label{eq:CKPZ}
\end{align}
is a natural domain for Hopf-Cole solutions to the KPZ equation in the sense that solutions to \eqref{eq:KPZ} from any function-valued initial condition $f$ either explode in finite time or can be viewed as continuous dynamical systems in $\CKPZ$.
\item Using our coupling and growth estimates, we prove conservation of slope at $\infty$ and $-\infty$. Specifically, we show that if there exist $\lambda_+,\lambda_- \in \bbR$ such that 
\[
\lim_{x\to \pm \infty} \frac{\log f(x)}{x} = \lambda_{\pm}, \qquad \text{ then } \qquad \lim_{x\to \pm \infty} \frac{\log \Sheb(t,x\viiva s;f)}{x} = \lambda_{\pm}
\]
 for all $s<t$ and $\beta \in \bbR$. 
\item We construct a coupling of the continuum directed polymer measures for all initial and terminal conditions and all inverse temperatures simultaneously. We also prove many basic properties including weak and total variation regularity in the initial and terminal conditions, variants of the Feller and strong Feller properties, simultaneous H\"older $1/2-$ path regularity, and stochastic monotonicity as the initial condition varies. Essentially all of these properties are new.
\item Through analysis of the polymer measures and an argument based on the Karlin-McGregor theorem, we prove that the map $(x,y) \mapsto \Sheb(t,x\viiva s,y)$ is strictly totally positive for all $t>s$ and $\beta\in \bbR$. 
More precisely, we show that for all $(x_1,\dots,x_n),(y_1,\dots,y_n) \in \bbW_n = \{(x_1,\dots,x_n) \in\R^n: x_1 < \dots < x_n\}$ and all $\beta \in\bbR$,
\begin{align}
 \det\big[\Sheb(t,x_i\viiva s,y_j) \big]_{i,j=1}^n  > 0.\label{eq:totpos}
\end{align}
Essentially the same strict total positivity appears in \cite{Lun-War-20} using a different argument. Our proof readily generalizes to other planar directed polymer models.
\end{enumerate}
This paper provides the setting for a companion paper by the last three authors, \cite{Jan-Ras-Sep-22-1F1S-}, where synchronization, a quenched one force -- one solution principle,  and a characterization of  ergodic stationary (modulo constants) distributions for the KPZ equation are proven. Combined with \cite{Dun-Sor-24-}, which rules out an exceptional class of measures not covered by the methods of \cite{Jan-Ras-Sep-22-1F1S-}, that paper shows that Brownian motions with drift are the only stationary distributions for the KPZ equation on the sharp space of non-explosive initial conditions described above. The results described above play a key role in that proof.

\subsection{Methods}
The key heuristic behind the regularity results discussed above is that if $t-s$ is small, then 
\[
\rnSheb(t,x\viiva s,y) = \oE_{(s,y),(t,x)}^{BB}\bigg[:\!\exp\!:\bigg\{\beta \int_s^t W(u,X_u)du\bigg\}\bigg]\]
should be close to one. This suggests that all of the singularity of $\Sheb(t,x\viiva s,y)$ as $t-s \to 0$ is contained in the heat kernel, $\heat(t-s,x-y)$. Using the fact that $\rnSheb(t,x\viiva s,y)$ is stationary in its spatial coordinates, it is possible to obtain upper and lower bounds for $\rnSheb(t,x\viiva s,y)$ and $1/\rnSheb(t,x\viiva s,y)$ which are uniform over compact intervals in time and sub-polynomial in space.  Our methods to prove these regularity estimates are classical and we do not need to use any of the pathwise solution theories which have been developed to handle rough equations. We prove Kolmogorov-Chentsov estimates from the chaos expansion of the solution to \eqref{eq:SHE}, keeping track of the growth rates of the constants. 

With the growth estimates in hand, we can then reduce most of the regularity questions to problems involving the heat semi-group, where standard analytic arguments imply regularity. 

Similarly, using the absolute continuity of finite dimensional distributions of the polymer with respect to Brownian Bridge in \eqref{eq:p2pdef}, most of the regularity statements for polymer measures can be reduced to uniform estimates for the finite dimensional distributions of Brownian bridge measures.

Our proof of strict total positivity comes from a coupling of the polymer measures and relies on the Karlin-McGregor \cite{Kar-McG-59} identity for the density function of independent copies of a Markov chain conditioned not to intersect. We provide a proof of this identity applicable to our setting that follows a now well-known argument originally due to Varadhan (recorded as Exercise 4.3.5 in \cite{Hou-etal-09}.).

\subsection{Related work}
Using the superposition principle to define general solutions to \eqref{eq:SHEf} has been studied previously. See in particular , \cite{Alb-Kha-Qua-14-jsp}, \cite[Lemma 1.18]{Cor-Ham-16}, and \cite{Hai-Lab-18}. \cite{Hai-Lab-18} proves a regularity result of the solution semi-group to \eqref{eq:SHE}  similar to our Theorem \ref{thm:jcont}, which is one of our main results discussed as point \eqref{main:i} above. In that result, the space on which the solution semi-group acts ($\ICM$ or $\CICM$) is replaced by a weighted Besov space. The construction in \cite{Hai-Lab-18} and the resulting coupling is sharper than ours in terms of local regularity of the solution (allowing negative index distribution initial conditions), but does not cover the sharp growth conditions that were of interest to us from the ergodic theory perspective. 

There has been extensive past work on mild solutions to \eqref{eq:SHE}. The first paper to study mild solutions was by Bertini and Cancrini in \cite{Ber-Can-95}, with subsequent work by Bertini and Giacomin \cite{Ber-Gia-97} allowing for random initial conditions which include the Brownian stationary distributions. The state-of-the-art for non-random initial conditions is the previously mentioned series of works by Chen and Dalang  \cite{Che-Dal-14,Che-Dal-15}, which worked on the same class of initial measures $\ICM$ considered here. 

In addition to \cite{Hai-Lab-18}, there is now a rich literature of pathwise approaches to solving stochastic partial differential equations including and related to \eqref{eq:SHE} and \eqref{eq:KPZ} using regularity structures \cite{Hai-13,Hai-14}, paracontrolled distributions \cite{Gub-Per-17, Per-Ros-19}, and energy solutions \cite{Gon-Jar-14,Gor-Occ-21,Gub-Per-18}. See also the recent surveys \cite{Cor-She-20,Gub-18}.

The CDRP was originally introduced in \cite{Alb-Kha-Qua-14-jsp} with a construction that requires fixing initial and terminal conditions. Existence then holds on an event of full probability that depends  on the initial and terminal conditions. Our contribution in this work is to couple these solutions and then prove the regularity properties discussed above.

\subsection{Organization of the paper}  
Nonstandard notation is defined when first encountered, and all 
notational conventions and some topological and measure-theoretic preliminaries are collected in Appendix \ref{app:top}. In Section \ref{sec:setting}, we discuss our setting and state our main results. Section \ref{sec:chaos} then constructs the field of solutions to \eqref{eq:SHE}, proves growth estimates, and then uses these to prove basic properties about solutions to \eqref{eq:SHEf}. In Section \ref{sec:poly}, we use these results to build and prove regularity properties of the continuum directed polymer measures, which lead in particular to the strict total positivity in \eqref{eq:totpos}. We study more refined regularity properties of solutions to \eqref{eq:SHEf} and of polymers in Section \ref{sec:reg}. Appendix \ref{app:mild} is devoted to the equivalence up to indistinguishability between the superposition solution to \eqref{eq:SHEf} and the mild solutions which have previously been studied. Appendix \ref{app:KC} includes a statement of the version of the Kolmogorov-Chentsov theorem that we use in our construction. Finally, Appendix \ref{app:comp} collects most of the purely computational aspects of the present work.

\subsection{Acknowledgements}
The authors are grateful to Le Chen, Davar Khoshnevisan, and Samy Tindel for helpful conversations and to Davar Khoshnevisan, Samy Tindel, and four anonymous referees for their feedback on previous drafts of this manuscript.

\section{Setting and results}\label{sec:setting} 
\subsection{Setting} We assume that $(\Omega,\sF,\bbP)$ is a complete probability space that supports a space-time white noise $W$ on $L^2(\bbR^2)$ and a group of measure-preserving automorphisms of $\Omega$ which we describe momentarily. A white noise $W$ is a mean zero Gaussian process indexed by $f \in L^2(\bbR^2)$ which satisfies $\bbP(W(a f+b g) = a W(f) + b W(g))=1$ and $\mathbb{E}[W(f)W(g)] = \int_{\bbR^2}f(x_{1:2})g(x_{1:2})dx_{1:2}$ for $a,b \in \bbR$ and $f,g \in L^2(\bbR^2)$.  

The  shift maps $\shiftf{s}{y}$,   time and space reflection maps  $\reff_1$ and $\reff_2$,  the shear  $\shearf{s}{\nu}$ by $\nu$ relative to temporal level $s$,  rescaled dilation maps $\dif{\alpha}{\lambda}$, and the negation map $\nef$ act on  $f\in L^2(\bbR^2)$ as follows: 
\be\label{S7}\begin{aligned}
\shiftf{s}{y}  f (t,x) &= f(t+s,x+y) \quad \text{for} \quad  s,y \in \bbR; \\
\reff_1 f (t,x) &= f(-t,x) \quad\text{and}\quad  \reff_2 f  (t,x) = f(t,-x); \\
\shearf{s}{\nu} f (t,x) &= f(t, x+ \nu(t-s)) \quad \text{for} \quad  s,\nu \in \bbR; \\
\dif{\alpha}{\lambda} f(t,x) &= \sqrt{\alpha \lambda\tspa} \tspb f(\alpha t, \lambda x) \quad \text{for} \quad \alpha,\lambda > 0;\\
\nef f (t,x) &= - f(t,x). 
\end{aligned}\ee
Their    inverses are $\shiftf{t}{x}^{-1} = \shiftf{-t}{-x}, \reff_1^{-1} = \reff_1, \reff_2^{-1}=\reff_2,$ $\shearf{s}{\nu}^{-1} = \shearf{s}{-\nu}$, $\dif{\alpha}{\lambda}^{-1} = \dif{\alpha^{-1}}{\lambda^{-1}}$, and $\nef^{-1} = \nef$. 

We assume that $(\Omega,\sF,\bbP)$ comes equipped with a group (under composition) of measure-preserving automorphisms generated by $\refd_1$, $\refd_2$ (reflection), $\{\shiftd{s}{y} : s,y \in \bbR\}$ (translation), $\{\sheard{s}{\nu} : s,\nu \in \bbR\}$ (shear), $\{\did{\alpha}{\lambda} : \alpha,\lambda >0\}$ (dilation), and $\ned$ (negation), which act on $W$ by $W \circ \shiftd{s}{y}(f) = W(\shiftf{-s}{-y} f )$, $W \circ \sheard{s}{\nu}(f) = W(\shearf{s}{-\nu} f)$, $W \circ \refd_1(f) = W(\reff_1 f)$, $W \circ \refd_2(f) = W(\reff_2 f)$,  $W \circ \did{\alpha}{\lambda}(f) = W(\dif{\alpha^{-1}}{\lambda^{-1}} f)$, and $W \circ \ned = W(\nef f)$. The identity is given by $\Id = \shiftd{0}{0} = \refd_1 \circ \refd_1 = \refd_2 \circ \refd_2 = \sheard{0}{0} = \did{1}{1}$. An example of a Polish space satisfying these hypotheses in which all of these maps are continuous is described in detail in \cite[Appendix A]{Jan-Ras-Sep-22-1F1S-}.

Denote by $\nulls$ the $\sigma$-algebra generated by the $\bbP$-null sets in $\sF$. For $-\infty \leq a < b \leq \infty$, let $\timefun{a}{b}$ denote the $\sB([a,b]\times\R)$ measurable functions in $L^2(\R^2)$. 

Let $\filt{W,0}_{s,t} = \sigma(W(f) : f \in \timefun{s}{t})\vee \nulls$ be the $\sigma$-algebra generated by the white noise evaluated at such functions and $\nulls$. For each $s \leq t$, we define $\fil_{s,t} = \filt{W,0}_{s-,t+} = \bigcap_{a < s \leq t < b} \filt{W,0}_{a,b}$ to be the associated natural augmented filtration of the white noise. 
\subsection{Chaos expansions}
We take as given the existing results in the literature on existence of solutions to \eqref{eq:SHE} for fixed initial space-time points. We recap the results that we use in Appendix \ref{app:mild}. We will understand \eqref{eq:SHE} for fixed $s,y \in \bbR$ and $\beta \in \bbR$ through the mild equation
\begin{align}
\Sheb(t,x\viiva s,y) &= \heat(t-s,x-y) + \beta \int_s^t \int_{-\infty}^\infty \heat(t-u,x-z) \Sheb(u,z\viiva s,y)W(du dz), \label{eq:Greens}
\end{align}
where  $t>s$, $x \in \bbR$, and the stochastic integral is understood in the sense of Walsh \cite{Wal-86}. In  the case of $\beta = 0$, $\She_{0}(t,x\viiva s,y) = \heat(t-s,x-y)$ is the usual heat kernel. For $\beta \in \R$, existence of an event $\Omega_{s,y,\beta}$ on which there exists a process $\Sheb(\aabullet,\aabullet\viiva s,y) \in \sC((s,\infty)\times \bbR,\bbR)$ which is adapted to $(\fil_{s,t} : s \leq t)$  and solves \eqref{eq:Greens} is originally due to \cite{Ber-Can-95} and is included as part of Lemma \ref{lem:fixGreen} in Appendix \ref{app:mild} below.

Solving \eqref{eq:Greens} by Picard iteration leads to a chaos expansion of the solution. For $k\in\bbN$, we define $\heat_k: (\bbR \times \bbR)^{k} \to \bbR$ by
\[
\heat_k(t_{1:k},x_{1:k}\viiva s,y;t,x) = \prod_{i=0}^k \heat(t_{i+1}-t_i, x_{i+1}-x_i),
\]
 with the conventions that $t_0 = s, x_0 = y, t_{k+1}=t, x_{k+1}=x$. It is shown in \cite{Alb-Kha-Qua-14-aop} (see also \cite[Theorem 2.2]{Cor-18}) that for fixed $s,y \in \bbR$ and $\beta \in \R$ the unique continuous and adapted solution to \eqref{eq:Greens}, $\Sheb(t,x\viiva s,y)$, admits a chaos series representation as
\be\label{eq:SHEchaos}\begin{aligned}
\Sheb(t,x\viiva s,y) &= \heat(t-s,y-x) \\
&+ \sum_{k=1}^\infty \beta^k  \int_{\bbR^k} \int_{\bbR^k} \heat_{k}(t_{1:k}, x_{1:k}| s,y; t,x) W(dt_1 dx_1)\dotsm W(dt_k dx_k).
\end{aligned}\ee

The equality above is understood to hold almost surely and in $L^2(\Omega,\sF,\bbP)$ for each fixed quadruple $(s,y,t,x) \in \varsets= \{(s,y,t,x)\in \bbR^4 : s<t\}$ and $\beta \in \R$. We include these properties of $\Sheb(t,x\viiva s,y)$ in Lemma \ref{lem:fixGreen} below as well and refer the reader to \cite{Jan-97,Nua-06} for technical details concerning chaos expansions.

It will be convenient for us to normalize \eqref{eq:SHEchaos} by dividing through by the heat kernel. Define for $s < t$ 
\be\label{eq:rnchaos}\begin{aligned}
\rnSheb(t,x\viiva s,y) &= \frac{\Sheb(t,x\viiva s,y)}{\heat(t-s,x-y)} 
\\
&= 1 + \sum_{k=1}^{\infty} \beta^k \int_{\bbR^k} \int_{\bbR^k} \frac{\heat_{k}(t_{1:k}, x_{1:k} | s,y; t,x)}{\heat(t-s,x-y)} W(dt_1 dx_1)\dotsm W(dt_k dx_k),
\end{aligned}\ee
again in $L^2(\Omega,\sF,\bbP)$. We take the conventions that for all $\beta \in \R$, $\rnSheb(t,x|t,y)=1$ for all $t,x,y \in \bbR$ and $\rnShe_{0}(t,x\viiva s,y)=1$ for all $(s,y,t,x)\in\varset= \{(s,y,t,x) \in \bbR^4 : s \leq t\}$. The expression in \eqref{eq:rnchaos} is a rigorous version of the Feynman-Kac interpretation \eqref{eq:FK}.

\subsection{Solutions to the PAM and KPZ}
We begin with the observation that the process $\rnShe$ admits a modification $\tspb\tspb\widetilde{\!\rnShe}$  that is   $(\fil_{s,t} : s \leq t)$-adapted, with paths as functions of $(s,y,t,x,\beta)$ taking values in $\sC(\varset \times \bbR,\bbR)$. Then we define our solution of \eqref{eq:SHE} through $\widetilde{\She}_{\beta}(t,x\viiva s,y) = \tspb\widetilde{\!\rnShe}_{\beta}(t,x\viiva s,y)\heat(t-s,x-y)$. 
 In the theorem below,  $\bbD^d = \{(\frac{k_1}{2^{n_1}}, \dots, \frac{k_d}{2^{n_d} }) : k_1,\dots,k_d \in \bbZ  ,n_1,\dots,n_d \in\bbN\}$ is the set of dyadic rational numbers.

\begin{theorem}\label{thm:rnreg}
There exists an event $\Omega_0$ with $\bbP(\Omega_0) = 1$ and a $\sigma(W(f) : f \in L^2(\R^2))$-measurable random variable $\tspb\tspb\widetilde{\!\rnShe}_{\aabullet}(\aabullet,\aabullet|\aabullet,\aabullet)$ taking values in $\sC(\varset \times \bbR,\bbR)$ such that
\begin{enumerate} [label={\rm(\roman*)}, ref={\rm\roman*}]   \itemsep=3pt  
\item For all $\omega \in \Omega_0$, all $(s,y,t,x,\beta) \in \bbD^5$ with $s<t$,
\begin{align*}
\rnSheb(t,x\viiva s,y) = \tspb\widetilde{\!\rnShe}_{\beta}(t,x\viiva s,y) \qquad  \text{ and } \qquad \Sheb(t,x\viiva s,y) = \heat(t,x\viiva s,y)\,\widetilde{\!\rnShe}_{\beta}(t,x\viiva s,y),
\end{align*}
where $\rnSheb(t,x\viiva s,y)$ is given by \eqref{eq:rnchaos} and $\Sheb(t,x\viiva s,y)$ is given by \eqref{eq:SHEchaos}.
\item For all $\omega \in \Omega_0$, $\tspb\widetilde{\!\rnShe}_{\beta}(t,x|t,y)=1$ for all $t,x,y \in \bbR$ and $\beta \in \bbR$.
\item For all $\omega \in \Omega_0$, $\tspb\widetilde{\!\rnShe}_{0}(t,x\viiva s,y)=1$ for all $(s,y,t,x)\in\varset$. 
\item For all $\omega \in \Omega_0$ and all $(s,y,t,x) \in \varset$ and all $\beta \in \bbR$, $0<\tspb\widetilde{\!\rnShe}_{\beta}(t,x\viiva s,y)<\infty$. 
\item For all $s<t$, $\tspb\widetilde{\!\rnShe}_{\aabullet}(t,\aabullet|s,\aabullet)$ is $\fil_{s,t}$-measurable.
\item\label{rn:vi}  Fix the initial time-space point  $(s,y) \in \bbR^2$ and the inverse  temperature $\beta \neq 0$. Define  the process $\widetilde{\She}_{\beta}(t,x\viiva s,y)$  on $\{(t,x) \in \bbR^2 : t>s\}$ by $\widetilde{\She}_{\beta}(t,x\viiva s,y) = \tspb\widetilde{\!\rnShe}_{\beta}(t,x\viiva s,y)\heat(t-s,x-y)$. Then  
\begin{align*}
\bbP\bigg(\forall t\in(s,\infty) \ \text{\rm and }  x\in \bbR,\, \widetilde{\She}_{\beta}(t,x\viiva s,y) = \Sheb(t,x\viiva s,y) \bigg) = 1,
\end{align*}
where $\Sheb(t,x\viiva s,y)$ is the unique {\rm(}up to indistinguishability{\rm)}  continuous and adapted mild solution to \eqref{eq:Greens} satisfying the moment hypotheses of Lemma \ref{lem:fixGreen}.
\end{enumerate}
\end{theorem}
Except in the proofs leading up to the proof of Theorem \ref{thm:rnreg} in Section \ref{sec:chaos},
henceforth we work exclusively with these modifications and drop the tildes. That is, the process $\tspb\widetilde{\!\rnShe}_\beta$ given by the theorem is denoted simply by $\rnSheb$, and the process $\widetilde{\She}_{\beta}$ defined in part \eqref{rn:vi} of the theorem is denoted simply by   $\Sheb$.

 The processes $\Sheb$ and $\rnSheb$ inherit certain distributional invariance properties from the joint symmetries of the Brownian transition probabilities appearing in the multiple stochastic integrals in \eqref{eq:SHEchaos} and \eqref{eq:rnchaos} and the symmetries of the white noise. The following properties are all reasonably well-known. We include the details for completeness.
\begin{proposition}\label{prop:cov}
The processes $\rnSheb(t,x\viiva s,y)$ and $\Sheb(t,x\viiva s,y)$ satisfy the following properties:
\begin{enumerate}[label={\rm(\roman*)}, ref={\rm\roman*}]   \itemsep=3pt 
\item {\rm(Shift)} For each $u,z \in \bbR$, there is an event $\Omega_{\text{shift}(u,z)}$ with $\bbP(\Omega_{\text{shift}(u,z)}) =1$ so that on $\Omega_{\text{shift}(u,z)}$ for all $(s,y,t,x,\beta) \in \varset \times\bbR$,
\begin{align*}
\rnSheb(t+u,x+z|s+u,y+z)\circ \shiftd{-u}{-z} &= \rnSheb(t,x\viiva s,y)
\end{align*}
and, for all $(s,y,t,x,\beta) \in \varsets\times\bbR$,
\begin{align*}
\Sheb(t+u,x+z|s+u,y+z)\circ \shiftd{-u}{-z} &= \Sheb(t,x\viiva s,y).
\end{align*}
\item {\rm(Reflection)} There is an event $\Omega_{\refd}$ with $\bbP(\Omega_{\refd}) = 1$ so that on $\Omega_{\refd}$, for all $(s,y,t,x,\beta) \in \varset\times\bbR$,
\begin{align*}
\rnSheb(-s,y|-t,x) \circ \refd_1 = \rnSheb(t,x\viiva s,y) \text{ and }
\rnSheb(t,-x\viiva s,-y)\circ \refd_2 = \rnSheb(t,x\viiva s,y)
\end{align*}
and, for all $(s,y,t,x,\beta) \in \varsets\times\bbR$,
\begin{align*}
\Sheb(-s,y|-t,x) \circ \refd_1 = \Sheb(t,x\viiva s,y) \text{ and }
\Sheb(t,-x\viiva s,-y)\circ \refd_2 = \Sheb(t,x\viiva s,y).
\end{align*}
\item {\rm(Shear)} For each $(r,\nu) \in \bbR^2$, there exists an event $\Omega_{\text{\rm shear}(r,\nu)}$ with $\bbP(\Omega_{\text{\rm shear}(r,\nu)}) =1$ so that on $\Omega_{\text{\rm shear}(r,\nu)}$, 
 for all $(s,y,t,x,\beta) \in \varset\times\bbR$,
\begin{align*}
\rnSheb(t,x+\nu (t-r)|s, y + \nu(s-r))\circ \sheard{r}{-\nu} &= \rnSheb(t,x\viiva s, y),
\end{align*}
and, for all $(s,y,t,x,\beta) \in\varsets\times\bbR$,
\begin{align*}
e^{\nu(x-y) + \frac{\nu^2}{2}(t-s)}\Sheb(t,x+\nu (t-r)|s, y + \nu(s-r))\circ \sheard{r}{-\nu} &=  \Sheb(t,x\viiva s,y).
\end{align*}
\item {\rm(Scaling)} For each $\lambda >0$, there is an event $\Omega_{\text{\rm scale}(\lambda)}$ with $\bbP(\Omega_{\text{\rm scale}(\lambda)})=1$ so that on $\Omega_{\text{\rm scale}(\lambda)}$,  for all $(s,y,t,x,\beta) \in \varset\times\bbR$,
\begin{align*}
\rnShe_{\beta/\sqrt{\lambda}}(\lambda^2 t, \lambda x | \lambda^2 s, \lambda y) \circ \did{\lambda^{-2}}{\lambda^{-1}}&= \rnSheb(t,x\viiva s,y),
\end{align*}
and for all $(s,y,t,x,\beta) \in \varsets\times\bbR$,
\begin{align*}
\lambda \She_{\beta/\sqrt{\lambda}}(\lambda^2 t, \lambda x | \lambda^2 s, \lambda y) \circ \did{\lambda^{-2}}{\lambda^{-1}} &= \Sheb(t,x\viiva s,y).
\end{align*}
\item {\rm(Negation)} There is an event $\Omega_{\ned}$ with $\bbP(\Omega_{\ned}) =1$ so that on $\Omega_{\ned}$, for all $(s,y,t,x,\beta)\in \varset\times\bbR$,
\begin{align*}
\rnSheb(t,x\viiva s,y)\circ \ned &= \rnShe_{-\beta}(t,x\viiva s,y)
\end{align*}
and, for all $(s,y,t,x,\beta) \in \varsets\times\bbR$,
\begin{align*}
\Sheb(t,x\viiva s,y)\circ \ned &= \She_{-\beta}(t,x\viiva s,y).
\end{align*}

\end{enumerate}
\end{proposition}
We note the following corollary, the distributional part of which has appeared previously as \cite[Proposition 1.4]{Ami-Cor-Qua-11}, with the same argument.
\begin{corollary}\label{cor:stat}
For each $s <t$ and $z \in \bbR$, let $r=s$ and set $\nu = z/(t-s)$. Then, on $\Omega_{\text{\rm shear}(s,\nu)}$, for all $x,y,\beta \in \bbR$,
\begin{align*}
\rnSheb(t,x+z\viiva s,y) \circ \sheard{s}{-\nu} &= \rnSheb(t,x\viiva s,y) .
\end{align*}
Consequently, for each $s,y,t,\beta \in \bbR$ with $s<t$, the process $x \mapsto \rnSheb(t,x\viiva s,y)$ is stationary.
\end{corollary}

It will be convenient to extend the notation \eqref{eq:superpos} to allow for positive Borel measures $\zeta,\mu \in \sM_+(\bbR)$ as initial or terminal conditions by setting, for $s<t$ and $\beta \in \R$,
\begin{align}
\Sheb(t;\zeta|s;\mu) &= \int_{\bbR} \int_{\bbR} \Sheb(t,x|s,y) \mu(dy)\zeta(dx) \in [0,\infty] . \label{eq:Shem2m}
\end{align}
The special cases   $\Sheb(t,x\viiva s;\mu)=\Sheb(t;\delta_x | s;\mu)$ and $\Sheb(t;\zeta| s,y)$ $=\Sheb(t;\zeta | s;\delta_y)$ return   \eqref{eq:superpos}. Some of our results will rely on finiteness of certain joint moments of these measures. For $p\in[0,\infty)$, recalling \eqref{eq:ICM} we define
\begin{align}
\ICMM(p) &= \bigg\{(\mu,\zeta) \in \ICM^2 : \forall a>0, \int_{\bbR} \int_{\bbR} e^{-a(w-z)^2}(1+|w|^p+|z|^p)\mu(dw)\zeta(dz) <\infty \bigg\} \label{eq:ICMM}
\end{align}
and note that $\ICMM(q) \subset \ICMM(p)$ for $0 \leq p<q < \infty$. These spaces mostly serve as bookkeeping tools to simplify the statements of our results. We usually work with the case where at least one of the two measures is Dirac, which always results in a pair of measures in $\ICMM(p)$ for all $p$, as recorded in  the following  remark: 
\begin{remark}\label{rem:ICMMdelta}
For all $x,y \in \bbR$ and all $\mu,\zeta \in \ICM$, and all $p \in [0,\infty)$, $(\zeta,\delta_y),(\delta_x,\mu) \in \ICMM(p)$. 
\end{remark}
Our next result discusses some basic properties of $\She_{\aabullet}(\aabullet,\aabullet|\aabullet,\aabullet)$ viewed as the solution semi-group of \eqref{eq:SHEf}, including finiteness, sharpness of the restriction to $\ICM$, and regularity of the processes in \eqref{eq:Shem2m}.

Before stating the result, we collect some notation which is also recorded in Appendix \ref{app:top}. In the statement of part  \eqref{prop:IC:Holder}, the local H\"older semi-norm
\be\label{H-56}
|f|_{\sC^{\alpha,\gamma,\eta}(\Gamma)}
\,=\!\!\!\!\sup_{\substack{(t_1,x_1,s_1,\beta_1) \\[1pt]  
\ne (t_2,x_2,s_2,\beta_2) \,\in\,   \Gamma}}  
\frac{|f(t_1,x_1,s_1,\beta_1) \,-\, f(t_2,x_2,s_2,\beta_2)|}{|t_2-t_1|^{\alpha}+ |s_2-s_1|^{\alpha} + |x_2-x_1|^{\gamma} +|\beta_2-\beta_1|^\eta}
\ee
 is defined on the space 
\begin{align*}
\Gamma&=\varsetthg{T}{K}{\delta}\times[-B,B] \\
&= \{(s,t,x,\beta) \in \R^4:  -T \leq s,t \leq T, \,-K \leq x \leq K,\, t-s\geq\delta, \,-B\le\beta\le B\}.
\end{align*}  
Recall also the spaces $\ICM$, $\CICM$ from \eqref{eq:ICM} and \eqref{eq:CICM}. In Appendix \ref{app:top}, we define Polish topologies on $\ICM$ and $\CICM$, with  explicit complete  metrics $d_{\ICM}$ and $d_{\CICM}$ in equations \eqref{eq:ICMm} and \eqref{eq:CICMm}. Convergence in $\ICM$ is characterized by vague convergence combined with convergence of integrals of the form $\int e^{-ax^2}\mu(dx)$ for each $a>0$. In $\CICM$, the convergence is uniform convergence of $f$ and $1/f$ on compact sets combined with convergence of the integrals $\int_{\R} e^{-ax^2}f(x)dx$ for $a>0$.

\begin{theorem}\label{prop:IC}
There is an event $\Omega_0$ with $\bbP(\Omega_0) = 1$ so that for all $\omega \in \Omega_0$, the following hold:
\begin{enumerate}[label={\rm(\roman*)}, ref={\rm\roman*}]   \itemsep=3pt 
\item \label{prop:IC:1} \textup{(Finiteness and positivity)}  For all $(\mu,\zeta) \in \ICMM(4)$ and all $s,t,\beta \in \bbR$ with $s<t$, \[0<\Sheb(t;\zeta|s;\mu) < \infty.\]
\item\label{prop:IC:ICMpres} \textup{(Preservation of $\ICM$)} If $\mu,\zeta \in \ICM$ then for all $s<t$ and all $\beta \in \bbR$, \[\Sheb(t,x\viiva s;\mu)dx,\quad \Sheb(t;\zeta|s;y)dy \in \ICM.\]
\item \label{prop:IC:ICMsharp} \textup{(Explosion off $\ICM$)} If $\mu,\zeta \in \sM_+(\bbR)$ are both not the zero measure, 
then
\begin{align*}
&\frac{1}{2(t-s)} < \sup\bigg\{a > 0 : \int e^{-ay^2}\mu(dy) = \infty \text{ or } \int e^{-ax^2}\zeta(dx) = \infty\bigg\}
\end{align*}
implies that $\Sheb(t;\zeta|s;\mu)  = \infty$ for all $\beta \in \bbR$.
\item \label{prop:IC:Holder} \textup{(Local ($1/4-$, $1/2-$, $1-$) H\"older continuity)} For all $\mu,\zeta \in \ICM$, all $\delta>0$, all $B,K,T>0$, and all $\alpha \in (0,1/4)$, $\gamma \in (0,1/2)$, and $\eta \in (0,1)$, $|\She_{\aabullet}(\aabullet,\aabullet|\aabullet;\mu)|_{\sC^{\alpha,\gamma,\eta}(\Gamma)} < \infty$ and $|\She_{\aabullet}(\aabullet;\zeta|\aabullet,\aabullet)|_{\sC^{\alpha,\gamma,\eta}(\Gamma)} < \infty$. 
\item \label{prop:IC:semi-group}\textup{(Semi-group property)} For all $\mu \in \sM_+(\R)$, all $(s,y,t,x,\beta)\in \varsets\times\R$, and all $r \in (s,t)$, \begin{align*}
\Sheb(t,x\viiva s;\mu) &= \int_{\bbR} \Sheb(t,x|r,z)\Sheb(r,z\viiva s;\mu)dz  \text{ and }\\
 \Sheb(t;\mu \viiva s,y) &= \int_{\bbR} \Sheb(t;\mu|r,z)\Sheb(r,z\viiva s,y)dz.
\end{align*}
\item \label{prop:IC:fcont} \textup{(Function-valued initial conditions)}
For  all $B,K,T>0$ and all $f \in \CICM$,
\begin{align*}
&\lim_{\delta \searrow0}\, \sup_{\substack{y\in[-K,K],\beta\in[-B,B] \\ s,t\in[-T,T], t-s\in(0,\delta ) \\ d_{\CICM}(f,g)<\delta}}\,\bigg| \int_{\bbR} \Sheb(t,x\viiva s,y)g(x)dx - f(y)\bigg|=0  \,\text{ and }\, \\ 
&\lim_{\delta \searrow0}\, \sup_{\substack{x\in[-K,K],\beta\in[-B,B] \\ s,t\in[-T,T], t-s\in(0,\delta )\\ d_{\CICM}(f,g)<\delta}}\,\bigg| \int_{\R}  \Sheb(t,x|s,y)g(y)dy - f(x)\bigg|=0.
\end{align*}
\item \label{prop:IC:musat} \textup{(Measure-valued initial conditions)}
For all $T,B>0$, all $\mu \in \ICM$, and all $f \in \sC(\bbR,\R_+)$ for which there exist $A,a>0$ such that $0 \leq f(x) \leq Ae^{-ax^2}$, 
\begin{align*}
\lim_{\delta \searrow 0} \sup_{\substack{\beta\in[-B,B], s,t\in[-T,T] \\d_{\ICM}(\mu,\zeta)<\delta, t-s\in(0,\delta )}} \,\bigg|\int_{\bbR}\int_{\bbR} f(x) \Sheb(t,x\viiva s,y)\tspa dx \tspa\zeta(dy) -\int_{\bbR} f(y) \tspa\mu(dy)\biggr| &= 0  \text{ and }\\
\lim_{\delta \searrow 0} \sup_{\substack{\beta\in[-B,B], s,t\in[-T,T] \\ d_{\ICM}(\mu,\zeta)<\delta, t-s\in(0,\delta )}} \,\bigg|\int_{\bbR}\int_{\bbR} f(y) \Sheb(t,x\viiva s,y)\tspa dy \tspa\zeta(dx) -\int_{\bbR} f(x) \tspa\mu(dx)\biggr| &= 0.
\end{align*}
\end{enumerate} 
\end{theorem}

\begin{remark}
The case of $\zeta(dz)=\delta_x(dz)$ in part \eqref{prop:IC:1} and the claim in part \eqref{prop:IC:Holder} were shown for fixed $\mu \in \ICM$  and fixed $s, \beta \in \bbR$ on a full probability event depending on $\mu,s,$ and $\beta$ in Theorem 3.1 of \cite{Che-Dal-14}. Strict positivity for certain fixed function-valued initial conditions and fixed initial times $s$ is originally due to Mueller \cite{Mue-91}, though the proof generalizes to other initial conditions. We rely on the later estimates of Moreno-Flores in \cite{Mor-14}. Part \eqref{prop:IC:musat} improves on Proposition 3.4 in \cite{Che-Dal-14}, where the limit holds pointwise in $L^2(\Omega,\sF,\bbP)$ for fixed $\varphi, \mu=\zeta,\beta$, and $s$ and $\varphi$ is taken to be compactly supported. 
\end{remark}

\begin{remark}
A boundary continuity result similar to \eqref{prop:IC:fcont} appears in Theorem 3.1 of \cite{Che-Dal-14}, which includes a H\"older regularity estimate at the boundary for H\"older continuous initial data, again for fixed $\mu,s$.  Our methods can prove similar boundary regularity, but with suboptimal H\"older exponents. We leave this improvement to future work. The only result which needs to be improved to obtain optimal regularity at the boundary is Lemma \ref{lem:withgap}, where the bound needs to not depend on $\delta$ in order to obtain optimal regularity.
\end{remark}

Theorem \ref{prop:IC}\eqref{prop:IC:semi-group} says that \eqref{eq:superpos} defines a solution semi-group to \eqref{eq:SHEf}. We next turn to the regularity of this semi-group on natural spaces of measures and functions given by $\ICM$ and $\CICM$. We take the following notational conventions. If $\mu \in \ICM$ and $f \in \CICM$, then for all $\beta \in \R$ and $s\leq t$, we set
\be\label{eq:mbdydef}\begin{aligned}
\Sheb(t,dx\viiva s;\mu) &= \begin{cases}
\Sheb(t,x\viiva s;\mu)dx & s<t \\ 
\mu(dx) & s=t
\end{cases} \qquad \text{ and }\\
\Sheb(t;\mu \viiva s, dy) &= \begin{cases}
\Sheb(t;\mu \viiva s, y)dy & s<t \\
\mu(dy) & s=t.
\end{cases}
\end{aligned}\ee
Similarly, if $f \in \CICM$, then for all $\beta,t\in\R$, we set
\begin{align*}
\Sheb(t, \aabullet \viiva t;f) &= f(\aabullet) \qquad \text{ and } \qquad \Sheb(t;f \viiva t;\aabullet) = f(\aabullet).
\end{align*}
With these conventions, we have the following regularity result.
\begin{theorem}\label{thm:jcont}
There exists an event $\Omega_0$ with $\bbP(\Omega_0)=1$ on which the following hold.
\begin{enumerate}[label={\rm(\roman*)}, ref={\rm\roman*}]   \itemsep=3pt 
\item\label{thm:jcont:MM}  The following maps from $\R\times \ICM \times \{(s,t) \in \R^2 : s \leq t\}$ to $\ICM$ are continuous:
\[
(\beta, \mu,s,t) \mapsto \Sheb(t,dx \viiva s;\mu) \qquad \text{ and }\qquad (\beta, \mu,s,t) \mapsto \Sheb(t;\mu \viiva s, dy).
\]
\item \label{thm:jcont:MC}  The following maps from $\R\times \ICM \times \{(s,t) \in \R^2 : s < t\}$ to $\CICM$ are continuous:
\[
(\beta, \mu,s,t) \mapsto \Sheb(t,\aabullet \viiva s;\mu) \qquad \text{ and }\qquad (\beta, \mu,s,t) \mapsto \Sheb(t;\mu \viiva s, \aabullet).
\]

\item\label{thm:jcont:CC} The following maps from $\R\times \CICM \times \{(s,t) \in \R^2 : s \leq t\}$ to $\CICM$ are continuous:
\[
(\beta, f ,s,t) \mapsto \Sheb(t,\aabullet \viiva s;f) \qquad \text{ and }\qquad (\beta, f ,s,t) \mapsto \Sheb(t;f \viiva s, \aabullet).
\]
\end{enumerate}
\end{theorem}
Recall the space $\CKPZ$ from \eqref{eq:CKPZ}. Note that the bijection $g(x) = e^{f(x)}$ between $g \in \CICM$ and $f \in \CKPZ$ defines a Polish topology on $\CKPZ$. Convergence in this topology is local uniform convergence of $f$ combined with convergence of integrals of the form $\int_{\R}e^{f(x) - ax^2}dx$ for $a>0$. We record the following immediate corollary of Theorem \ref{thm:jcont}\eqref{thm:jcont:CC} for the KPZ equation
\begin{corollary}\label{cor:KPZ}
There exists an event $\Omega_0$ with $\bbP(\Omega_0)=1$ on which the map  from $\R\times \CKPZ \times \{(s,t) \in \R^2 : s \leq t\}$ to $\CKPZ$ given by
\[
(\beta,f,s,t) \mapsto \KPZb(t,\aabullet \viiva s;f),
\]
where $\KPZb$ is defined through \eqref{eq:KPZdef} if $s<t$ and $\KPZb(s, \aabullet \viiva s;f) = f(\aabullet)$ is continuous.
\end{corollary}
\begin{remark}\textup{(Uniqueness)} Because $\R\times \CKPZ \times \{(s,t) \in \R^2 : s \leq t\}$ is separable and Lemma \ref{lem:uniq} shows that for fixed $s,\beta \in \R$ and $f \in \CKPZ$, \eqref{eq:KPZdef} agrees with the Hopf-Cole mild formulation, Corollary \ref{cor:KPZ} shows that \eqref{eq:KPZdef} describes the unique continuous modification of the field of Hopf-Cole solutions of the KPZ equation, up to indistinguishability.
\end{remark}
\begin{remark}
\textup{(Feller continuity of KPZ on $\CKPZ$)} An immediate consequence of Corollary \ref{cor:KPZ} is that if $F:\CKPZ \to \R$ is bounded and continuous and if $f \to g$ in the topology on $\CKPZ$ and $s<t$, then $\bbE[F(\KPZb(t,\aabullet\viiva s;f))] \to \bbE[F(\KPZb(t,\aabullet\viiva s;g))]$. This is Feller continuity of the solution semi-group to the KPZ equation, viewed as a $\CKPZ$-valued Markov process.
\end{remark}

We next show that the solution $h=\log\Sheb$  to the KPZ equation   \eqref{eq:KPZ} satisfies the same conservation law (of asymptotic slope) as is preserved by the unforced viscous Burgers equation. To that end, for $\lambda_-, \lambda_+ \in \bbR$, define 
\be\label{eq:HOA}\begin{aligned}
H(\lambda_-,\lambda_+) &= \bigg\{f:\bbR\to\R \text{ Borel measurable, locally bounded, }\lim_{x\to \pm\infty} \frac{f(x)}{x} = \lambda_\pm  \bigg\}. 
\end{aligned}\ee
\begin{proposition}\label{prop:cons}
There is an event $\Omega_0$ with $\bbP(\Omega_0)=1$ so that on $\Omega_0$, the following holds: for all $\lambda_+,\lambda_-,\beta \in \bbR$, all $f \in H(\lambda_-,\lambda_+)$, and all $s<t$, $\KPZb(t, \aabullet \viiva s;f), \KPZb(t;f \viiva s, \aabullet) \in H(\lambda_-,\lambda_+)$, where $\KPZb$ is defined through \eqref{eq:KPZdef}.
\end{proposition}

\subsection{Quenched continuum directed polymers}
Next, we turn to the structure of polymer measures. We first show that there is an event of full probability on which the quenched point-to-point measures defined in \eqref{eq:p2pdef} all exist, are supported on H\"older $(1/2)-$ paths, are Feller, and satisfy basic continuity and measurability properties.

In the statement of the following result, $\oE_{(s,y),(t,x)}^{\Polyb}$ is the expectation under  $\Polyb_{(s,y),(t,x)}$. The path space  $\sC_{[s,t]}=\sC([s,t],\R)$ is endowed with its uniform topology and Borel $\sigma$-algebra $\sB(\sC_{[s,t]})$,   and $|f|_{\sC^{\eta}_{[s,t]}}$ is the standard $\eta$-H\"older seminorm, which is defined in \eqref{app-H7} in Appendix \ref{app:top}.   $X=(X_u)_{u\in[s,t]}$ is the path variable on $\sC_{[s,t]}$ and $\sG_{u,v}=\sigma(X_r: r\in[u,v])$ is the natural filtration.  $\sB_b(\sC_{[s,t]})$ is the space of bounded Borel functions and $\sM_1(\sC_{[s,t]})$ the space of probability measures on $\sC_{[s,t]}$, equipped with its standard topology of weak convergence generated by $\sC_b(\sC_{[s,t]},\R)$ test functions.

\begin{theorem}\label{prop:polyexist}
There exists an event $\Omega_0$ with $\bbP(\Omega_0) = 1$ so that for all $\omega \in \Omega_0$, the following holds: 
\begin{enumerate}[label={\rm(\roman*)}, ref={\rm\roman*}]   \itemsep=3pt 
\item \textup{(Existence and uniqueness)} For each $(s,y,t,x,\beta) \in \varsets\times\bbR$, there exists a unique probability measure $\Polyb_{(s,y),(t,x)}$ on $(\sC_{[s,t]},\sB(\sC_{[s,t]}))$ with finite-dimensional marginals \eqref{eq:p2pdef}.
\item \textup{(H\"older $1/2-$ path regularity)} For each $(s,y,t,x,\beta)\in \varsets \times \bbR$ and $\eta \in (0,1/2)$,
\begin{align}
\Polyb_{(s,y),(t,x)}(|X|_{\sC^{\eta}_{[s,t]}} < \infty)=1 \label{eq:p2phold}
\end{align} 
and
\begin{align}
\Polyb_{(s,y),(t,x)}(X_s = y) = \Polyb_{(s,y),(t,x)}(X_t=x) = 1. \label{eq:p2pIC}
\end{align}
\item\label{prop:polyexist:Markov} \textup{(Markov property)} For each $(s,y,t,x,\beta)\in \varsets\times\bbR$, each $u,v$ satisfying $s < u < v < t$, and each $F \in \sB_b(\sC_{[u,v]})$
\begin{align*}
\oE^{\Polyb}_{(s,y),(t,x)}[F(X|_{[u,v]})|\sG_{s,u},\sG_{v,t}] &= \oE^{\Polyb}_{(u,X_u),(v,X_v)}[F(X)] \qquad \Polyb_{(s,y),(t,x)}\text{-a.s.}
\end{align*}
\item \label{prop:polyexist:cmeas} \textup{(Continuity)} For all $s<t$, the map $(x,y,\beta) \mapsto \Polyb_{(s,y),(t,x)}$ from $\R^3$ into $\sM_1(\sC_{[s,t]})$ is continuous. 
\end{enumerate}
\end{theorem}

We next use the point-to-point polymers to construct measure-to-measure polymers for  $(\mu,\zeta) \in \ICMM(4)$. By Theorem \ref{prop:IC}\eqref{prop:IC:1}, $0 < \Sheb(t;\zeta|s;\mu) <\infty$ for all such $(\mu,\zeta)$, $s<t$ and $\beta \in \bbR$.  Measure-to-measure polymer distributions on the path space  $\sC_{[s,t]}$ are defined  for $A \in \sB(\sC_{[s,t]})$ by 
\begin{align}
\Polyb_{(s;\mu),(t;\zeta)}(A) &= \frac1{\Sheb(t;\zeta|s;\mu)}{ \int_{\bbR}\int_{\bbR}\Sheb(t,x\viiva s,y)\tspc\Polyb_{(s,y),(t,x)}(A) \tspc\zeta(dx)\mu(dy) } .  
\label{eq:m2mdef}
\end{align}
Note that $\Polyb_{(s,y),(t,x)} = \Polyb_{(s;\delta_y),(t;\delta_x)}$  with this definition. We have the following basic properties of the measure-to-measure polymers, encompassing their existence, H\"older support, Markovian structure, and regularity properties. In part \eqref{prop:m2mboch:TV} below, if $\mu-\zeta$ is not a well-defined signed measure on $\R$, the total variation measure $|\mu-\zeta|$ is defined as a limit of the (well-defined) total variation measures restricted to compact sets. See Appendix \ref{app:top}. 
 
\begin{theorem}\label{prop:m2mbasic}
There is an event $\Omega_0$ with $\bbP(\Omega_0)=1$ so that on $\Omega_0$, the following hold:
\begin{enumerate}[label={\rm(\roman*)}, ref={\rm\roman*}]   \itemsep=3pt 
\item \label{prop:m2mbasic:def} \textup{(Existence and density)} For all $(s,y,t,x,\beta) \in \varsets\times\bbR$ and all $(\mu,\zeta) \in \ICMM(4)$, \eqref{eq:m2mdef} defines a probability measure on $(\sC_{[s,t]},\sB(\sC_{[s,t]}))$. For all  $t_{1:k}$ satisfying $s =t_0 < t_1 < \dots < t_k <  t_{k+1}=t$, the finite dimensional distributions of this measure are given by
\be\begin{aligned} \label{eq:m2mpoly}
&\Polyb_{(s;\mu),(t;\zeta)}(X_s \in dx_0, X_{t_1} \in dx_1, \dots X_{t_k} \in dx_k, X_t \in dx_{k+1})\\
&\qquad \qquad\qquad = \; \mu(dx_0) \zeta(dx_{k+1}) \frac{\prod_{i=0}^{k}\Sheb(t_{i+1},x_{i+1} | t_i , x_i)}{\Sheb(t;\zeta|s;\mu)}dx_{1:k}.
\end{aligned}\ee
\item \label{prop:m2mboch:cont} \textup{(H\"older $1/2-$ path regularity)} For each $s<t$, each $\beta \in \bbR$, each $(\mu,\zeta) \in \ICMM(4)$ and each $\eta \in (0,1/2)$,
\begin{align}
\Polyb_{(s;\mu),(t;\zeta)}(|X|_{\sC^{\eta}_{[s,t]}} < \infty)=1.
\end{align}
\item\label{prop:m2mboch:IC} \textup{(Initial condition)} For all $\beta \in \bbR$, all $s<t$, all $\mu,\zeta \in \ICM$ and all $x,y \in \bbR$,
\begin{align*}
\Polyb_{(s,y),(t;\zeta)}(X_s = y) = \Polyb_{(s;\mu),(t;x)}(X_t=x) = 1.
\end{align*}
\item\label{prop:m2mbasic:Markov} \textup{(Markov property)}  For all $s<t$, all $\beta \in \bbR$, all $u,v \in (s,t)$ satisfying $s < u < v < t$, all $(\mu,\zeta)\in \ICMM(4)$, the following holds: for each  $F \in \sB_b(\sC([u,v],\R))$,
\begin{align*}
\oE^{\Polyb}_{(s;\mu),(t;\zeta)}[F(X|_{[u,v]})|\sG_{s,u},\sG_{v,t}] &= \oE^{\Polyb}_{(u,X_u),(v,X_v)}[F(X)] \qquad \Polyb_{(s;\mu),(t;\zeta)}\text{-a.s.}
\end{align*}
\item\label{prop:m2mboch:TV} \textup{(Total variation norm comparison)} For all $(\mu_1,\zeta_1),(\mu_2,\zeta_2)\in \ICMM(4)$, all $s<t$, and all $\beta \in \R$, we have \begin{align*}
&\frac{1}{2}\|\Polyb_{(s;\mu_1),(t;\zeta_1)}-\Polyb_{(s;\mu_2),(t;\zeta_2)}\|_{TV}\\
&\leq \frac{\Sheb(t;|\zeta_2-\zeta_1||s;\mu_1)}{\Sheb(t;\zeta_1|s;\mu_1)} + \Sheb(t;\zeta_2|s;\mu_1)\tspb \bigl|\Sheb(t;\zeta_1|s;\mu_1)^{-1}-\Sheb(t;\zeta_2|s;\mu_1)^{-1}\bigr| \\
&\qquad +\frac{\Sheb(t;\zeta_2|s;|\mu_2-\mu_1|)}{\Sheb(t;\zeta_2|s;\mu_1)} + \Sheb(t;\zeta_2|s;\mu_2)\tspb \bigl|\Sheb(t;\zeta_2|s;\mu_1)^{-1}-\Sheb(t;\zeta_2|s;\mu_2)^{-1}\bigr|.
\end{align*}
\end{enumerate}
\end{theorem}

We turn to continuity properties of polymer measures, including a version of the strong Feller property, with the caveat that in order to connect back to the usual formulation of a time inhomogeneous Markov process, one needs to view this as a chain which is  killed before time $s$ if run backward in time starting from $t$ and after time $t$ if run forward in time starting from $s$. 
\begin{theorem}\label{thm:polyreg}
There is an event $\Omega_0$ with $\bbP(\Omega_0)=1$ on which the following hold. 
\begin{enumerate}[label={\rm(\roman*)}, ref={\rm\roman*}]   \itemsep=3pt 
\item \label{thm:polyreg:weak} \textup{(Weak continuity)} For all $s<t$ in $\R^2$, the maps from $\R^2\times \ICM$ to $\sM_1(\sC_{[s,t]})$
\[(\beta,x,\mu) \mapsto \Polyb_{(s;\mu),(t,x)} \qquad \text{ and }  \qquad (\beta,y,\zeta) \mapsto \Polyb_{(s;y),(t;\zeta)}\]
are continuous with the weak topology on $\sM_1(\sC_{[s,t]})$.
\item \label{thm:polyreg:TV} \textup{(Total variation continuity)} If $\lim_{n\to\infty}d_{\ICM}(\mu_n,\mu) = 0$ and the total variation measure $|\mu_n-\mu|$ converges to the zero measure vaguely, then for all $(s,y,t,x,\beta)\in\varsets\times\R$,
\[
\lim_{n\to\infty}\|\Polyb_{(s;\mu_n),(t,x)} - \Polyb_{(s;\mu),(t,x)}\|_{TV} = 0 =\lim_{n\to\infty}\|\Polyb_{(s,y),(t;\mu_n)} - \Polyb_{(s,y),(t;\mu)}\|_{TV}.
\]
\item \label{thm:polyreg:SF} \textup{(Strong Feller property)} For all $\mu,\zeta \in \ICM$, all $f \in \sB_b(\bbR)$, and all $r \in \R$, the maps
\begin{align*}
(\beta,s,t,y) \mapsto \oE_{(s,y),(t;\zeta)}^{\Polyb}[f(X_r)] \qquad \text{ and } \qquad (\beta,s,t,x) \mapsto \oE_{(s;\mu),(t,x)}^{\Polyb}[f(X_r)]
\end{align*}
are continuous on $\{(\beta,s,t,x)\in \bbR^4 : s < r < t\}$. 
\item \label{thm:polyreg:vbr} \textup{(Vague boundary regularity)} For all $\mu,\zeta \in \ICM$,  $f \in \sC_c(\R,\R)$, and $r \in \R$, the maps
\begin{align*}
(\beta,s,y) \mapsto \oE^{\Polyb}_{(s,y),(t;\zeta)}[f(X_r)] \qquad \text{ and }\qquad (\beta,t,x) \mapsto \oE^{\Polyb}_{(s;\mu),(t,x)}[f(X_r)]
\end{align*} 
are continuous on $\{(\beta,s,y) \in \bbR^3 : s \leq r\}$ and $\{(\beta,t,x) \in \bbR^3 : t \geq r\}$, respectively.
\end{enumerate}
\end{theorem}

The path spaces come with a natural partial order: $f\le g$ means that $f(u)\le g(u)$ for all $u$ in the domain of these functions.   
Because the paths of the polymer measures are continuous and we are in a planar setting, it is natural to expect that the polymer measures are stochastically ordered. One way to rigorize this intuition is through the Karlin-McGregor identity \cite{Kar-McG-59}, which we show implies a much stronger strict total positivity condition, recorded as Theorem \ref{thm:strpos} below.   The Weyl chamber $\bbW_n$ was defined above \eqref{eq:totpos}.
\begin{theorem}\label{thm:strpos}
There is an event $\Omega_0$ with $\bbP(\Omega_0)=1$ so that on $\Omega_0$, the following holds. For all $n\in\N$, $s<t$, all $\beta \in \bbR$, and all $(x_1,\dots,x_n)$, $(y_1,\dots,y_n)\in\bbW_n$,
\begin{align*}
\det\big[\Sheb(t,x_j\viiva s,y_i) \big]_{i,j=1}^n > 0.
\end{align*}
\end{theorem}
The statement of the Karlin-McGregor identity for the polymer measures is Proposition \ref{prop:KMG} below. The resulting stochastic monotonicity is the content of our next result.
\begin{proposition}\label{prop:stochmon}
There is an event $\Omega_0$ with $\bbP(\Omega_0)=1$ so that on $\Omega_0$ the following holds: For all $s<t$, all $\beta \in \bbR$, all $x_1<x_2$, all $y_1<y_2$, and all $\mu,\zeta\in\ICM$,
\begin{align}
\Polyb_{(s,y_1),(t;\zeta)} \std \Polyb_{(s,y_2),(t;\zeta)} \qquad \text{ and }\qquad \Polyb_{(s;\mu),(t,x_1)} \std \Polyb_{(s;\mu),(t,x_2)}. \label{eq:stochmon:1}
\end{align}
\end{proposition}
In the statement of the previous result, $\std$ denotes stochastic dominance. See Appendix \ref{app:top} for a precise definition. 

With our main results stated, we turn to the proofs.

\section{Continuity, invariance, growth, and the conservation law}\label{sec:chaos}
In this section, we prove most of our results about the structure of our solutions to \eqref{eq:SHE} and \eqref{eq:SHEf}, beginning with the proofs of  Theorem \ref{thm:rnreg} and Proposition \ref{prop:cov}. We begin this section with a brief outline. We take as given the previous results on the existence of mild solutions to \eqref{eq:SHE} and their chaos series representations. We quickly recap these results and the relevant references in Appendix \ref{app:mild}. The important points for now are that for a fixed $(s,y)$ and $\beta\in\R$, a unique continuous and adapted solution to \eqref{eq:Greens} exists and for fixed $s,y,t,x,\beta$, this process admits a representation as the chaos series \eqref{eq:SHEchaos}. We use Kolmogorov-Chentsov to glue these together and then verify that this process satisfies our assumptions. We include a version of Kolmogorov-Chentsov satisfying our needs as Theorem \ref{thm:KC}  in Appendix \ref{app:KC} below. The purely computational parts of the argument are deferred to Appendix \ref{app:comp}.

We then turn to constructing the modification $\tspb\widetilde{\!\rnShe}_{\beta}(t,x\viiva s,y)$ in Theorem \ref{thm:rnreg}, which we obtain by gluing together the processes $\rnSheb(t,x\viiva s,y)$ defined through the chaos series \eqref{eq:rnchaos} at dyadic rational space-time points. Next, we verify that this process is consistent, i.e.~it defines a version of the processes we started with off  the dyadic rationals. This is essentially immediate from our previous  two point estimates. Our Kolmogorov-Chentsov estimates imply growth bounds, which then allow us to prove most of the remaining results in the paper. 

The first lemma is the  restricted version of Proposition \ref{prop:cov} for the chaos series in \eqref{eq:rnchaos}.  Recall the definitions of these transformations at and below  \eqref{S7}. 
\begin{lemma}\label{lem:cov}
The processes $\rnSheb(t,x\viiva s,y)$ and $\Sheb(t,x\viiva s,y)$ satisfy the following:
\begin{enumerate}[label={\rm(\roman*)}, ref={\rm\roman*}]   \itemsep=3pt 
\item \label{lem:cov:shift} {\rm(Shift)} For each $u,z \in \bbR$ and $(s,y,t,x,\beta) \in \varsets\times\bbR$, 
there exists an event $\Omega_1$$=\Omega_1(u,z,s,y,t,x,\beta)$ with $\bbP(\Omega_{1}) =1$ so that on $\Omega_1$,
\begin{align*}
\rnSheb(t+u,x+z|s+u,y+z)\circ \shiftd{-u}{-z} &= \rnSheb(t,x\viiva s,y) \qquad \text{ and} \\
\Sheb(t+u,x+z|s+u,y+z)\circ \shiftd{-u}{-z} &= \Sheb(t,x\viiva s,y).
\end{align*}
\item \label{lem:cov:ref} {\rm(Reflection)} For each  $(s,y,t,x,\beta) \in \varsets\times\bbR$ 
there is an event $\Omega_{1}=\Omega_1(s,y,t,x,\beta)$ with $\bbP(\Omega_1) = 1$ so that on $\Omega_{1}$,
\begin{align*}
\rnSheb(-s,y|-t,x)\circ \refd_1 &= \rnSheb(t,x\viiva s,y), \qquad
\Sheb(-s,y|-t,x)\circ \refd_1 = \Sheb(t,x\viiva s,y), \\
\rnSheb(t,-x\viiva s,-y)\circ \refd_2 &= \rnSheb(t,x\viiva s,y), \text{ and }
\Sheb(t,-x\viiva s,-y)\circ \refd_2 = \Sheb(t,x\viiva s,y).
\end{align*}
\item \label{lem:cov:shear} {\rm(Shear)}  For  each $(s,y,t,x,\beta)\in \varsets\times\bbR$ and each $r,\nu \in \bbR$, there exists an event $\Omega_{1}=$ $\Omega_1$$(s,y,t,x,\beta,r,\nu)$  with $\bbP(\Omega_{1})=1$ so that on $\Omega_1$,
\begin{align*}
\rnSheb(t,x+\nu (t-r)|s, y + \nu(s-r))\circ \sheard{r}{-\nu} &= \rnSheb(t,x\viiva s, y)
\end{align*}
and
\begin{align*}
e^{\nu(x-y) + \frac{\nu^2}{2}(t-s)} \Sheb(t,x+\nu (t-r)|s, y + \nu(s-r))\circ \sheard{r}{-\nu} &=  \Sheb(t,x\viiva s,y).
\end{align*}
\item\label{lem:cov:sca} {\rm(Scaling)} For each $(s,y,t,x,\beta)\in\varsets\times\bbR$ and each $\lambda>0$, there is an event $\Omega_1=\Omega_1(s,y,t,x,\beta,\lambda)$ with $\bbP(\Omega_1) =1$ so that on $\Omega_1$,
\begin{align*}
\rnShe_{\beta/\sqrt{\lambda}}(\lambda^2 t, \lambda x | \lambda^2 s, \lambda y)\circ \did{\lambda^{-2}}{\lambda^{-1}} &= \rnSheb(t,x\viiva s,y)\qquad \text{ and }\\ 
\lambda \She_{\beta/\sqrt{\lambda}}(\lambda^2 t, \lambda x | \lambda^2 s, \lambda y) \circ \did{\lambda^{-2}}{\lambda^{-1}} &= \Sheb(t,x\viiva s,y). 
\end{align*}
\item\label{lem:cov:neg}  {\rm(Negation)} For each $(s,y,t,x,\beta) \in \varsets\times\bbR$, there is an event $\Omega_1$$=\Omega_1(s,y,t,x,\beta)$ with $\bbP(\Omega_1) =1$ so that on $\Omega_1$,
\begin{align*}
\rnShe_{-\beta}(t,x\viiva s,y)\circ \ned&= \rnSheb(t,x\viiva s,y)\qquad \text{ and }
\She_{-\beta}(t,x\viiva s,y) \circ \ned = \Sheb(t,x\viiva s,y).
\end{align*}
 \end{enumerate}
\end{lemma}
\begin{proof}
%
%
We write the details of parts \eqref{lem:cov:ref}, \eqref{lem:cov:shear}, and \eqref{lem:cov:sca}. Parts \eqref{lem:cov:shift} and \eqref{lem:cov:neg} are similar, but easier. For $m \in \bbN$ and $f \in L^2((\bbR^2)^m)$, denote the multiple Wiener-It\^o stochastic integral by
\begin{align*}
I_m(f) &= \int_{(\bbR^2)^m}  f \, dW^m.
\end{align*}
Let $\sG \in \{\shiftd{-u}{-z},\refd_1, \refd_2,\sheard{r}{-\nu}, \did{\lambda^{-2}}{\lambda^{-1}}, \ned\}$ be one of the transformations as in the statement and let $\oG \in \{\shiftf{u}{z},\reff,\shearf{r}{\nu}, \dif{\lambda^{2}}{\lambda}, \nef\}$ be the associated dual transformation on functions in $L^2(\bbR^2)$. Then by \cite[Theorem 4.5]{Jan-97}, which we may apply by \cite[Theorems 7.25 and 7.26]{Jan-97},
\begin{align*}
I_m(f)\circ \sG &= I_m (\oG_m f), \qquad \bbP-\text{a.s}.
\end{align*}
where $\oG_m$ is the bounded linear operator mapping $L^2((\bbR^2)^m)$ to itself which acts on product form functions $f(\bfx_1,\dots,\bfx_m) = \prod_{i=1}^m f_i(\bfx_i)$ by $\oG_m f(\bfx_1,\dots,\bfx_m) = \prod_{i=1}^m (\oG f_i)(\bfx_i)$, where $\bfx_{1:m} \in (\bbR^2)^m$. In all of the computations below, we view the heat kernel which appears as the first term in the chaos expansion as a constant random variable (in the $\w$ coordinate), but omit the effect of $\sG$ on such random variables, because the action is trivial.

Take $\sG = \refd_1$. Call $t_{1:m}=(t_1,\dots,t_m)$ and $-t_{m:1}=(-t_m,\dots,-t_1)$. Note that
 \begin{align}
s<t_1<\dotsm<t_m< t \iff\    -t< -t_m < \dotsm<-t_1<-s. \label{eq:torder}
 \end{align}
Moreover, for $t_{1:m}$ as in \eqref{eq:torder}, and with the convention that $t_{m+1} = t$ and $t_0 = s$, we have
 \begin{align*}
\oG_m \heat_m(\aabullet\viiva  -t,x;-s,y) &(t_{m:1},x_{m:1}) = \heat_m(-t_{m:1}, x_{m:1} \viiva -t,x; -s,y ) \\
&=    \prod_{i=0}^{m} \rho\bigl(-t_i-(-t_{i+1}) , x_i - x_{i+1} )\bigr) \\
&= \prod_{i=0}^{m} \rho\bigl(t_{i+1}-t_i , x_{i+1}-x_i\bigr) =\heat_m(t_{1:m}, x_{1:m} \viiva s,y ;t,x).  
\end{align*} 

Note that if any non-identity permutation of $[m]$ is applied to the indices $i$ of the coordinates $(t_i,x_i)$, then all of the above expressions would be equal to zero. For $s<t$, the chaos series representation \eqref{eq:SHEchaos} gives
\begin{align*}
&\Sheb(-s,y\viiva-t,x) \circ \sG = \heat(-s+t,y-x) + \sum_{m=1}^\infty \beta^m  I_m[\tspa\heat_m(\aabullet\viiva -t,x \tspb;\tspb -s,y)\tspa]  \circ \sG\\
&\qquad = \heat(-s+t,y-x) + \sum_{m=1}^\infty \beta^m  I_m[\tspa \oG_m \heat_m(\aabullet \viiva -t,x\tspb;\tspb-s,y )\tspa]\\
&\qquad = \heat(t-s,x-y) + \sum_{m=1}^\infty \beta^m  I_m[\tspa\heat_m(\aabullet \viiva s,y\tspb;\tspb  t,x  )\tspa]  = \Sheb(t,x\viiva s,y). 
\end{align*}
In the expression above, we have used the symmetrization  in the definition of the multiple stochastic integral for general $L^2((\bbR^2)^m)$ functions (for example, item (ii) on p.~9 of \cite{Nua-06})  to re-order the coordinates into the unique order for which the integrand is non-zero. Dividing through by $\rho(t-s,x-y) = \rho(-s-(-t),y-x)$ gives the analogous identity for $\rnSheb$. The proof for $\sG = \refd_2$ is similar. This completes the proof of part \eqref{lem:cov:ref}.


\smallskip 

Next, we consider $\sG = \sheard{r}{-\nu}$. Once again, take $t_{1:m}$ satisfying the order in \eqref{eq:torder} and let $x_{1:m} \in \bbR^m$ be arbitrary. We maintain the convention that $t_{m+1}=t$, $t_0=s$ and $x_0 = y, x_{m+1}=x$. Introduce the shorthands $\Delta x_i = x_{i+1} -x_i$ and $\Delta t_i = t_{i+1}-t_i$. Write $x'_i = x_i + \nu (t_i-r)$ and set $\Delta x'_i = x'_{i+1} - x'_i.$ Then $(\Delta x'_i)^2 = (\Delta x_i)^2 + 2 \nu \Delta t_i \Delta x_i + \nu^2(\Delta t_i)^2$. We have for $m\in \bbN$,
\begin{align*}
&\oG_m \heat_m(\aabullet \viiva  s,y +\nu(s-r)  \tspb;\tspb t,x+\nu(t-r) ) (t_{1:m}, x_{1:m}) \\
&\qquad = \prod_{i=0}^m (2\pi \Delta t_i)^{-1/2} \exp\bigg\{-\sum_{i=0}^m \frac{(\Delta x_i)^2}{2 \Delta t_i} - \nu \sum_{i=0}^m \Delta x_i - \frac{\nu^2}{2} \sum_{i=0}^m \Delta t_i \bigg\},\\
&\qquad= e^{-\nu(x-y) - \frac{\nu^2}{2}(t-s)}\heat_m(\aabullet\viiva   s,y \tspb;\tspb t,x )(t_{1:m},x_{1:m})
\end{align*}
Similarly, 
\begin{align}
\heat(t-s,x-y +\nu(t-s)) &= e^{-\nu(x-y) - \frac{\nu^2}{2}(t-s)}\heat(t-s,x-y).\label{eq:heatshear}
\end{align}
Consequently,
\begin{align*}
&\Sheb(t,x+\nu(t-r)\viiva s,y+\nu(s-r)) \circ \sG \\
&= \heat(t-s,x-y +\nu(t-s)) + \sum_{m=1}^\infty \beta^m I_m\big[\heat_m(\aabullet\viiva s,y +\nu(s-r)  \tspb;\tspb t,x+\nu(t-r)) \big] \circ \sG\\
&= \heat(t-s,x-y +\nu(t-s)) + \sum_{m=1}^\infty \beta^m I_m\big[\oG_m \heat_m(\aabullet\viiva s,y +\nu(s-r)  \tspb;\tspb t,x+\nu(t-r)) \big] \\
&= e^{-\nu(x-y) - \frac{\nu^2}{2}(t-s)}\Big(\heat(t-s,x-y) +\sum_{m=1}^\infty \beta^m I_m\big[\heat_m(\aabullet\viiva   s,y \tspb;\tspb t,x) \big] \Big) \\
&= e^{-\nu(x-y) - \frac{\nu^2}{2}(t-s)}\Sheb(t,x\viiva s,y).
\end{align*}
Dividing by $\heat(t-s,x-y+\nu(t-s))$ and appealing to \eqref{eq:heatshear} gives the corresponding result in part \eqref{lem:cov:shear} for $\rnSheb$.

\smallskip 

Finally, turning to $\sG=\did{\lambda^{-2}}{\lambda^{-1}}$, we have for $s<t$, $x,y \in \bbR$, and $\lambda>0$,
\begin{align*}
\lambda \heat(\lambda^2(t-s),\lambda(x-y)) &= \frac{\lambda}{\sqrt{2\pi\lambda^2(t-s)}} e^{-\frac{(\lambda(x-y))^2}{2\lambda^2(t-s)}}= \heat(t-s,x-y).
\end{align*}
Similarly, for $t_{1:m}$ as in \eqref{eq:torder} and $x_{1:m} \in \bbR^m$, 
\begin{align*}
&\lambda \oG_m \heat_m(\aabullet\viiva   \lambda^2 s, \lambda y \tspb;\tspb  \lambda^2t, \lambda x  )  (t_{1:m},x_{1:m}) =\lambda^{\frac32m+1} \prod_{i=0}^m \heat(\lambda^2(t_{i+1}-t_i),\lambda(x_{i+1} - x_i)) \\
&\qquad= \lambda^{m/2} \prod_{i=0}^m  \heat(t_{i+1}-t_i,x_{i+1}-x_i) = \lambda^{m/2} \heat_m(\aabullet\viiva  s, y \tspb;\tspb t, x )(t_{1:m},x_{1:m}).
\end{align*}
Then
\begin{align*}
&\lambda \She_{\beta/\sqrt{\lambda}}(\lambda^2 t, \lambda x | \lambda^2 s, \lambda y)\circ \sG \\
&\qquad = \lambda\heat(\lambda^2(t-s),\lambda(x-y)) + \lambda\sum_{m=1}^\infty \bigg(\frac{\beta}{\sqrt{\lambda}}\bigg)^{\!m}   I_m\big[\heat_m(\aabullet\viiva \lambda^2 s, \lambda y ;  \lambda^2 t, \lambda x) \big] \circ \sG\\
&\qquad = \lambda\heat(\lambda^2(t-s),\lambda(x-y)) +  \sum_{m=1}^\infty \bigg(\frac{\beta}{\sqrt{\lambda}}\bigg)^{\!m}  I_m\big[\lambda \oG_m\heat_m(\aabullet\viiva  \lambda^2 s, \lambda y ; \lambda^2 t, \lambda x)\big] \\
&\qquad= \heat(t-s,x-y) + \sum_{m=1}^\infty \beta^m I_m\big[\heat_m(\aabullet\viiva s, y ; t, x)\big] = \Sheb(t,x\viiva s,y).
\end{align*}
The second part of \eqref{lem:cov:sca} follows by dividing by $\lambda \heat(\lambda^2(t-s),\lambda(x-y))=\heat(t-s,x-y).$
\end{proof}

We next turn toward the moment estimates which we use in our application of Kolmogorov-Chentsov. 
\begin{lemma}\label{lem:mombd}
For all $p \in\bbR$, all $\beta \in \bbR$, and all $t\geq 0$,
\begin{align*}
\Mob{p}{t}{\beta}&:= \sup_{\substack{0 \leq s \leq t}} \bbE[\rnSheb(s,0\viiva 0,0)^p] < \infty.
\end{align*}
Moreover, calling
\begin{align}
\Mom{p}{t}:= \sup_{\substack{0 \leq s \leq t}} \bbE[\rnShe_{1}(s,0\viiva 0,0)^p], \qquad \text{we have} \qquad \Mob{p}{t}{\beta}= \Mom{p}{t\beta^4}.\label{eq:Mobsca}
\end{align}
\end{lemma}
\begin{proof}
For $\beta =0$, $t=0$, or $p=0$, there is nothing to prove in either claim because $\rnSheb(t,0\viiva 0,0)^p=1$ in any of these cases. If $\beta,p \neq0$ and $t>0$, \eqref{eq:Mobsca} follows from Lemma \ref{lem:cov}\eqref{lem:cov:sca} and \eqref{lem:cov:neg}. 

Using the Burkholder-Davis-Gundy inequality and heat kernel estimates, it can be shown that for each $p>0$, there exists $C=C(p)>0$ so that 
\begin{align}
\Mob{p}{t}{\beta}\leq Ce^{t \beta^4 C} \label{eq:momexp}
\end{align}
for all $t>0$. See \cite[Example 2.10]{Che-Dal-15}.

For $\beta \neq 0$, $t>0$, and $p<0$, finiteness follows from \cite[Theorem 1]{Mor-14}. See also \cite[Theorem 1.7]{Das-Gho-21-} for a more refined version of the same idea. 
\end{proof}
\begin{remark}
Using inputs from integrable probability, Das and Tsai showed in \cite[Theorem 1.2]{Das-Tsa-21} the sharp result that for all $p>0$, $\lim_{t\to\infty} t^{-1} \log \Mom{p}{t}  = \frac{p^3-p}{12}$.
\end{remark}

With the previous notation in mind, the first main goals in this section are the moment estimates on the increments of $\rnSheb(t,x\viiva s,y)$ and $\Sheb(t,x\viiva s,y)$.  We start with the spatial increments.
\begin{lemma}\label{lem:xymoment}
For $p > 2$, there exists $C=C(p)$ so that for all $t\geq0$, $x,y \in \bbR$, and $\beta \in \bbR$, 
\begin{align*}
\bbE[\,|\rnSheb(t,y\viiva 0,0) - \rnSheb(t,x\viiva 0,0)|^p] &\leq C \Mob{p}{t}{\beta} \tspa |\beta|^p \tspa |x-y|^{p/2} 
\end{align*}
\end{lemma}

\begin{proof}
The result is trivial when either $t=0$ or $\beta = 0$, since both terms in the absolute value are then equal to $1$, so we assume that $t>0$. By Lemma \ref{lem:fixGreen}, $\rnSheb(\aabullet,\aabullet\viiva 0,0)$ admits a modification solving the mild equation obtained from \eqref{eq:Greens} by dividing by the heat kernel: 
\begin{align*}
\rnSheb(t,x\viiva 0,0) &= 1 + \beta \int_0^t \int_{\bbR} \frac{\heat(t-r,x-z)\heat(r,z)}{\heat(t,x)}\rnSheb(r,z\viiva 0,0)W(dz\,dr).
\end{align*}
Using the Burkholder-Davis-Gundy inequality \cite[Theorem 4.2.12]{Bic-02} and then H\"older's inequality with conjugate exponents $p/2$ and $p/(p-2)$, there exists $C=C(p)$ so that 
\begin{align*}
&\bbE[|\rnSheb(t,y\viiva 0,0) - \rnSheb(t,x\viiva 0,0)|^p]   \\
&\leq C |\beta|^p \mathbb{E}\bigg[\bigg(\int_0^t \int_{\bbR}  \bigg( \frac{\heat(t-r,y-z )\heat(r,z)}{\heat(t,y)}-  \frac{\heat(t-r,x-z) \heat(r,z)}{\heat(t,x)} \bigg)^2\rnSheb(r,z\viiva 0,0)^2 dz dr\bigg)^{p/2}   \bigg] \\
&= C |\beta|^p \mathbb{E}\bigg[ \bigg(\int_0^t \int_{\bbR} \bigg( \frac{\heat(t-r,y-z)\heat(r,z)}{\heat(t,y)}-  \frac{\heat(t-r,x-z) \heat(r,z)}{\heat(t,x)}\bigg)^{2-4/p} \\ 
&\qquad\qquad \times\bigg( \frac{\heat(t-r,y-z )\heat(r,z)}{\heat(t,y)}-  \frac{\heat(t-r,x-z) \heat(r,z)}{\heat(t,x)}\bigg)^{4/p} \rnSheb(r,z\viiva 0,0)^2dz dr \bigg)^{p/2}   \bigg] \\
&\leq C |\beta|^p  \bigg[\int_0^t \int_{\bbR} \bigg( \frac{\heat(t-r,y-z )\heat(r,z)}{\heat(t,y)}-  \frac{\heat(t-r,x-z) \heat(r,z)}{\heat(t,x)}\bigg)^{2}dz dr \bigg]^{\frac{p-2}{2}}  \\
&\qquad \qquad \times\bigg[\int_0^t \int_{\bbR} \bigg( \frac{\heat(t-r,y-z)\heat(r,z)}{\heat(t,y)}-  \frac{\heat(t-r,x-z) \heat(r,z)}{\heat(t,x)}\bigg)^{2} \bbE[\rnSheb(r,u\viiva 0,0)^p] dz dr\bigg] \\
&\leq C |\beta|^p \Mob{p}{t}{\beta} \bigg[ \int_0^t  \int_{\bbR} \bigg( \frac{\heat(t-r,y-z)\heat(r,z)}{\heat(t,y)}-  \frac{\heat(t-r,x-z) \heat(r,z)}{\heat(t,x)}\bigg)^{2}dz dr \bigg]^{p/2}.
\end{align*}
To finish, 
\begin{align*}
&\int_0^t   \int_{\bbR} \left( \frac{\heat(t-r,y-z)\heat(r,z)}{\heat(t,y)} -  \frac{\heat(t-r,x-z) \heat(r,z)}{\heat(t,x)}\right)^{2} dz dr\\
&= \int_0^t  \int_{\bbR} \bigg[ \frac{\heat(t-r,y-z)^2\heat(r,z)^2}{\heat(t,y)^2} + \frac{\heat(t-r,x-z)^2 \heat(r,z)^2}{\heat(t,x)^2} \\
&\qquad\qquad\qquad\qquad - 2  \frac{\heat(t-r,y-z)\heat(r,z)}{\heat(t,y)} \frac{\heat(t-r,x-z) \heat(r,z)}{\heat(t,x)}\bigg]dz dr \leq |x-y|,
\end{align*}
where the last bound follows from Lemmas \ref{lem:int1} and  \ref{lem:xybd}.
\end{proof}
We now estimate the increments of the process in the inverse temperature $\beta$:
\begin{lemma}\label{lem:betamoment}
For each $p>2$, there exists $C=C(p)>0$ such that for all $t \geq 0$ and $x, \beta_1,\beta_2\in \bbR$,  
\begin{align*}
\bbE[|\rnShe_{\beta_1}(t,x\viiva 0,0) - \rnShe_{\beta_2}(t,x\viiva 0,0)|^p] \leq C t^{p/4}|\beta_1-\beta_2|^p \Mob{p}{t}{\beta_2}e^{ C |\beta_1|^p t^{p/4}}.
\end{align*}
\end{lemma}
\begin{proof}
Notice that the shear invariance in Lemma \ref{lem:cov}\eqref{lem:cov:shear} implies that for all $z \in \bbR$ and $r>0$,
\begin{align*}
\bbE\big[\big|\rnShe_{\beta_1}(r,z\viiva 0,0) - \rnShe_{\beta_2}(r,z\viiva 0,0)\big|^p\big] = \bbE\big[\big|\rnShe_{\beta_1}(r,0\viiva 0,0) - \rnShe_{\beta_2}(r,0\viiva 0,0)\big|^p\big]
\end{align*}
and similarly, that
\begin{align*}
\bbE[\rnSheb(t,x\viiva 0,0)^p]=\bbE[\rnSheb(t,0\viiva 0,0)^p].
\end{align*}
Abbreviate $\rnShe_{\beta}(r,z) =\rnShe_{\beta}(r,z\viiva 0,0)$. Appealing to the Burkholder-Davis-Gundy and H\"older inequalities (again with conjugate exponents $p/2$ and $p/(p-2)$) in the same way as in the proof of Lemma \ref{lem:xymoment}, there exist $C,C'$ depending only on $p$ so that
\begin{align*}
&\bbE[\tspb|\!\rnShe_{\beta_1}(t,x) - \rnShe_{\beta_2}(t,x)|^p]\\
&\leq|\beta_1-\beta_2|^p\bbE\bigg[\bigg| \int_0^t \int_{\bbR}\frac{\heat(t-r,x-z)\heat(r,z)}{\heat(t,x)} \rnShe_{\beta_2}(r,z)W(dz\,dr) \bigg|^p \bigg] \\
&\qquad+|\beta_1|^p\bbE\bigg[\bigg| \int_0^t \int_{\bbR}\frac{\heat(t-r,x-z)\heat(r,z)}{\heat(t,x)} (\rnShe_{\beta_2}(r,z) -\rnShe_{\beta_1}(r,z))W(dz\,dr) \bigg|^p \bigg] \\
&\leq C |\beta_1-\beta_2|^p \bigg(\int_0^t\int_{\bbR} \frac{\heat(t-r,x-z)^2\heat(r,z)^2}{\heat(t,x)^2}dzdr\bigg)^{\frac{p}{2}-1}  \\
&\qquad\qquad\times\int_0^t\int_{\bbR} \frac{\heat(t-r,x-z)^2\heat(r,z)^2}{\heat(t,x)^2}dz\bbE[\rnShe_{\beta_2}(r,0)^p]dr\\
&\qquad+ C |\beta_1|^p\bigg(\int_0^t\int_{\bbR} \frac{\heat(t-r,x-z)^2\heat(r,z)^2}{\heat(t,x)^2}dzdr\bigg)^{\frac{p}{2}-1}\\
&\qquad\qquad\qquad\times\int_0^t\int_{\bbR} \frac{\heat(t-r,x-z)^2\heat(r,z)^2}{\heat(t,x)^2}dz\bbE[|\rnShe_{\beta_2}(r,0) -\rnShe_{\beta_1}(r,0)|^p]dr\\
&\leq C' |\beta_1-\beta_2|^p t^{\frac{p}{4}}\Mob{p}{t}{\beta_2} + C'|\beta_1|^p t^{\frac{p}{4}-\frac{1}{2}} \int_0^t \sqrt{\frac{t}{(t-r)r}}\bbE[|\rnShe_{\beta_2}(r,x) -\rnShe_{\beta_1}(r,x)|^p]dr.
\end{align*}
In the third inequality, we appealed to the computations in Lemmas \ref{lem:int2} and \ref{lem:int1}. In the last step, we used the shear invariance coming from Lemma \ref{lem:cov}\eqref{lem:cov:shear} again to switch $0$ to $x$ in the expectation. It follows from Gronwall's inequality \cite[Lemma A.2.35]{Bic-02} and the computation in Lemma \ref{lem:int1} that there is $C'' =C''(p)>0$ so that
\begin{align*}
\bbE[|\rnShe_{\beta_1}(t,x\viiva 0,0) - \rnShe_{\beta_2}(t,x\viiva 0,0)|^p] &\leq C'' t^{p/4}|\beta_1-\beta_2|^p \Mob{p}{t}{\beta_2}e^{ C''|\beta_1|^p t^{p/4}}. \qedhere
\end{align*}
\end{proof}

We include two estimates for time differences.  The first one will result in a non-sharp H\"older exponent for all nonnegative times,  while the second one results in a sharp H\"older exponent at times  bounded away from zero.   The reason our bounds are not sharp at the boundary $t=0$ is that we use a crude bound in Lemma \ref{lem:withgap} to simplify the computation.
\begin{lemma}\label{lem:thmoment}
For $p>2$, there exists $C=C(p)>0$ so that for all $\beta \in \bbR$, all $T,K\geq 1$ and all $h \in (0,1)$, if $t,t+h \in [0,T]$ and $x \in [-K,K]$, then
\begin{align*}
\bbE[\tspb|\!\rnSheb(t+h,x\viiva 0,0) - \rnSheb(t,x\viiva 0,0)|^p] &\leq C \Mom{p}{T} |\beta|^p T^{3p/4} K^{p} h^{p/14}.
\end{align*}
Moreover, for each $\delta>0$, if in addition we have $t,t+h \in [\delta,T]$, then
\begin{align*}
\bbE[\tspb|\!\rnSheb(t+h,x\viiva 0,0) - \rnSheb(t,x\viiva 0,0)|^p] \leq C \Mom{p}{T} |\beta|^p \delta^{-3p/2}T^{3p/4} K^{p} h^{p/4}.
\end{align*}
\end{lemma}
\begin{proof}
Again for $r>0$ and $z \in \bbR$, abbreviate $\rnSheb(r,z) =\rnSheb(r,z\viiva 0,0)$. Appealing again to the Burkholder-Davis-Gundy and H\"older inequalities (again, with conjugate exponents $p/2$ and $p/(p-2)$), there exist $C,C',C''>0$ depending only on $p$ so that the following hold, with the convention that integrals below on $[0,t]$ are defined to be zero if $t=0$:
\begin{align*}
&\bbE[|\rnSheb(t+h,x) - \rnSheb(t,x)|^p] \leq  \\
&C|\beta|^p \mathbb{E}\bigg[ \bigg|\int_0^t  \int_{\bbR} \bigg( \frac{\heat(t-r,x-z)\heat(r,z)}{\heat(t,x)}-  \frac{\heat(t+h-r,x-z) \heat(r,z)}{\heat(t+h,x)}\bigg)\rnSheb(r,z) W(dz\,dr)\bigg|^{p}   \bigg] \\
&\qquad+ C'|\beta|^p \mathbb{E}\left[ \bigg|\int_t^{t+h} \int_{\bbR}  \frac{\heat(t+h-r,x-z) \heat(r,z)}{\heat(t+h,x)} \rnSheb(r,z)W(dz\, dr)\bigg|^{p}   \right] \\
&\leq C'|\beta|^p\mathbb{E}\bigg[ \bigg|\int_0^t  \int_{\bbR} \bigg( \frac{\heat(t-r,x-z)\heat(r,z)}{\heat(t,x)}-  \frac{\heat(t+h-r,x-z) \heat(r,z)}{\heat(t+h,x)}\bigg)^2\rnSheb(r,z)^2 dz\,dr\bigg|^{p/2}\bigg]\\
&\qquad+ C'|\beta|^p \mathbb{E}\left[ \bigg|\int_t^{t+h} \int_{\bbR}  \frac{\heat(t+h-r,x-z)^2 \heat(r,z)^2}{\heat(t+h,x)^2} \rnSheb(r,z)^2 dz\, dr\bigg|^{p/2}   \right] \\ 
&\leq C'|\beta|^p\Mob{p}{T}{\beta} \bigg[ \int_0^t \int_{\bbR} \left( \frac{\heat(t+h-r,x-z)\heat(r,z)}{\heat(t+h,x)}-  \frac{\heat(t-r,x-z) \heat(r,z)}{\heat(t,x)}\right)^{2}dz dr \bigg]^{p/2}\\
&\qquad+C'|\beta|^p \Mob{p}{T}{\beta} \bigg[ \int_t^{t+h} \int_{\bbR} \frac{\heat(t+h-r,x-z)^2 \heat(r,z)^2}{\heat(t+h,x)^2} dz dr \bigg]^{p/2} \\
&\leq C''|\beta|^p \Mob{p}{T}{\beta} T^{\frac{3p}{4}}K^{p}h^{\frac{p}{14}}.
\end{align*}
The last bound comes from Proposition \ref{prop:nogap} and Lemma \ref{lem:int4bd} in general. If instead we also require that $t,t+h\in[\delta,T]$, then the last bound becomes
\begin{align*}
C''|\beta|^p \Mob{p}{T}{\beta} \delta^{-\frac{3 p}{2}} T^{\frac{3p}{4}} K^{p} h^{\frac{p}{4}},
\end{align*}
again by Proposition \ref{prop:nogap} and Lemma \ref{lem:int4bd}.
\end{proof}
The previous estimates combine into the following bounds. Because of the different growth rates of our bounds and the different H\"older exponents that they imply, we estimate several H\"older semi-norms below. We restrict attention in the following results to the estimates which are required for the results of this paper and the companion \cite{Jan-Ras-Sep-22-1F1S-}.
\begin{proposition} \label{prop:nogapbd}
For $p>14$, there exists $C=C(p)$ so that for $(s_i,y_i,t_i,x_i) \in \varset$, $i \in \{1,2\}$, if we call $T = \max\{|t_1-t_2|,|t_1-s_1|,|t_2-s_2|\}\vee 1$ and $K = \max\{|x_i-x_j|,|y_i-y_j|,|x_i-y_j|: i,j \in \{1,2\}\}\vee 1$ and take $B>0$, then
\begin{enumerate}[label={\rm(\roman*)}, ref={\rm\roman*}]   \itemsep=3pt 
\item For $\beta \in [-B,B]$, \label{prop:nogapbd:main}  
\begin{align*}
&\bbE\left[\big|\rnSheb(t_1,y_1|s_1,x_1) - \rnSheb(t_2,y_2|s_2,x_2)\big|^p\right]\\
&\leq  C \Mob{p}{T}{B} B^p T^{2p} K^{p} \left(|y_1-y_2|^{p/2} +|x_1-x_2|^{p/2} +|t_1-t_2|^{p/14} + |s_1-s_2|^{p/14} \right)
\end{align*}
and for $\beta_1,\beta_2 \in [-B,B]$,
\begin{align*}
&\bbE\left[\big|\rnShe_{\beta_1}(t_1,y_1|s_1,x_1) - \rnShe_{\beta_2}(t_2,y_2|s_2,x_2)\big|^p\right]\\
&\leq  C e^{CB^p T^{p/4}}  \Mob{p}{T}{B} K^{p} \Big(|y_1-y_2|^{p/2} +|x_1-x_2|^{p/2} +|t_1-t_2|^{p/14} + |s_1-s_2|^{p/14} + |\beta_1-\beta_2|^p \Big).
\end{align*}
\item \label{prop:nogapbd:inverse} With the same notation as in part \eqref{prop:nogapbd:main}, for any $\theta_1,\theta_2 \in (0,1)$ which satisfy $\theta_1 + 2 \theta_2 = 1$, there is  $C'=C'(p,\theta_1,\theta_2)>0$ so that for all $\beta_1,\beta_2\in[-B,B]$,
\begin{align*}
&\bbE\left[\big|\rnShe_{\beta_1}(t_1,y_1|s_1,x_1)^{-1} - \rnShe_{\beta_2}(t_2,y_2|s_2,x_2)^{-1}\big|^p\right]\\
&\leq C' \exp\!\big\{C' B^{\frac{p}{\theta_1}} T^{\frac{p}{4\theta_1}} \big\}  \Mob{-2p/\theta_2}{T}{B}^{2\theta_2 }\Mob{p/\theta_1}{T}{B}^{\theta_1} K^{p} \\
&\qquad\times\bigg(|y_1-y_2|^{p/2} +|x_1-x_2|^{p/2} +|t_1-t_2|^{p/14} + |s_1-s_2|^{p/14} + |\beta_2-\beta_1|^p \bigg).
\end{align*}
\item \label{prop:nogapbd:delta} If, in addition, for some $\delta\in (0,1)$, we have $t_i - s_i>\delta$ for $i \in \{1,2\}$, we also have
\begin{align*}
&\bbE\left[\big|\rnShe_{\beta_1}(t_1,y_1|s_1,x_1) - \rnShe_{\beta_2}(t_2,y_2|s_2,x_2)\big|^p\right] \leq  C e^{CB^p T^{p/4}}  \Mob{p}{T}{B} \delta^{-3p/2} K^{p} \\
&\qquad\times\bigg(|y_1-y_2|^{p/2} +|x_1-x_2|^{p/2} +|t_1-t_2|^{p/4} + |s_1-s_2|^{p/4} + |\beta_1-\beta_2|^p \bigg)
\end{align*}
and
\begin{align*}
&\bbE\left[\big|\rnShe_{\beta_1}(t_1,y_1|s_1,x_1)^{-1} - \rnShe_{\beta_2}(t_2,y_2|s_2,x_2)^{-1}\big|^p\right]\\
&\quad\leq C' \exp\!\big\{C' B^{\frac{p}{\theta_1}} T^{\frac{p}{4\theta_1}} \big\}  \Mob{-2p/\theta_2}{T}{B}^{2\theta_2 }\Mob{p/\theta_1}{T}{B}^{\theta_1} \delta^{-3p/2} K^{p} \\
&\quad\qquad\times \Big(|y_1-y_2|^{p/2} +|x_1-x_2|^{p/2} +|t_1-t_2|^{p/4} + |s_1-s_2|^{p/4} + |\beta_2-\beta_1|^p \Big).
\end{align*}
\end{enumerate}
\end{proposition}

\begin{proof} 
Without loss of generality, we assume $s_1 \leq s_2$. Abbreviate again   $\rnShe(r,z) = \rnShe(r,z\viiva 0,0)$ and also $\beta=\beta_2$. 
\be\label{eq:H-1}
\begin{aligned}
&2^{-5p}\bbE\left[\big|\rnShe_{\beta_1}(t_1,y_1|s_1,x_1) -  \rnShe_{\beta_2}(t_2,y_2|s_2,x_2)\big|^p\right]  \\
&\qquad\leq \bbE\left[\big|\rnShe_{\beta_1}(t_1,y_1|s_1,x_1) - \rnShe_{\beta_2}(t_1,y_1|s_1,x_1)\big|^p\right] \\
&\qquad+ \bbE\left[\big|\rnShe_{\beta_2}(t_1,y_1|s_1,x_1) - \rnShe_{\beta_2}(t_2,y_1|s_1,x_1)\big|^p\right] \\
&\qquad+ \bbE\left[\big|\rnShe_{\beta_2}(t_2,y_1|s_1,x_1) - \rnShe_{\beta_2}(t_2,y_2|s_1,x_1)\big|^p\right] \\
&\qquad+ \bbE\left[\big|\rnShe_{\beta_2}(t_2,y_2|s_1,x_1) - \rnShe_{\beta_2}(t_2,y_2|s_1,x_2)\big|^p\right] \\
&\qquad+ \bbE\left[\big|\rnShe_{\beta_2}(t_2,y_2|s_1,x_2) - \rnShe_{\beta_2}(t_2,y_2|s_2,x_2)\big|^p\right]\\
&\qquad\qquad=\bbE\left[\big|\rnShe_{\beta_1}(t_1-s_1,x_1-y_1) - \rnShe_{\beta_2}(t_1-s_1,x_1-y_1)\big|^p\right] \\
&\qquad\qquad+\bbE\left[\big|\rnSheb(t_1-s_1,y_1-x_1) - \rnSheb(t_2-s_1,y_1-x_1)\big|^p\right] \\
&\qquad\qquad+ \bbE\left[\big|\rnSheb(t_2-s_1,y_1-x_1) - \rnSheb(t_2-s_1,y_2-x_1)\big|^p\right] \\
&\qquad\qquad+ \bbE\left[\big|\rnSheb(t_2-s_1,y_2-x_1) - \rnSheb(t_2-s_1,y_2-x_2)\big|^p\right] \\
&\qquad\qquad+ \bbE\left[\big|\rnSheb(t_2-s_1,y_2-x_2) - \rnSheb(t_2-s_2,y_2-x_2)\big|^p\right].
\end{aligned}\ee
The equalities follow from Lemma \ref{lem:cov}, either by shifting $(s_1,x_1)$ to the origin, or by shifting $(t_2,y_2)$ to the origin followed by a reflection:  
\begin{align*}
\rnSheb(t_2,y_2|s_1,x_1) - \rnSheb(t_2,y_2|s_1,x_2)   &\stackrel{\tiny{d}}{=} 
\rnSheb(0,0|s_1-t_2,x_1-y_2) - \rnSheb(0,0|s_1-t_2,x_2-y_2) \\
&\stackrel{\tiny{d}}{=} \rnSheb(t_2-s_1,y_2-x_1) - \rnSheb(t_2-s_1,y_2-x_2).
\end{align*}

By Lemmas \ref{lem:xymoment}, \ref{lem:betamoment}, and \ref{lem:thmoment}, in general, there exists $C=C(p)>0$ so that for any fixed $\beta>0$,
\begin{align*}
&\bbE\left[\big|\rnSheb(t_1,y_1|s_1,x_1) -  \rnSheb(t_2,y_2|s_2,x_2)\big|^p\right]  \\
&\leq C |\beta|^p \Mob{p}{T}{B} T^{p} K^{p}\left(|y_1-y_2|^{p/2} +|x_1-x_2|^{p/2}+|t_1-t_2|^{p/14} + |s_1-s_2|^{p/14} \right),
\end{align*}
and for $\beta_1,\beta_2 \in [-B,B]$,
\begin{align*}
&\bbE\left[\big|\rnShe_{\beta_1}(t_1,y_1|s_1,x_1) - \rnShe_{\beta_2}(t_2,y_2|s_2,x_2)\big|^p\right]\\
&\leq  C e^{CB^p T^{p/4}}  \Mob{p}{T}{B} B^p T^{2p}K^p \\
&\qquad\times\bigg(|y_1-y_2|^{p/2} +|x_1-x_2|^{p/2} +|t_1-t_2|^{p/14} + |s_1-s_2|^{p/14} + |\beta_1-\beta_2|^p \bigg).
\end{align*}
Note that to obtain these bounds using Lemma \ref{lem:thmoment}, we need to consider times separated by at most $1$ in Lemma \ref{lem:thmoment}. To achieve this, consider for example the last expectation in \eqref{eq:H-1}. Take $u_{1:\ell}$ such that $s_1 = u_0 < ... < u_\ell = s_2$  so that  $u_i - u_{i-1} = 1$ for each $i \in [\ell-1]$ and $u_\ell - u_{\ell-1}\le 1$. In particular, $\ell \leq 2T+1$. Write
\be
\begin{aligned}
&\big|\rnSheb(t_2-s_1,y_2-x_2) - \rnSheb(t_2-s_2,y_2-x_2)\big|^p\\
&\qquad=  \big|  \sum_{i=1}^\ell \bigl(  \rnSheb(t_2-u_{i-1},y_2-x_2) - \rnSheb(t_2-u_i, y_2-x_2) \bigr) \big|^p \\
&\qquad\le  \ell^{p-1}   \sum_{i=1}^\ell \big|  \rnSheb(t_2-u_{i-1},y_2-x_2) - \rnSheb(t_2-u_i, y_2-x_2)\big|^p
\end{aligned}
\ee
Now apply Lemma \ref{lem:thmoment} to each term above. The other time increment can be handled similarly.

Similarly, with the gap of $\delta>0$, the bound becomes
\begin{align*}
&\bbE\left[\big|\rnSheb(t_1,y_1|s_1,x_1) -  \rnSheb(t_2,y_2|s_2,x_2)\big|^p\right]  \\
&\qquad\leq C \Mob{p}{T}{B} B^p\delta^{-3p/2}T^{2p}K^p\left(|y_1-y_2|^{p/2} +|x_1-x_2|^{p/2}+|t_1-t_2|^{p/4} + |s_1-s_2|^{p/4} \right).
\end{align*}
By Lemma \ref{lem:mombd}, $\rnSheb(t,y|s,x)$ is almost surely non-zero for each $(s,x,t,y,\beta) \in \varset \times \bbR$, so we may divide by it for fixed space-time-inverse temperature quintuples. We have
\begin{align*}
\big|\rnShe_{\beta_1}(t_1,y_1|s_1,x_1)^{-1} -  \rnShe_{\beta_2}(t_2,y_2|s_2,x_2)^{-1}\big| = \frac{\big|\rnShe_{\beta_1}(t_1,y_1|s_1,x_1) -  \rnShe_{\beta_2}(t_2,y_2|s_2,x_2)\big|}{|\rnShe_{\beta_1}(t_1,y_1|s_1,x_1) \rnShe_{\beta_2}(t_2,y_2|s_2,x_2)|}.
\end{align*}
It then follows from H\"older's inequality that for any $\theta_1,\theta_2 \in (0,1)$ which satisfy $\theta_1 + 2 \theta_2 = 1$, letting $q= p/\theta_1$ and $r=p/\theta_2$, 
\begin{align*}
&\bbE\left[\big|\rnShe_{\beta_1}(t_1,y_1|s_1,x_1)^{-1} -  \rnShe_{\beta_2}(t_2,y_2|s_2,x_2)^{-1}\big|^p\right]^{1/p} \\
&\leq \bbE[|\rnShe_{\beta_1}(t_1,y_1|s_1,x_1)|^{-r}]^{1/r} \bbE[|\rnShe_{\beta_2}(t_2,y_2|s_2,x_2)|^{-r}]^{1/r} \\
&\qquad\times\bbE\left[\big|\rnShe_{\beta_1}(t_1,y_1|s_1,x_1) -  \rnShe_{\beta_2}(t_2,y_2|s_2,x_2)\big|^q\right]^{1/q} \\
&\leq C\Mob{-2r}{T}{B}^{2/r }\Mob{q}{T}{B}^{1/q} B e^{\frac{C}{q} B^q T^{q/4}}T^{2+ \frac{1}{4\theta_1}} K \\
&\qquad\times\bigg((|y_1-y_2|^{q/2} +|x_1-x_2|^{q/2}  +|t_1-t_2|^{q/14}+ |s_1-s_2|^{q/14} + |\beta_2-\beta_1|^q \bigg)^{1/q}.
\end{align*}
Recalling that $p/q = \theta_1$, $p/r=\theta_2$, and that $x^{\theta_1}$ is a subadditive function for $x>0$, we see that there is $C'=C'(\theta_1,\theta_2,p)$ so that
\begin{align*}
&\bbE\left[\big|\rnShe_{\beta_1}(t_1,y_1|s_1,x_1)^{-1} - \rnShe_{\beta_2}(t_2,y_2|s_2,x_2)^{-1}\big|^p\right]\\
&\leq C' \exp\bigg\{C' B^{\frac{p}{\theta_1}} T^{\frac{p}{4\theta_1}} \bigg\}  \Mob{-2p/\theta_2}{T}{B}^{2\theta_2 }\Mob{p/\theta_1}{T}{B}^{\theta_1}K^{p} \\
&\qquad\times \bigg(|y_1-y_2|^{p/2} +|x_1-x_2|^{p/2} +|t_1-t_2|^{p/14} + |s_1-s_2|^{p/14} + |\beta_2-\beta_1|^p \bigg).
\end{align*}
If, in addition, we have $t_i-s_i>\delta$ for $i \in \{1,2\}$, then by the same argument,
\begin{align*}
&\bbE\left[\big|\rnShe_{\beta_1}(t_1,y_1|s_1,x_1)^{-1} - \rnShe_{\beta_2}(t_2,y_2|s_2,x_2)^{-1}\big|^p\right]\\
&\leq C' \exp\bigg\{C' B^{\frac{p}{\theta_1}} T^{\frac{p}{4\theta_1}} \bigg\}  \Mob{-2p/\theta_2}{T}{B}^{2\theta_2 }\Mob{p/\theta_1}{T}{B}^{\theta_1} \delta^{-3p/2}K^{p}  \\
&\qquad\times\bigg(|y_1-y_2|^{p/2} + |x_1-x_2|^{p/2}+|t_1-t_2|^{p/4} + |s_1-s_2|^{p/4} +|\beta_2-\beta_1|^p\bigg). \qedhere
\end{align*}
\end{proof}
Proposition \ref{prop:nogapbd} implies the existence of a H\"older continuous modification of $\rnSheb(t,x\viiva s,y)$ and, away from the line $t=s$, of $\Sheb(t,x\viiva s,y)$.
\begin{proposition}\label{prop:rncont}
The process $(s,y,t,x,\beta) \mapsto \rnSheb(t,x\viiva s,y)$ defined on $(\varset\times \bbR)\cap \bbD^5$ by \eqref{eq:rnchaos} admits a unique (up to indistinguishability) modification $\tspb\widetilde{\!\rnShe}_{\aabullet}(\aabullet,\aabullet|\aabullet,\aabullet) \in \sC(\varset\times\bbR,\bbR)$. This modification satisfies the following conditions:
\begin{enumerate} [label={\rm(\roman*)}, ref={\rm\roman*}]   \itemsep=3pt  
\item \label{prop:rncont:Holderb} For each $T,K,B\geq 1$, $p>70$, 
$\alpha \in (0,1/14-5/p)$, $\gamma \in(0, 1/2-5/p)$, and $\eta\in(0,1-5/p)$,
we have
\begin{align}
\bbE\Big[\bigl\lvert \tspb\widetilde{\!\rnShe}_{\aabullet}(\aabullet,\aabullet|\aabullet,\aabullet)\bigr\rvert_{\sC^{\alpha,\gamma,\eta}(\varsett{T}{K} \times [-B,B])}^p\Big] \leq C e^{CB^p T^{p/4}} K^{p}, \label{eq:rnHolderb}
\end{align}
for some $C=C(p,\alpha,\gamma,\eta)>0$. Moreover, if $\delta \in (0,1)$, then for  $\alpha \in (0,1/4-5/p)$, $\gamma \in(0, 1/2-5/p)$, and $\eta\in(0,1-5/p)$,
\begin{align}
\bbE\Big[\bigl\lvert \tspb\widetilde{\!\rnShe}_{\aabullet}(\aabullet,\aabullet|\aabullet,\aabullet)\bigr\rvert_{\sC^{\alpha,\gamma,\eta}(\varsetsg{T}{K}{\delta}\times[-B,B])}^p\Big] \leq  C e^{CB^p T^{p/4}} K^{p} \delta^{-3p/2}. \label{eq:rnHoldergapb}
\end{align}
\item \label{prop:rncont:Holder} For each $\beta \in \bbR$, $T,K\geq 1$, $p>56$, 
$\alpha \in (0,1/14-4/p)$, $\gamma \in(0, 1/2-4/p)$, we have
\begin{align}
\bbE\Big[\bigl\lvert \tspb\widetilde{\!\rnShe}_{\beta}(\aabullet,\aabullet|\aabullet,\aabullet)\bigr\rvert_{\sC^{\alpha,\gamma}(\varsett{T}{K})}^p\Big] \leq  C |\beta|^p  \Mob{p}{T}{\beta} T^{2p}K^p, \label{eq:rnHolder}
\end{align}
for some $C=C(p,\alpha,\gamma)>0$. If $\delta \in (0,1)$, then moreover for 
$\alpha \in (0,1/4-4/p)$, $\gamma \in(0, 1/2-4/p)$,
\begin{align}
\bbE\Big[\bigl\lvert \tspb\widetilde{\!\rnShe}_{\beta}(\aabullet,\aabullet|\aabullet,\aabullet)\bigr\rvert_{\sC^{\alpha,\gamma}(\varsetsg{T}{K}{\delta})}^p\Big] \leq  C \Mob{p}{T}{\beta}\delta^{-3p/2} |\beta|^p T^{2p}K^p, \label{eq:rnHoldergap}
\end{align}
\item  \label{prop:rncont:invHolder} For each $T,K,B \geq 1$ and $\theta_1,\theta_2 \in (0,1)$ with $\theta_1+2\theta_2=1$, for all $p>70$, and all $\alpha \in (0,1/14-5/p)$, $\gamma \in(0, 1/2-5/p)$, and $\eta \in (1-5/p)$,
\begin{align*}
&\bbE\bigg[\bigl\lvert \tspb\widetilde{\!\rnShe}_{\aabullet}(\aabullet,\aabullet|\aabullet,\aabullet)^{-1}\bigr\rvert _{\sC^{\alpha,\gamma,\eta}(\varsett{T}{K}\times[-B,B])}^p\bigg] \leq C e^{C B^{\frac{p}{\theta_1}} T^{\frac{p}{4\theta_1}} } \Mob{-2p/\theta_2}{T}{B}^{2\theta_2 } K^{p} 
\end{align*} 
for some $C=C(p,\alpha,\gamma,\eta,\theta_1,\theta_2)>0$. In particular, $\rnSheb(t,x\viiva s,y)>0$ for all $(s,y,t,x,\beta)$ $\in \varset\times\bbR$. 
\end{enumerate}
\end{proposition}
\begin{proof}
To apply Theorem \ref{thm:KC} as stated in Appendix \ref{app:KC},  we fix $T,K,B>1$ and map $\varsett{T}{K}\times[-B,B]$ to $\sK = \{(s,y,t,x,\beta) :0 \leq x, y,\beta \leq 1, 0 \leq s \leq t \leq 1\}$ by defining 
\begin{align}
f_{T,K,B}(s,y,t,x,\beta) &= 
\rnShe_{-B+2B\beta}(-T + 2T t, -K + 2K x| -T +  2T s, -K + 2Ky) \label{eq:fcases}
\end{align}
for $(s,y,t,x,\beta) \in \sK$, where $\rnSheb(t,x\viiva s,y)$ is defined to be equal to one if $\beta = 0$ or $s=t$ and is given by \eqref{eq:rnchaos} otherwise.  By Proposition \ref{prop:nogapbd}\eqref{prop:nogapbd:main},
 \begin{align}
&\bbE[\,|f_{T,K,B}(t_1,x_1,s_1,y_1,\beta_1) - f_{T,K,B}(t_2,x_2,s_2,y_2,\beta_2)|^p\,] \leq  C e^{CB^p T^{p/4}}  \Mob{p}{T}{B} K^{p}\times \label{489}\\
&\Big(|t_1-t_2|^{p/14} + |s_1-s_2|^{p/14} + |x_1-x_2|^{p/2} + |y_1-y_2|^{p/2} + |\beta_2-\beta_1|^p\Big),  \notag
\end{align}
for $(t_1,x_1,s_1,y_1,\beta_1),(t_2,x_2,s_2,y_2,\beta_2)\in \sK$. 

With $p>70$, \eqref{489} gives assumption \eqref{kc-hyp1} with $\nu=p$, $d=5$, $\alpha_i=\frac{p}{14}-5$ for the time increments, $\alpha_i=\frac{p}{2}-5$ for the space increments, and $\alpha_i = p-5$ for the inverse temperature increments.
By Theorem \ref{thm:KC},   there is an event $\Omega_{T,K,B}$ with $\bbP(\Omega_{T,K,B})=1$ such that $f_{T,K,B}$ admits an almost surely unique continuous modification of $f_{T,K,B}$ to $\sK$ which agrees with \eqref{eq:fcases}, defined through \eqref{eq:rnchaos}, on the dyadic rationals.  Then, the bound \eqref{kolm-c-4} and the estimate in \eqref{eq:momexp} implies that
\be\label{493} 
\bbE[|f_{T,K,B}|_{\sC^{\alpha,\gamma,\eta} (\sK)}^p] \leq  C e^{CB^p T^{p/4}} K^{p}
\ee
for some $C=C(p,\alpha,\gamma,\eta)>0$. To conclude the proof of \eqref{prop:rncont:Holder}, 
the extension of  $f_{T,K}$ to $\sK$ is, by definition, an almost surely unique continuous modification of $\rnSheb(t,x\viiva s,y)$ to $\varsett{T}{K} \times [-B,B]$  which agrees with $\rnShe_{-B + 2T \hat{\beta}}(-T + 2T \hat{t}, -K + 2K \hat{x}| -T +  2T \hat{s}, -K + 2K\hat{y})$ defined via \eqref{eq:rnchaos} for dyadic rational $(\hat{t},\hat{x},\hat{s},\hat{y},\hat{\beta})\in\sK$. On the intersection $\bigcap_{n \in \bbN} \Omega_{2^n,2^n,2^n}$, consistency then gives a unique extension of $\rnSheb(t,x\viiva s,y)$ to $\varset \times \R$ which agrees with \eqref{eq:rnchaos} on $(\varset \times \R) \cap \bbD^5$. 
By transferring the variables again as in \eqref{eq:fcases}, 
\begin{align*}
&|f_{T,K}|_{C^{\alpha,\gamma,\eta} (\sK)} \\
&=\sup_{\substack{ (t_1',x_1',s_1',y_1',\beta_1') \neq  \\
(t_2',x_2',s_2',y_2',\beta_2') \\ (t_i',x_i',s_i',y_i',\beta_i') \in \sK, \\i \in \{1,2\}} } \frac{|f_{T,K,B}(t_1',x_1',s_1',y_1',\beta_1') - f_{T,K,B}(t_2',x_2',s_2',y_2',\beta_2')|}{|t_1' - t_2'|^\alpha + |s_1'-s_2'|^\alpha + |x_1'-x_2'|^\gamma + |y_1'-y_2'|^\gamma + |\beta_1' -\beta_2'|^{\eta}}  \\
&=\sup_{\substack{(t_1,x_1,s_1,y_1,\beta_1)\neq \\
(t_2,x_2,s_2,y_2,\beta_2), \\ (t_i,x_i,s_i,y_i,\beta_i) \in\\
\varsett{T}{K}\times[-B,B],\\ i \in \{1,2\} } } \bigg\{ \frac{|\rnShe_{\beta_1}(t_1,x_1|s_1,y_1) - \rnShe_{\beta_2}(t_2,x_2|s_2,y_2)|}{\frac{|t_1-t_2|^\alpha + |s_1-s_2|^\alpha}{(2T)^\alpha}+ \frac{|x_1-x_2|^\gamma + |y_1-y_2|^{\gamma}}{(2K)^\gamma} + \frac{|\beta_1-\beta_2|^\eta}{(2B)^\eta}}\bigg\} \\
&\geq |\rnSheb(\aabullet,\aabullet|\aabullet,\aabullet)|_{C^{\alpha,\gamma,\eta}(\varsett{T}{K} \times[-B,B])}
\end{align*}
and \eqref{eq:rnHolderb} follows from \eqref{493}. The proofs of \eqref{eq:rnHolder} and \eqref{eq:rnHoldergap} are similar.

\medskip 

To prove \eqref{prop:rncont:invHolder}, we repeat the argument verbatim with \eqref{eq:fcases} replaced by
\begin{align*}
f_{T,K,B}(s,y,t,x,\beta) &= \rnShe_{-B + 2B\beta}(-T + 2T t, -K + 2K x| -T +  2T s, -K + 2Ky)^{-1}.
\end{align*}
That we are permitted to divide by $\rnSheb$ for dyadic rational $(s,y,t,x,\beta) \in \sK$ follows from Lemma \ref{lem:mombd}. With this definition, take $\theta_1,\theta_2$ as in Proposition \ref{prop:nogapbd}\eqref{prop:nogapbd:inverse}. That result and the estimate in \eqref{eq:momexp} gives
\begin{align*}   
&\bbE[\,|f_{T,K,B}(t_1,x_1,s_1,y_1,\beta_1) - f_{T,K,B}(t_2,x_2,s_2,y_2,\beta_2)|^p\,] \\
&\leq C \exp\{C B^{\frac{p}{\theta_1}} T^{\frac{p}{4\theta_1}}\} K^{p} \Mob{-2p/\theta_2}{T}{B}^{2\theta_2 }  \\
&\qquad\times\big(|y_1-y_2|^{p/2} +|x_1-x_2|^{p/2}+|t_1-t_2|^{p/14} + |s_1-s_2|^{p/14} + |\beta_1-\beta_2|^p \big)
\end{align*}
for some $C=C(p,\theta_1,\theta_2)>0$. The remainder of the proof is now identical to the previous case.
\end{proof}
Next, we show that this construction is consistent with the processes that we started with, i.e., for each fixed $s,y,\beta \in \bbR$, it defines a version of the unique continuous and adapted solution to \eqref{eq:Greens}.
\begin{lemma}\label{lem:version}
Let $\tspb\widetilde{\!\rnShe}_{\aabullet}(\aabullet,\aabullet\viiva\aabullet,\aabullet)$ be the process constructed in Proposition \ref{prop:rncont}. Then
\begin{enumerate} [label={\rm(\roman*)}, ref={\rm\roman*}]   \itemsep=3pt  
\item For all $(s,y,t,x,\beta)\in \varset \times \bbR,$ $\tspb\widetilde{\!\rnShe}_{\beta}(s,y,t,x)$ is $\fil_{s,t}$ measurable. 
\item For any fixed $s,y,\beta \in \bbR$, the process $\widetilde{\She}_{\beta}(t,x\viiva s,y)$ defined on $\{(t,x) \in \bbR^2 : t>s\}$ by $\widetilde{\She}_{\beta}(t,x\viiva s,y) = \tspb\widetilde{\!\rnShe}_{\beta}(t,x\viiva s,y)\heat(t-s,x-y)$, satisfies
\begin{align*}
\bbP\bigg(\forall t>s, x\in \bbR,\, \widetilde{\She}_{\beta}(t,x\viiva s,y) = \Sheb(t,x\viiva s,y) \bigg) = 1,
\end{align*}
where $\Sheb(t,x\viiva s,y)$ is the mild solution to \eqref{eq:Greens} coming from Lemma \ref{lem:fixGreen}.
\end{enumerate} 
\end{lemma}
\begin{proof}
Take dyadic rational sequences $s_n,t_n,x_n,y_n,\beta_n$ with $ s_n<s \leq  t < t_n$ and  $t_n \to t$, $x_n \to x$, $s_n \to s$, $y_n \to y$, and $\beta_n\to \beta$. Then by construction, $\tspb\widetilde{\!\rnShe}_{\beta_n}(t_n,x_n,s_n,y_n)$ is $\fil_{s_n,t_n}$ measurable. Taking limits and appealing to continuity and the left- and right- continuity of $\fil_{s,t}$ as defined in Section \ref{sec:setting} gives the first claim.

Fix $s,y,\beta \in \bbR$ as in the second part of the statement. By Lemma \ref{lem:fixGreen}, there exists a unique (up to indistinguishability) continuous and adapted solution \eqref{eq:Greens} satisfying the moment assumptions of that result and for each fixed $t>s$ and $x,\beta \in \bbR$, this process agrees with the chaos expansion \eqref{eq:SHEchaos} with probability one. Because both $\Sheb(\aabullet,\aabullet\viiva s,y)$ and $\widetilde{\She}_{\beta}(\aabullet,\aabullet\viiva s,y)$ are continuous, it suffices to show that for fixed dyadic rationals $t$ and $x$ with $t>s$, we have $\bbP(\widetilde{\She}_{\beta}(t,x\viiva s,y) = \Sheb(t,x\viiva s,y))=1$. Now, for a sequence $s_n,y_n$ of dyadic rational points with $s_n \in (s,t)$, $s_n \to s$, $y_n \to y$, we have $\bbP\left(\widetilde{\She}_{\beta}(t,x\viiva s_n,y_n) = \Sheb(t,x\viiva s_n,y_n)\right) = 1$ by the construction of $\widetilde{\She}_{\beta}$ in Proposition \ref{prop:rncont} and the fact that $\Sheb(t,x\viiva s_n,y_n)$ agrees with \eqref{eq:SHEchaos} with probability one. Similarly, $\Sheb(t,x\viiva s,y)$ agrees with \eqref{eq:SHEchaos} with probability one for each fixed $t,s,x,y$ with $t>s$. By Proposition \ref{prop:nogapbd}, we have $\bbE[|\Sheb(t,x\viiva s_n,y_n)-\Sheb(t,x\viiva s,y)|^p]\to0$ for all $p>2$, which implies that $\widetilde{\She}_{\beta}(t,x\viiva s_n,y_n)=\Sheb(t,x\viiva s_n,y_n)$ converges to $\Sheb(t,x\viiva s,y)$ in probability. The result now follows from almost sure continuity.
\end{proof}
With Lemma \ref{lem:version} in hand, we now complete the proof of Theorem \ref{thm:rnreg} and Proposition \ref{prop:cov}.
\begin{proof}[Proof of Theorem \ref{thm:rnreg}] 
We verify that the process constructed in Proposition \ref{prop:rncont} satisfies all of the desired conditions. Proposition \ref{prop:rncont} and Lemma \ref{lem:version} show the first, fourth, fifth, and sixth parts of the claim. To see that the second and third hold, note that by definition, we have $\tspb\widetilde{\!\rnShe}_{\beta}(t,x|t,y) = 1$ and $\tspb\widetilde{\!\rnShe}_{0}(t,x\viiva s,y) = 1$ for $(s,y,t,x,\beta)\in(\varset \times \bbR)\cap \bbD^5$. Continuity extends to all $(s,y,t,x,\beta) \in \varset\times\bbR$.
\end{proof}

\begin{proof}[Proof of Proposition \ref{prop:cov}]
Proposition \ref{prop:cov} follows from Lemma \ref{lem:cov}, Proposition \ref{prop:rncont}, and Lemma \ref{lem:version}, except when either $\beta =0$ or $t=s$. In either of these cases, $\rnSheb(t,x\viiva s,y)=1$ and so the result is trivially true.
\end{proof}

Throughout the remainder of the paper, we write $\rnSheb(t,x\viiva s,y)$ to mean the unique continuous extension provided by Proposition \ref{prop:rncont} and set $\Sheb(t,x\viiva s,y) = \heat(t-s,x-y)\rnSheb(t,x\viiva s,y)$. Lemma \ref{lem:version} justifies this, because the mild solutions to \eqref{eq:Greens} are defined only up to indistinguishability in any case. An immediate corollary of Proposition \ref{prop:rncont} is almost sure sub-polynomial growth and decay of $\rnSheb(t,x\viiva s,y)$ as a function of the spatial coordinates $x,y$ for all times and inverse temperatures simultaneously.

\begin{corollary}\label{cor:growth}
There exists $p_0>0$ such that for all $p \geq p_0$, there exists $C=C(p)$ so that for all $T,K,B>1$,
\begin{enumerate} [label={\rm(\roman*)}, ref={\rm\roman*}]   \itemsep=3pt 
\item We have
\begin{align*}
\bbE\bigg[\sup_{(s,y,t,x,\beta) \in \varsett{T}{K}\times[-B,B]}|\rnSheb(t,x\viiva s,y)|^p\bigg] &\leq C e^{CB^p T^{p/4}} K^{3p},\\
\bbE\bigg[\sup_{(s,y,t,x,\beta) \in \varsett{T}{K}\times[-B,B]} |\rnSheb(t,x\viiva s,y)|^{-p}\bigg] &\leq C e^{C B^{2p} T^{\frac{p}{2}}} \Mob{-8p}{2T}{2B}^{1/4 } K^{3p} .
\end{align*}
\item For each $\beta \in \bbR$, we have
\begin{align*}
\bbE\bigg[\sup_{(s,y,t,x) \in \varsett{T}{K}}|\rnSheb(t,x\viiva s,y)|^p\bigg] &\leq C \Mob{p}{2T}{\beta} |\beta|^p(TK)^{3p} , \\
\bbE\bigg[\sup_{(s,y,t,x) \in \varsett{T}{K}}|\rnSheb(t,x\viiva s,y)|^{-p}\bigg] &\leq C\Mob{-8p}{2T}{\beta}^{1/4 }\Mob{2p}{2T}{\beta}^{3/2} |\beta|^p(TK)^{3p}.
\end{align*}
\item We have the almost sure growth bounds 
\begin{align}
\bbP\bigg(\exists T>0 : \varlimsup_{K\to\infty} K^{-4}\sup_{(s,y,t,x,\beta) \in \varsett{T}{K}\times[-B,B]} |\rnSheb(t,x\viiva s,y)| > 0 \bigg)= 0, \label{eq:cor:growth} \\
\bbP\bigg(\exists T>0 : \varlimsup_{K\to\infty} K^{-4}\sup_{(s,y,t,x) \in \varsett{T}{K}\times[-B,B]} |\rnSheb(t,x\viiva s,y)|^{-1} > 0 \bigg)= 0. \label{eq:cor:growthinv}
\end{align}
\item We have the almost sure H\"older semi-norm growth bounds
\be\label{eq:cor:Hgrowth}\begin{aligned}
&\bbP\bigg(\exists T,B>0, \alpha \in (0,1/2), \gamma \in (0,1/4), \eta \in (0,1) :  \\
&\qquad\qquad \varlimsup_{K\to\infty}  K^{-7}|\rnShe_{\aabullet}(\aabullet,\aabullet|\aabullet,\aabullet)|_{\sC^{\alpha,\gamma,\eta}(\varsetsg{T}{K}{1/K} \times [-B,B])} > 0 \bigg)= 0.
\end{aligned}\ee
\end{enumerate}
\end{corollary}

\begin{proof}
Recalling that $\rnSheb(0,0\viiva 0,0)=1$, we have for $p>70$ and $\gamma \in (0,1/14-5/p)$,
\begin{align*}
&\sup_{(s,y,t,x,\beta) \in \varsett{T}{K}\times[-B,B]}\rnSheb(t,x\viiva s,y)^p \\
&\leq 2^p + 2^p\sup_{(s,y,t,x,\beta)\in \varsett{T}{K}\times[-B,B]} \bigg(\frac{|\rnSheb(t,x\viiva s,y)-1|}{|t|^\gamma + |x|^\gamma + |s|^\gamma + |y|^\gamma + |\beta|^\gamma}\bigg)^p 5^p(TKB)^{\gamma p} \\
&\leq 2^p + 10^p(TKB)^{p} |\rnSheb(\aabullet,\aabullet|\aabullet,\aabullet)|^{p}_{C^{\gamma,\gamma,\gamma}(\varsett{T}{K}\times[-B,B])}.
\end{align*}
Proposition \ref{prop:rncont}\eqref{prop:rncont:Holderb} then gives the first claim. The second comes from the same argument and Proposition \ref{prop:rncont} \eqref{prop:rncont:invHolder} with $\theta_1 = 1/2$ and $\theta_2 = 1/4$. 

We now prove \eqref{eq:cor:growth}. Take $p>70$, $T,B>1$ and $\epsilon>0$. We have for $N \in \bbN$,
\begin{align*}
\bbP\big(\sup_{(s,y,t,x,\beta) \in \varsett{T}{N} \times[-B,B]} \rnSheb(t,x\viiva s,y) \geq \epsilon N^{4}\big) \leq C N^{-p} \epsilon^{-p}
\end{align*}
for some $C=C(p,T,B)$. The Borel-Cantelli lemma then gives, along the sequence $N\to\infty$, 
\begin{align*}
\bbP\bigg(\varlimsup_{N\to\infty} N^{-4}\sup_{(s,y,t,x,\beta) \in \varsett{T}{N} \times[-B,B]} \rnSheb(t,x\viiva s,y)> 0 \bigg)= 0.
\end{align*}
Note that $\sup_{(s,y,t,x,\beta) \in \varsett{T}{K}\times[-B,B]} \rnSheb(t,x\viiva s,y)$ is nondecreasing in $T$, $K$, and $B$. By considering $T,B \in \bbN$, \eqref{eq:cor:growth} follows. The proof of \eqref{eq:cor:growthinv} is similar.

To see \eqref{eq:cor:Hgrowth}, we again appeal to monotonicity of $|\rnSheb(\aabullet,\aabullet|\aabullet,\aabullet)|_{\sC^{\gamma,\eta}(\varsetsg{T}{K}{1/K} \times[-B,B])}$ in $T,K,$ and $B$. Take  $\alpha,\alpha' \in (0,1/2)$ satisfying $\alpha < \alpha'$, $\gamma,\gamma' \in (0,1/4)$ satisfying $\gamma<\gamma'$, and $\eta,\eta' \in (0,1)$ satisfying $\eta < \eta'$, and let $T,B>1$  and $K>\max\{T,B\}$. Then for $(t_1,x_1,s_1,y_1,\beta_1)$ and $(t_2,x_2,s_2,y_2,\beta_2) \in \varsetsg{T}{K}{1/K} \times[-B,B]$  with $(t_1,x_1,s_1,y_1,\beta_1)$ $\neq(t_2,x_2,s_2,y_2,\beta_2)$, observing for example that $|x_2-x_1|^{\alpha'}= |x_2-x_1|^{\alpha}|x_2-x_1|^{\alpha'-\alpha} \leq |x_2-x_1|^{\alpha}\sqrt{2K},$ we have
\begin{align*}
\frac{|x_1-x_2|^{\alpha'} + |y_2-y_1|^{\alpha'} + |t_1-t_2|^{\gamma'}+ |s_2-s_1|^{\gamma'} + |\beta_2-\beta_1|^{\eta'}}{|x_1-x_2|^\alpha + |y_2-y_1|^\alpha + |t_1-t_2|^\gamma+ |s_2-s_1|^\gamma+ |\beta_2-\beta_1|^{\eta}}  \leq 5 \sqrt{2K}.
\end{align*}
This implies the bound
\begin{align*}
|\rnSheb(\aabullet,\aabullet|\aabullet,\aabullet)|_{\sC^{\alpha,\gamma,\eta}(\varsetsg{T}{K}{1/K} \times[-B,B])}  \leq5 \sqrt{2K} |\rnSheb(\aabullet,\aabullet|\aabullet,\aabullet)|_{\sC^{\alpha',\gamma',\eta'}(\varsetsg{T}{K}{1/K} \times[-B,B])}.
\end{align*}
Therefore, it is sufficient to show that for each $n \in \N$, with $\alpha_n = 1/2 - 1/4^n$, $\gamma_n = 1/4-1/8^n$, and $\eta_n = 1-1/2^n$, we have
\begin{align*}
\bbP\bigg(\exists T,B>0, : \varlimsup_{K\to\infty}  K^{-7}|\rnSheb(\aabullet,\aabullet|\aabullet,\aabullet)|_{\sC^{\alpha_n,\gamma_n,\eta_n}(\varsetsg{T}{K}{1/K} \times[-B,B])} > 0 \bigg)= 0.
\end{align*} 
This follows from the estimates in Proposition \ref{prop:rncont} exactly as in the proof of \eqref{eq:cor:growth}.
\end{proof} 

\begin{remark}\label{rem:nogo}
One quick consequence of Corollary \ref{cor:growth}, which we use frequently, is that for any $T>1$, there exists $C=C(T,B,\omega)$ so that for all $x,y \in \bbR$, all $t,s \in [-T,T]$, and all $\beta \in [-B,B]$,
\begin{align*}
C^{-1} (1+|x|^4+|y|^4)^{-1} \leq \rnSheb(t,x\viiva s,y)\leq C(1+|x|^4+|y|^4).
\end{align*}
The power $4$ above is purely an artifact of our proof.
\end{remark}

We next turn to the proof of Theorem \ref{prop:IC}, which shows basic properties of our solution to \eqref{eq:SHEf}. We start with the semi-group property of the fundamental solutions, which is essentially the Chapman-Kolmogorov identity for the continuum polymer. This result is already contained in \cite[Theorem 3.1(vii)]{Alb-Kha-Qua-14-jsp}, though the proof there is light on details. For completeness, we include a more detailed proof here.
\begin{lemma}\label{lem:ChaKol}
There exists an event $\Omega_0$ with $\bbP(\Omega_0)=1$ so that for all $(s,y,t,x,\beta)\in \varsets\times\R$ and all $r \in (s,t)$,
\begin{align}
\Sheb(t,x\viiva s,y) &= \int_{\bbR} \Sheb(t,x|r,z)\Sheb(r,z\viiva s,y)dz.\label{eq:C-K}
\end{align}
\end{lemma}
\begin{proof}
First, fix $s,y \in \bbR$, $\beta \in\R$, and $r>s$. Recall that the $\sC((s,\infty),\bbR)$-valued random variable $\Sheb(\aabullet,\aabullet\viiva s,y)$ is the unique mild solution to \eqref{eq:Greens} coming from Lemma \ref{lem:fixGreen}. Call $f(z) = \Sheb(r,z | s,y)$ and notice that Lemma \ref{lem:mombd} implies that \eqref{eq:BGcond} holds for $f$. It then follows from Lemma \ref{lem:uniq} that for $t>r$ and $x \in \bbR$,
\begin{align*}
\Sheb(t,x\viiva r;f) &= \int_{\bbR} \Sheb(t,x|r,z)\Sheb(r,z\viiva s,y)dz
\end{align*}
is the unique $\sC((r,\infty),\R)$-valued and adapted solution to the mild equation
\begin{align*}
U(t,x) &= \int_{\bbR} \heat(t-r,x-z)f(z)dz + \int_{\bbR} \int_r^t \heat(t-v,x-w)U(v,w)W(dv\,dw)
\end{align*}
satisfying the moment hypothesis in \eqref{eq:BGcond}. Recalling the mild formulation \eqref{eq:Greens}, we also have for all $z \in \R$,
\begin{align*}
f(z) &= \Sheb(r,z\viiva s,y) = \heat(r-s,z-y) +  \int_{\R}\int_s^r \heat(r-v,z-w)\Sheb(v,w\viiva s,y)W(dv\,dw)
\end{align*}
The moment estimates in  Lemma \ref{lem:mombd} imply that we may use the stochastic Fubini theorem, see  \cite[Theorem 4.33]{DaP-Zab-14} or \cite[Theorem 2.6]{Wal-86}, to write
\begin{align*}
\Sheb(t,x\viiva r;f) &= \int_{\bbR} \heat(t-r,x-z)f(z)dz + \int_{\bbR} \int_r^t \heat(t-v,x-w)\Sheb(v,w|r;f)W(dv\,dw)\\
&= \int_{\bbR} \heat(t-r,x-z)\heat(r-s,z-y)dz\\
&+\int_{\R} \heat(t-r,x-z) \int_{\R}\int_s^r \heat(r-v,z-w)\Sheb(v,w\viiva s,y)W(dv\,dw)dz \\
&+ \int_{\bbR} \int_r^t \heat(t-v,x-w)\Sheb(v,w|r;f)W(dv\,dw)\\
&= \heat(t-s,x-y) +\int_{\R}\int_s^r \heat(t-v,x-w)\Sheb(v,w\viiva s,y)W(dv\,dw) \\
&+ \int_{\bbR} \int_r^t \heat(t-v,x-w)\Sheb(v,w|r;f)W(dv\,dw).
\end{align*}
We then have for all $t>r$ and $x \in \bbR$,
\begin{align*}
&\Sheb(t,x\viiva r;f) - \Sheb(t,x\viiva s,y) \\
&\qquad= \int_{\R} \int_r^t \heat(t-v,x-w)(\Sheb(v,w\viiva r;f) - \Sheb(v,w \viiva s,y))W(dv\,dw).
\end{align*}
The Burkholder-Davis-Gundy and Gronwall inequalities now imply that
\[\Sheb(\aabullet ,\aabullet \viiva r;f) - \Sheb(\aabullet ,\aabullet\viiva s,y)= 0.\]
\eqref{eq:C-K} then holds for all $(s,x,t,y,\beta) \in \varsets\times\bbR \cap \bbD^5$ and $r \in \bbD$ with $r \in(s,t)$ on a single set of full probability. Continuity of the left hand side of the expression in \eqref{eq:C-K} and the integrand of the right hand side, combined with the growth estimates in Corollary \ref{cor:growth} and the dominated convergence theorem, imply that the result holds simultaneously for all $(s,x,t,y,\beta) \in \varsets\times\bbR$ and $r\in (s,t)$.
\end{proof}

\begin{proof}[Proof of Theorem \ref{prop:IC}]
On the complement of the event in equation \eqref{eq:cor:growth} of Corollary \ref{cor:growth}, for each $T,B>0$, there exists a constant $C=C(T,B,\omega)$ so that for all $s<t$ in $[-T,T]$, all $\beta \in [-B,B]$, and all $x,y \in \bbR$, $C^{-1}(1+ |x|^4+|y|^4)^{-1} \leq \rnSheb(t,x\viiva s,y) \leq C(1+ |x|^4+|y|^4)$. Then, for $\mu,\zeta \in \ICMM$, we have
\begin{align*}
&0 < C^{-1} \int_{\bbR}\int_{\bbR}(1+|z|^4+|w|^4)^{-1} \heat(t-s,z-w)\zeta(dz)\mu(dw) \\
&\leq \int_{\bbR}\int_{\bbR} \Sheb(t,z|s,w)\zeta(dz)\mu(dw) \leq  C \int_{\bbR}\int_{\bbR}(1+|z|^4+|w|^4)\heat(t-s,z-w) \zeta(dz)\mu(dw) <\infty,
\end{align*}
by the conditions defining $\ICMM(4)$. This implies \eqref{prop:IC:1}. 

We now turn to the first case of \eqref{prop:IC:ICMpres}. Fix $a>0$ and  $s,t \in \bbR$ with $s<t$. Call $b= \frac{1}{2(t-s)}$. Again, by Corollary \ref{cor:growth}, there exist $C,C'>0$ depending on $s,t,B,\omega$ so that whenever $\beta \in [-B,B]$,
\begin{align*}
&\int_{\bbR} e^{-a x^2}\Sheb(t,x\viiva s;\mu)dx = \int_{\bbR}\int_{\bbR} e^{-a x^2}\heat(t-s,x-y)\rnSheb(t,x\viiva s;y)\mu(dy)dx \\
&\leq C \int_{\bbR}\int_{\bbR} e^{-a x^2}e^{-b (x-y)^2 }(1+|x|^4+|y|^4)dx\mu(dy) \\
&\leq C' \int_{\bbR} \int_{\bbR }e^{-\frac{a}{2}x^2}e^{-b(x-y)^2}dx\mu(dy)  +  C' \int_{\bbR} \int_{\bbR}e^{-a x^2}e^{- b (x-y)^2}(1+|y|^4)dx\mu(dy).
\end{align*}
The inner $dx$ integrals can now be computed and result in Gaussian density functions in $y$, up to normalization. Therefore, both terms are finite by the condition that $\mu \in \ICM$. The remaining case of \eqref{prop:IC:ICMpres} is similar.

To prove \eqref{prop:IC:ICMsharp}, suppose that for some $K>0$, $\zeta[-K,K] >0$ and $\mu \in \sM_+(\bbR)$. Call $A>0$ the supremum in the statement of \eqref{prop:IC:ICMsharp} and, without loss of generality assume that $A$ is finite and consider the case that for all $a<A$, $\int e^{-a y^2}\mu(dy) = \infty$. We have
\begin{align*}
\int_{\bbR}\int_{\bbR} \Sheb(t,z|s,w)\zeta(dz)\mu(dw) \geq C\zeta([-K,K]) \int_{\bbR}\frac{\min_{z\in[-K,K]}\heat(t-s,z-w)}{1+|K|^4+|w|^4} \mu(dw).
\end{align*}
Limit comparison now shows that if $\frac{1}{2(t-s)}<A$, then the above integral will be infinite, implying \eqref{prop:IC:ICMsharp}.

Next, we turn to part \eqref{prop:IC:Holder}. Fix $K,B,T>0$ and $\epsilon \in (0,T/2)$ and restrict attention to $s < t-\epsilon$, $x \in [-K,K]$, and $\beta \in [-B,B]$. Then we have
\begin{align*}
\heat(t-s,x-y)\rnSheb(t,x\viiva s,y) &\leq C \frac{1}{\sqrt{2\pi(t-s)}}e^{-\frac{(x-y)^2}{2(t-s)}} (1+|x|^4+|y|^4) \leq C' e^{-  \frac{y^2}{8T}}
\end{align*}
for some $C'=C'(K,T,B,\epsilon,\omega)$. Pointwise continuity on $\varsetth\times\bbR = \{(s,t,x,\beta) \in \bbR^4 : s<t\}$ follows from the previous estimate and the dominated convergence theorem. Now, we turn to controlling the H\"older semi-norms. It follows from \eqref{eq:cor:Hgrowth} that for each $T,B>1$,  $\alpha \in (0,1/4)$, $\gamma \in (0,1/2)$, and $\eta \in (0,1)$, there exists $C=C(T,B,\alpha,\gamma,\eta,\omega)>0$ so that for all $K>1$,
\begin{align*}
|\rnShe_{\aabullet}(\aabullet,\aabullet|\aabullet,\aabullet)|_{\sC^{\alpha,\gamma,\eta}(\varsetsg{T}{K}{1/K}\times[-B,B])} \leq C(1+K^7).
\end{align*}
To see H\"older regularity, take any $\delta\in (0,1)$ and let $M>1$ be sufficiently large that $1/M<\delta$. Take $(t_1,x_1,s_1,\beta_1), (t_2,x_2,s_2,\beta_2) \in \varsetthg{T}{M}{\delta}\times[-B,B]$ with $(t_1,x_1,s_1,\beta_1) \neq (t_2,x_2,s_2,\beta_2)$. We have
\begin{align*}
&\frac{|\She_{\beta_1}(t_1,x_1|s_1;\mu) - \She_{\beta_2}(t_2,x_2|s_2;\mu)|}{|t_2-t_2|^{\alpha}+ |s_2-s_1|^{\alpha} + |x_2-x_2|^{\gamma} +|\beta_2-\beta_1|^\eta}\\ 
&\leq \frac{\int_{\bbR}|\She_{\beta_1}(t_1,x_1|s_1,y) - \She_{\beta_2}(t_2,x_2|s_2,y)|\mu(dy)}{|t_2-t_1|^{\alpha}+ |s_2-s_1|^{\alpha} + |x_2-x_1|^{\gamma} + |\beta_2-\beta_1|^\eta} \\
&\leq \frac{\int_{\bbR}|\rnShe_{\beta_1}(t_1,x_1|s_1,y) - \rnShe_{\beta_2}(t_2,x_2|s_2,y)|\heat(t_2-s_2,x_2-y)\mu(dy)}{|t_2-t_1|^{\alpha}+ |s_2-s_1|^{\alpha} + |x_2-x_1|^{\gamma} + |\beta_2-\beta_1|^\eta}\\
&\qquad+ \frac{\int_{\bbR}|\heat(t_1-s_1,x_1-y) - \heat(t_2-s_2,x_2-y)|\rnShe_{\beta_1}(t_1,x_1|s_1,y) \mu(dy)}{|t_2-t_1|^{\alpha}+ |s_2-s_1|^{\alpha} + |x_2-x_1|^{\gamma} }.
\end{align*}
 Then, by Corollary \ref{cor:growth}, there exists $C=C(T,B,M,\alpha,\gamma,\eta,\omega)$ so that
\begin{align*}
&\int_{\bbR}\frac{|\rnShe_{\beta_1}(t_1-s_1,x_1-y) - \rnShe_{\beta_2}(t_2-s_2,x_2-y)|}{|t_2-t_1|^{\alpha}+ |s_2-s_1|^{\alpha} + |x_2-x_1|^{\gamma} + |\beta_2-\beta_1|^\eta}\heat(t_2-s_2,x_2-y)\mu(dy) \\
&\leq |\rnShe_{\aabullet}(\aabullet,\aabullet|\aabullet,\aabullet)|_{\sC^{\alpha,\gamma,\eta}(\varsetsg{T}{M}{\delta}\times[-B,B])} \int_{-M}^M \max_{(s,t,x) \in \varsetthg{T}{M}{\delta}} \heat(t-s,x-y)\mu(dy) \\
&\qquad+ \int_{\bbR\backslash[-M,M]}  |\rnShe_{\aabullet}(\aabullet,\aabullet|\aabullet,\aabullet)|_{\sC^{\alpha,\gamma,\eta}(\varsetsg{T}{y}{1/y}\times[-B,B])}\max_{(s,t,x) \in \varsetthg{T}{M}{\delta}} \heat(t-s,x-y)\mu(dy) \\
&\leq C\bigg( \int_{-M}^M \max_{(s,t,x) \in \varsetthg{T}{M}{\delta}} \heat(t-s,x-y)\mu(dy) \\
&\qquad+  \int_{\bbR\backslash[-M,M]}(1+|y|^7) \max_{(s,t,x) \in \varsetthg{T}{M}{\delta}} \heat(t-s,x-y)\mu(dy)\bigg) < \infty.
\end{align*}
Recall that
\begin{align}
\partial_t \heat(t,x) &= \bigg(\frac{x^2}{2t^2} - \frac{1}{2t} \bigg)\heat(t,x) \label{eq:heatdiff}
\end{align}
and $1/M < \delta$. By Corollary \ref{cor:growth} and the previous observation,  there is a constant $C=C(T,B,M,\omega)$ so that we also have
\begin{align*}
&\int_{\bbR}\frac{|\heat(t_1-s_1,x_1-y) - \heat(t_2-s_2,x_2-y)|\rnShe_{\beta_1}(t_1,x_1|s_1,y) }{|t_2-t_2|^{\alpha}+ |s_2-s_1|^{\alpha} + |x_2-x_2|^{\gamma}}\mu(dy) \\
&\leq C\int_{\bbR}\max_{(s,t,x)\in \varsetsg{T}{M}{\delta}} \heat(t-s,x-y) (1+|y|^4)\mu(dy) < \infty.
\end{align*}

Part \eqref{prop:IC:semi-group} follows from Lemma \ref{lem:ChaKol}, the definition of a physical solution via superposition in \eqref{eq:superpos}, and the Fubini-Tonelli theorem \cite[Theorem 8.8(a)]{Rud-87}.

Turning to \eqref{prop:IC:fcont}, let $B,K,T>0$ be as in the statement. Let  $\epsilon>0$ and $r \geq 4K$ be arbitrary and let $\varphi \in \sC_c(\R,[0,1])$ satisfy $\varphi(x) = 1$ on $[-r,r]$ and $\varphi(x) = 0$ on $\bbR\backslash [-2r,2r]$. Take any $\delta\in (0,1/2)$, any $g\in \CICM$ with $d_{\CICM}(f,g)<\delta$, any $x \in [-K,K]$, any $\beta \in [-B,B]$, and any $s< t$ in $[-T,T]$ with $t-s<\delta$. Use the triangle inequality to write
\begin{align}
&\bigg|\int_{\R} \Sheb(t,x\viiva s,y)g(y)dy -f(x)\bigg|  \notag\\
&\leq \int_{\R}\heat(t-s,x-y)g(y)|\rnSheb(t,x\viiva s,y)-1| \varphi(y)dy \label{eq:fcont-1} \\
&\qquad+ \int_{\R}\heat(t-s,x-y)g(y)|\rnSheb(t,x\viiva s,y)-1| (1-\varphi(y))dy \label{eq:fcont-2} \\
&\qquad + \int_{\R}\heat(t-s,x-y)|g(y)-f(y)|\varphi(y)dy \label{eq:fcont-3}\\
&\qquad + \int_{\R}\heat(t-s,x-y)|g(y)-f(y)|(1-\varphi(y))dy \label{eq:fcont-4}\\
&\qquad + \bigg|\int_{\R}\heat(t-s,x-y)f(y)\varphi(y)dy - f(x)\bigg| \label{eq:fcont-5} \\
&\qquad + \int_{\R}\heat(t-s,x-y)f(y)(1-\varphi(y))dy. \label{eq:fcont-6}
\end{align}
From the definition of $d_{\CICM}$, \eqref{eq:CICMm}, there exists $C=C(r,f)>0$ so that for all $\delta\in(0,1/2)$ and all $g \in \CICM$ with $d_{\CICM}(f,g)<\delta$,
\begin{align}
\sup_{-2r \leq x \leq 2r}|f(x)-g(x)|\leq C \delta, \quad \text{ and } \max_{-2r \leq x \leq 2r}g(x) \leq C. \label{eq:fcont-7}
\end{align}

We begin with the expression in \eqref{eq:fcont-1}. For each $r$ as above, there is a $\delta_0=\delta_0(r,K,\epsilon) \in(0,1/2)$ so that whenever $\delta<\delta_0$, $|\rnSheb(t,x\viiva s,y)-1| <\epsilon$ for all $x \in [-K,K]$, $y \in [-2r,2r]$, and $s \leq t$ in $[-T,T]$ with $t-s \in [0,\delta)$. Bounding $g$ by the constant $C$ from \eqref{eq:fcont-7} on the support of $\varphi$, then bounding $\varphi$ by $1$ and using the fact that the heat kernel integrates to $1$, it follows that applying
\begin{align}
\varlimsup_{r\to\infty}\varlimsup_{\epsilon\searrow 0}\varlimsup_{\delta \searrow0}\, \sup_{\substack{y\in[-K,K],\beta\in[-B,B] \\ s,t\in[-T,T], t-s\in(0,\delta ) \\ d_{\CICM}(f,g)<\delta}}\, \label{eq:fcont-8}
\end{align}
to the expression in \eqref{eq:fcont-1} results in a value of $0$. A similar argument controls the expression in \eqref{eq:fcont-3}: on the support of $\varphi$, use \eqref{eq:fcont-7} to bound $|g(y)-f(y)|$ by $C\delta$, then bound $\varphi$ by $1$ and use the fact that the heat kernel integrates to $1$ to see that applying \eqref{eq:fcont-8} to \eqref{eq:fcont-3} results in a value of $0$. Turning to \eqref{eq:fcont-5}, notice that $ f(x) = \varphi(x) f(x)$ for $x \in [-K,K]$. Since $\varphi(x)f(x)$ is compactly supported and continuous, convergence of \eqref{eq:fcont-5} to $0$ after applying \eqref{eq:fcont-8} is a standard fact about the heat kernel. 

By Corollary \ref{cor:growth}, there is a constant $C'=C'(\omega,B,T,K)>0$ so that for all $x \in [-K,K],y\in\R$, and $s\leq t$ in $[-T,T]$, $|\rnSheb(t,x\viiva s,y)-1| \leq C'(1+|y|^4)$. Then we may bound the sum of the terms in \eqref{eq:fcont-2}, \eqref{eq:fcont-4}, and \eqref{eq:fcont-6} by
\[
\int_{\R}\heat(t-s,x-y)\ind_{[-r,r]^c}(y)[C'g(y)(1+|y|^4) + g(y) + 2f(y)]dy
\]
For $|y| \geq 4K$ and $x \in [-K,K]$, we have \[-2(x-y)^2 \leq -2y^2 + 4 |y||x| - 2 x^2 \leq -2y^2 + 4K |y| \leq -y^2.\]

For $t-s < \delta < 1/8$, notice that $\frac{1}{2(t-s)}-2 \geq \frac{1}{4(t-s)}$. Since $r \geq 4K$, we have $|x-y|>K$ if $y \in [-r,r]^c$. For such values of $s,t,x,y$, this leads to
\[
\heat(t-s,x-y) = \frac{1}{\sqrt{2\pi(t-s)}}e^{-(x-y)^2[\frac{1}{2(t-s)}-2]}e^{-2(x-y)^2} \leq  \frac{1}{\sqrt{2\pi(t-s)}}e^{-\frac{K^2}{4(t-s)}} e^{-y^2}.
\]
It follows from the definition of $d_{\CICM}$ that there is a constant $C''=C''(f)$ so that if $d_{\CICM}(f,g)<1/8$, then
\[
\int_{\R} e^{-y^2}(1+|y|^4) g(y)dy < C''.
\]
Then there is $C'''=C'''(\omega,f,B,T,K)>0$ so that for all $g\in\CICM$ with $d_{\CICM}(f,g)<1/8$,  all $x \in [-K,K]$, all $s< t$ in $[-T,T]$ with $t-s<1/8$, and all $r \geq 4K$,
\[
\int_{\R}\heat(t-s,x-y)\ind_{[-r,r]^c}(y)[C'g(y)(1+|y|^4) + g(y) + 2f(y)]dy \leq \frac{C'''}{\sqrt{2\pi(t-s)}}e^{-\frac{K^2}{4(t-s)}}.
\]
It follows that applying \eqref{eq:fcont-8} to  the sum of the terms in \eqref{eq:fcont-2}, \eqref{eq:fcont-4}, and \eqref{eq:fcont-6} also results in a value of zero. The remaining case of \eqref{prop:IC:fcont} is similar.

Next, we turn to part \eqref{prop:IC:musat}, with the proof being similar to that of part \eqref{prop:IC:fcont}. Let $B,T>0$ be as in the statement. Let  $\epsilon>0$ and $r \geq 1$ be arbitrary and let $\varphi \in \sC_c(\R,[0,1])$ satisfy $\varphi(x) = 1$ on $[-r,r]$ and $\varphi(x) = 0$ on $\bbR\backslash [-2r,2r]$. Take any $\delta\in (0,1/2)$, any $\zeta\in \ICM$ with $d_{\ICM}(\zeta,\mu)<\delta$, any $\beta \in [-B,B]$, and any $s< t$ in $[-T,T]$ with $t-s<\delta$. Use the triangle inequality to write
\begin{align}
&\bigg|\int_{\R}\int_{\R}f(x) \Sheb(t,x\viiva s,y)dx\zeta(dy) -\int_{\R}f(y)\mu(dy)\bigg| \notag\\
&\qquad \leq\int_{\R} \bigg|\int_{\R}f(x) \Sheb(t,x\viiva s,y)dx - f(y)\bigg| \varphi(y)\zeta(dy) \label{eq:musat-1}\\
&\qquad\qquad +\int_{\R}\int_{\R}f(x) \Sheb(t,x\viiva s,y)dx (1-\varphi(y))\zeta(dy)+ \int_{\R} f(y)(1-\varphi(y))\zeta(dy) \label{eq:musat-2} \\
&\qquad\qquad + \bigg|\int_{\R} f(y)\zeta(dy) - \int_{\R}f(y)\mu(dy) \bigg|. \label{eq:musat-3}
\end{align}
The hypothesis implies that $f \in \CICM$ so by part \eqref{prop:IC:fcont}, there exists $\delta_0=\delta_0(\w,r,B,T)$ so that for all $\beta \in [-B,B]$, all $s < t$ in $[-T,T]$ with $t-s<\delta_0$, and all $y \in [-2r,2r]$,
\[
\bigg|\int_{\R}f(x) \Sheb(t,x\viiva s,y)dx - f(y)\bigg| < \epsilon.
\]
There exists $C=C(\mu,r)>0$ so that whenever $d_{\ICM}(\zeta,\mu)<1/2$, $\int_{\R}\varphi(y)\zeta(dy) \leq C$. It follows that applying
\begin{align}
\varlimsup_{r\to\infty}\varlimsup_{\epsilon\searrow0}\varlimsup_{\delta\searrow 0}  \sup_{\substack{\beta\in[-B,B], s,t\in[-T,T] \\d_{\ICM}(\mu,\zeta)<\delta, t-s\in(0,\delta )}} \label{eq:musat-4}
\end{align}
to the expression in \eqref{eq:musat-1} results in a value of $0$. 

Recall that we may bound $\heat(2/a,x)(1+|x|^4)$ by a constant depending only on $a$ times $\heat(1/a,x)$ for all $x \in \R$. By Corollary \ref{cor:growth} and the hypothesis on $f$ there exist $C',C'',C'''>0$, all depending on $\w,a,A,B,K$ and $T$, so that whenever $t-s < 1/2$, the expression in \eqref{eq:musat-2} is bounded by
\begin{align*}
&C'\int_{\R\backslash[-r,r]^c}\bigg(\int_{\R}\heat(2/a,x)\heat(t-s,y-x)(1+|x|^4+|y|^4)dx + e^{-a y^2}\bigg) \zeta(dy)\\
&\qquad \leq C'' \int_{\R\backslash[-r,r]^c} \heat(1/a+t-s,y) + \heat(1/a,y) + e^{-ay^2}\bigg)\zeta(dy) \leq C''' \int_{\R\backslash[-r,r]^c} e^{-\frac{a}{4}y^2}\zeta(dy) \\
&\qquad \leq C''' e^{-\frac{a}{8}r^2} \int_{\R} e^{-\frac{a}{8}y^2}\zeta(dy).
\end{align*}
From the definition of $d_{\ICM}$,  there exists $C''''=C''''(\mu,a)$ so that whenever $d_{\ICM}(\zeta,\mu)<1/2$,  $\int_{\R} e^{-\frac{a}{8}y^2}\zeta(dy) \leq C''''$. It follows that applying \eqref{eq:musat-4} to the expression in \eqref{eq:musat-2} results in a value of zero. Sending $d_{\ICM}(\zeta,\mu) \searrow 0$ sends the expression in \eqref{eq:musat-3} to zero, directly from the definition of $d_{\ICM}$ (the topology is generated by test functions satisfying the hypotheses satisfied by $f$). The remaining case of \eqref{prop:IC:musat} is similar and so the result follows.
\end{proof}

Next, we turn to the proof of the conservation of asymptotic slope, Proposition \ref{prop:cons}.

\begin{proof}[Proposition \ref{prop:cons}]
The condition defining $H(\lambda_-,\lambda_+)$ implies that $f(z)dz \in \ICM$. Local boundedness  of $\Sheb(t,\aabullet | s,f)$ follows from the continuity in Theorem \ref{prop:IC}.  Fix $T,B>0$ and restrict attention to $-T \leq s \leq t \leq T$ and $-B \leq \beta \leq B$. By Corollary \ref{cor:growth}, there exists $C=C(T,B,\omega)>0$  so that $C^{-1} (1+|x|^4+|z|^4)^{-1} \leq \rnSheb(t,x\viiva s,z) \leq C(1+|x|^4+|z|^4)$ for all $z,x \in \bbR$. Fix $\epsilon>0$ and $f \in H(\lambda_-,\lambda_+)$. By hypothesis, there exist $c',C'>0$ so that the following hold:
\begin{align*}
c'\bigg[\ind_{(-\infty,0]}(x)e^{(\lambda_-+\epsilon)x} + \ind_{(0,\infty)}(x) e^{(\lambda_+-\epsilon)x}\bigg] &\leq f(x) \leq C'\bigg[\ind_{(-\infty,0]}(x)e^{(\lambda_--\epsilon)x} + \ind_{(0,\infty)}(x) e^{(\lambda_++\epsilon)x}\bigg]\\
c' e^{-\epsilon(|x|+|z|)}&\leq  (1+|x|^4+|z|^4)^{-1} \leq 1, \qquad \text{ and }\\
1&\leq  (1+|x|^4+|z|^4) \leq C' e^{\epsilon(|x|+|z|)}.
\end{align*}
We prove one inequality, with the others being similar. We have
\begin{align*}
&\varliminf_{x\to\infty}\frac{1}{x}\log \int_{\bbR} \heat(t-s,x-z) \rnSheb(t,x\viiva s,z)f(z)dz  \\
&\qquad\geq \varliminf_{x\to\infty}\frac{1}{x}\log \int_{\bbR} \heat(t-s,x-z) (1+|x|^4+|z|^4)^{-1}  f(z)dz \\
&\qquad\geq \varliminf_{\epsilon \searrow 0} \varliminf_{x\to\infty}\frac{1}{x}\log \int_{0}^\infty \heat(t-s,x-z) e^{-\epsilon z}  e^{(\lambda_+- \epsilon)z}dz \\
&\qquad = \varliminf_{\epsilon \searrow 0} \varliminf_{x\to\infty}\frac{1}{x}\log \int_{\bbR} \heat(t-s,x-z) e^{(\lambda_+- 2\epsilon)z}dz = \lambda_+.
\end{align*}
To explain the last two steps, let $X$ be a Normal$(0,t-s)$ random variable on $(\Omega,\sF,\bbP)$. Then
\begin{align*}
\int_{-\infty}^0\heat(t-s,x-z)e^{(\lambda_+ -2\epsilon)z}dz &=   \int_{x}^\infty\heat(t-s,y)  e^{(\lambda_+ - 2\epsilon)(x-y)}dy =  \bbE\bigg[e^{(\lambda_+-2\epsilon)(x-X)}\ind_{\{X>x\}}\bigg],
\end{align*}
which has Gaussian decay as $x \to \infty$ and so is negligible compared to the full-space integral of the same function. The form of the Gaussian moment generating function then completes the proof.
\end{proof}

We defer the proof of Theorem \ref{thm:jcont} to Section \ref{sec:reg}, where we prove it along with Theorem \ref{thm:polyreg}.

\section{Continuum directed polymers}\label{sec:poly}
We initially view the measures $\Polyb_{(s,y),(t,x)}$ defined in \eqref{eq:p2pdef} as measures on the product space $(\bbR^{[s,t]},\sB(\bbR^{[s,t]}))$ and then show that each induces a unique measure on $(\sC_{[s,t]},\sB(\sC_{[s,t]}))$. That this is true for all $(s,y,t,x,\beta) \in \varsets\times\bbR = \{(s,y,t,x,\beta) \in \bbR^5 : s < t\}$ is the content of our first main result in this section, Theorem \ref{prop:polyexist}. This was previously shown for fixed $s,y,t,x,\beta \in \bbR$ on a event of full probability depending on all of these parameters using a different argument in \cite{Alb-Kha-Qua-14-jsp}. Before turning to our proof of Theorem \ref{prop:polyexist}, we begin with some preliminary observations.

For $s < t$ we define a random variable $\widetilde{X}$ on $\sC([s,t],\bbR)$, by
\begin{align*}
\widetilde{X}_u &= X_u - \bigg(\frac{t-u}{t-s} X_s + \frac{u-s}{t-s} X_t\bigg),
\end{align*}
where $X$ is the coordinate random variable. Note that the definition of $\widetilde{X}$ depends implicitly on $s$ and $t$, but we suppress this dependence. Denote by $\bfP_{(s,y),(t,x)}^{BB}$ and $\oE_{(s,y),(t,x)}^{BB}$ the law and expectation associated to a Brownian bridge from $(s,y)$ to $(t,x)$. We start this section with an easy lemma recording some well-known and basic properties of Brownian bridge.
\begin{lemma}\label{lem:BBest}
For $ -\infty < s < a < b < t < \infty$, $x,y \in \bbR$, and $p>1$, there exists $C=C(p)$ so that
\begin{align*}
\oE_{(s,y),(t,x)}^{BB}[|X_a - X_b|^p] &\leq C \bigg(|b-a|^{p/2} + \frac{|b-a|^p}{|t-s|^p}(|t-s|^{p/2} + |x-y|^p)\bigg), \\
\oE_{(s,y),(t,x)}^{BB}[\max_{s \leq a \leq t}|X_a|^p]^{1/p} &\leq C\bigg(|t-s|^{1/2}+ |y| + |x-y|\bigg).
\end{align*}
Moreover, the distribution of $\tilde{X}$ under $\bfP_{(s,y),(t,x)}^{BB}(\aabullet)$ is the same as the distribution of $X$ under $\bfP_{(s,0),(t,0)}^{BB}(\aabullet)$.
\end{lemma}
\begin{proof}
Let $\bfP^{BM}$ denote the two-sided Wiener measure $(\sC,\sB(\sC))$, i.e., the law of two-sided standard Brownian Motion, and let $\oE^{BM}$ denote the corresponding expectation. Computation of covariances shows that on the interval $u \in [s,t]$, under $\bfP^{BM}$,
\begin{align*}
u \mapsto X_{u-s} - \frac{u-s}{t-s}X_{t-s} + y + \frac{u-s}{t-s}(x-y)
\end{align*}
is a Brownian bridge between $(s,y)$ and $(t,x)$. In particular, this representation also shows that $\oE_{(s,y)(t,x)}^{BB}[F(\widetilde{X})] =\oE_{(s,0)(t,0)}^{BB}[F(X)]$ for all $F \in \sB_b(\sC_{[s,t]})$. 
\begin{align*}
\oE_{(s,y),(t,x)}^{BB}[|X_b - X_a|^p]^{1/p} &\leq \oE^{BM}[|X_b - X_a|^p]^{1/p} + \frac{|b-a|}{|t-s|}(\oE^{BM}[|X_{t-s}|^p]^{1/p} + |x-y|) \\
&\leq C_p( |b-a|^{1/2} + \frac{|b-a|}{|t-s|}(|t-s|^{1/2} + |x-y|)),
\end{align*}
for some $C_p>0$ by standard estimates for Normal random variables. A similar argument and the reflection principle gives for $s < a < t,$
\begin{align*}
\oE_{(s,y),(t,x)}^{BB}[\max_{s \leq a\leq t}|X_a|^p]^{1/p} &\leq C_p\bigg(|a-s|^{1/2} + \frac{|a-s|}{|t-s|}|t-s|^{1/2} + |y| + \frac{|a-s|}{|t-s|}|x-y|\bigg).
\end{align*}
Recalling that $|a-s|\leq |t-s|$, the result follows.
\end{proof}

For  $s < t_1 < \dots < t_k < t$, we denote by $\Polyb_{(s,y),(t,x)}\big|_{t_1,\dots,t_k}$ the distribution of $(X_{t_1},\dots,X_{t_k})$ under $\Polyb_{(s,y),(t,x)}$ defined in \eqref{eq:p2pdef} and by $\bfP_{(s,y),(t,x)}\big|_{t_1,\dots,t_k}$ the law of $(X_{t_1},\dots,X_{t_k})$ under $\bfP_{(s,y),(t,x)}$. We have the following bound on Radon-Nikodym derivatives which will play a key role in most of what follows. 
\begin{lemma}\label{lem:rndbd}
There exists an event $\Omega_0$ with $\bbP(\Omega_0) = 1$ so that on $\Omega_0$, the following holds. For each $T,B>1$ and each $k \in \bbN$, there exists $C=C(T,B,k,\omega)$ so whenever $- T \leq t_0 = s < t_1 < \dots < t_k < t= t_{k+1} \leq T$ and $-B \leq \beta \leq B$, 
\begin{align*}
C^{-1} \prod_{i=0}^{k}(1+|X_{t_{i+1}}|^4 + |X_{t_{i}}|^4)^{-1} &\leq \frac{d \Polyb_{(s,y),(t,x)}  \big|_{t_1,\dots,t_k}}{d\bfP_{(s,y),(t,x)}^{BB}\big|_{t_1,\dots,t_k}}(X_{t_1},\dots,X_{t_k}) \\
&\leq C \prod_{i=0}^{k}(1+|X_{t_{i+1}}|^4 + |X_{t_{i}}|^4), \qquad \bfP_{(s,y),(t,x)}^{BB}\text{-a.s.}
\end{align*}
\end{lemma}
\begin{proof}
Observe that by \eqref{eq:p2pdef}, we have $ \bfP_{(s,y),(t,x)}^{BB}$ almost surely,
\begin{align*}
\frac{d \Polyb_{(s,y),(t,x)}  \big|_{t_1,\dots,t_k}}{d\bfP_{(s,y),(t,x)}^{BB}\big|_{t_1,\dots,t_k}}(X_{t_1},\dots,X_{t_k}) &=  
 \frac{\prod_{i=0}^{k}\rnSheb(t_{i+1},X_{t_{i+1}}|t_i, X_{t_i})}{\rnSheb(t_{k+1},X_{t_{k+1}}|t_0,X_{t_0})} .
\end{align*}
The result follows from the growth bounds, \eqref{eq:cor:growth} and \eqref{eq:cor:growthinv}, in Corollary \ref{cor:growth}.
\end{proof}

Next, we turn to the proof of Theorem \ref{prop:polyexist}, which we prove along with the following estimate, which will play a role in the proof of Proposition \ref{prop:KMG} below. 
\begin{lemma}\label{lem:Qbridge}
There exists an event $\Omega_0$ with $\bbP(\Omega_0)=1$ so that on $\Omega_0$, the following holds. 

\begin{enumerate}[label={\rm(\roman*)}, ref={\rm\roman*}]   \itemsep=3pt 
\item\label{prop:polyexist:hold}   For each $T,B>1$, $\eta \in (0,1/2)$ and $\epsilon\in(0,1)$, there exists $C= C(T,B,\eta,\epsilon,\omega)>0$ so that for all $s < t$ in $[-T,T]$ and all $\beta \in [-B,B]$,
\begin{align}
\oE_{(s,y),(t,x)}^{\Polyb} [|X|_{\sC^{\eta}_{[s,t]}}] \leq C (1+|x|^{\frac{4+2\epsilon}{1-2\eta}+16} + |y|^{\frac{4+2\epsilon}{1-2\eta}+16}).\label{eq:PolyHolder}
\end{align}
\item For each $T,B>1$ and each $\eta \in (0,1/2)$, there exists $C= C(T,B,\eta,\omega)>0$, so that for all $s<t$ in $[-T,T]$ and all $\beta \in [-B,B]$, 
\begin{align*}
\oE_{(s,y),(t,x)}^{\Polyb}[|\widetilde{X}|_{\sC^{\eta}_{[s,t]}}] \leq C(1+|x|^{16}+|y|^{16})
\end{align*}
\end{enumerate}
\end{lemma}
\begin{remark}
The bound in \eqref{eq:PolyHolder} may look odd in view of the shift invariance implied by Proposition \ref{prop:cov}, from which one might expect a bound depending only on $|x-y|$. Because we work on a single full probability event for all initial and terminal points simultaneously, likely no such bound is possible: the process $x \mapsto \oE_{(s,x),(t,x)}^{\Polyb}[|X|_{\sC^{\eta}_{[s,t]}}]$ should be expected to mix as $x \to \infty$ and its distribution can be shown to have unbounded support.
\end{remark}

\begin{proof}[Proof of Theorem \ref{prop:polyexist} and Lemma \ref{lem:Qbridge}]
First, notice that Lemma \ref{lem:ChaKol} implies that the measures defined in \eqref{eq:p2pdef} define probabilities. Now, take $T,B>1$ as in the statement and $-T \leq s  < t \leq T$ and $\beta \in [-B,B]$. For $ s \leq a < b \leq t$ and $\gamma>1$, appealing to the identity in \eqref{eq:p2pdef} and Lemma \ref{lem:rndbd}, there exists $C=C(T,B,\omega)$ so that we have
\begin{align*}
&\oE_{(s,y),(t,x)}^{\Polyb}[|X_b-X_a|^\gamma] =  \oE_{(s,y),(t,x)}^{BB}\bigg[|X_b-X_a|^\gamma \frac{\rnSheb(t,x|b,X_b)\rnSheb(b,X_b|a,X_a)\rnSheb(a,X_a\viiva s,y)}{\rnSheb(t,x\viiva s,y)} \bigg] \\
&\qquad\leq C \oE_{(s,y),(t,x)}^{BB}[|X_b-X_a|^\gamma(1+|x|^4+|y|^4+|X_b|^4+|X_a|^4)^4].
\end{align*}
By the Cauchy-Schwarz inequality and Lemma \ref{lem:BBest},
\begin{align}
&\oE_{(s,y),(t,x)}^{\Polyb}[|X_b-X_a|^\gamma] \leq C' \oE_{(s,y),(t,x)}^{BB}[|X_b-X_a|^{2\gamma}]^{1/2}(1+|x|^{16}+|y|^{16}) \notag\\ 
&\qquad\leq C''\bigg(|b-a|^{\gamma/2} + \frac{|b-a|^\gamma}{|t-s|^\gamma}(|t-s|^{\gamma/2} + |x|^\gamma + |y|^\gamma)\bigg) (1+|x|^{16}+|y|^{16}), \notag\\
&\qquad\leq  C'''|b-a|^{\gamma/2} (1 +|x|^{16+\gamma} + |y|^{16+\gamma})\label{eq:Q2pt}
\end{align}
for some $C',C'',C'''$ which depend on $T,B,\gamma,$ and $\omega$. To obtain the last line, it helps to observe that $|b-a|/|t-s| \leq 1$ and $|t-s| \leq 2T$. Similarly, we have
\begin{align*}
&\oE_{(s,y),(t,x)}^{\Polyb}[|\widetilde{X}_b-\widetilde{X}_a|^\gamma] = \oE_{(s,y),(t,x)}^{BB}\bigg[|\widetilde{X}_b-\widetilde{X}_a|^\gamma \frac{\rnSheb(t,x|b,X_b)\rnSheb(b,X_b|a,X_a)\rnSheb(a,X_a\viiva s,y)}{\rnSheb(t,x\viiva s,y)} \bigg]  \\
&\leq C\oE_{(s,y),(t,x)}^{BB}\bigg[|\widetilde{X}_b-\widetilde{X}_a|^\gamma(1+|x|^4+|y|^4+|X_b|^4+|X_a|^4)^4  \bigg] \\
&\leq C' \oE_{(s,y),(t,x)}^{BB}[|\widetilde{X}_b-\widetilde{X}_a|^{2\gamma}]^{1/2}(1+|x|^{16}+|y|^{16})  \leq C''|b-a|^{\gamma/2}(1+|x|^{16}+|y|^{16})
\end{align*}
for some $C,C',C''$ depending on $T,B,\gamma$ and $\omega$, where in the last step we appeal to the distributional identity in Lemma \ref{lem:BBest}.


The existence of a unique measure on $(\sC_{[s,t]}, \sB(\sC_{[s,t]}))$ with finite dimensional marginal distributions given by \eqref{eq:p2pdef} follows from the Kolmogorov-Chentsov theorem, recorded as Theorem \ref{thm:KC}, \eqref{eq:Q2pt}, and a standard measure theoretic argument as in \cite[Theorem 2.1.6]{Str-Var-97}. Now viewing $\Polyb_{(s,y),(t,x)}$ as this measure on $(\sC_{[s,t]}, \sB(\sC_{[s,t]}))$, the estimate in \eqref{eq:Q2pt}, with say $q=2$, combined with Kolmogorov-Chentsov where we choose $\gamma = (4+2\epsilon)/(1-2\eta)$, which satisfies $\eta = 1/2 - (2+\epsilon)/\gamma \in (0,1/2 - 2/\gamma)$, implies that \eqref{eq:PolyHolder} holds. Lemma \ref{lem:Qbridge} follows similarly.

To show \eqref{eq:p2pIC}, by path continuity, it suffices to show that
\begin{align*}
\lim_{r \searrow s}\oE^{\Polyb}_{(s,y),(t,x)}[|X_r-y|]  = 0 \qquad \text{ and } \qquad \lim_{r \nearrow t}\oE^{\Polyb}_{(s,y),(t,x)}[|X_r-x|] = 0.
\end{align*}
As above, we have
\begin{align*}
\oE^{\Polyb}_{(s,y),(t,x)}[|X_r-y|] &= \oE_{(s,y),(t,x)}^{BB}\bigg[\big|X_r-x\big|\frac{\rnSheb(t,x|r,X_r) \rnSheb(r,X_r\viiva s,y)}{\rnSheb(t,x\viiva s,y)}  \bigg] \\
&\leq \frac{C}{\rnSheb(t,x\viiva s,y)}(1+|x|^2+|y|^2)  \oE_{(s,y),(t,x)}^{BB}\big[\big|X_r-x\big|\big] \to 0,
\end{align*}
as $r \searrow s$, by dominated convergence (dominating by $\|X_{\aabullet}\|_\infty + |x|$) and path continuity of Brownian bridge.

We now turn to \eqref{prop:polyexist:Markov}. For $f_{0:n+1} \in \sC_c(\R,\R)^{n+2}$ and $s=t_0 < t_1 < \dots t_n < t_{n+1}=t$, and calling $z_0 = y$ and $z_{n+1}=x$, whenever $1\leq i<j\leq n$, multiplying and dividing by $\Sheb(t_j,z_j|t_i,z_i)$ inside the integral, we have
\begin{align*}
&\oE_{(s,y),(t,x)}^{\Polyb}\bigg[\prod_{k=0}^{n+1} f_k(X_{t_k})\bigg] = \frac{\int\prod_{k=0}^{n+1} f_k(z_k)\prod_{i=0}^{n}\Sheb(t_{k+1},z_{k+1}|t_k,z_k) dz_{1:n}}{\Sheb(t,x\viiva s,y)}\\
&= \frac{f_0(x) f_{k+1}(y)}{\Sheb(t,x\viiva s,y)} \int\prod_{k=0}^{i-1}f(z_k)\Sheb(t_{k+1}, z_{k+1}|t_k,z_k) \oE_{(t_i,z_i),(t_j,z_j)}\bigg[\prod_{k=i}^j f(X_{t_k})\bigg]\times \\
&\qquad\qquad\qquad\Sheb(t_j,z_j|t_i,z_i) \prod_{k=j+1}^{n}f(z_k)\Sheb(t_{k+1}, z_{k+1}|t_k,z_k) dz_{1:i} dz_{j:n}\\
&= \oE_{(s,y),(t,x)}^{\Polyb}\bigg[\prod_{k=0}^{i-1} f_k(X_{t_k})\oE_{(t_i,X_{t_i}),(t_j,X_{t_j})}\bigg[\prod_{k=i}^j f_k(X_{t_k})\bigg] \prod_{k=j+1}^{n+1}f_k(X_{t_k})\bigg]
\end{align*}
The monotone class theorem \cite[Appendix Theorem 4.3]{Eth-Kur-86} then implies the result.

To show part \eqref{prop:polyexist:cmeas}, note that by \cite[Theorem 8.2]{Bil-99}, \eqref{eq:p2pIC} and \eqref{eq:PolyHolder}  imply that the family $\{\Polyb_{(s,y),(t,x)(\aabullet) : -K\leq x \leq y \leq K, -B\leq \beta \leq B}\}$ is tight. It therefore suffices to show continuity of the finite dimensional distributions. This follows from \eqref{eq:p2pdef}, continuity of $\She_{\aabullet}(t,\aabullet|s,\aabullet)$ and Scheffe's lemma.
\end{proof}
Now, we turn to the study of measure-to-measure polymers and Theorem \ref{prop:m2mbasic}.
\begin{proof}[Proof of Theorem \ref{prop:m2mbasic}]
We begin by proving part \eqref{prop:m2mbasic:def}. Our first claim is that for any $A \in \sB(\sC_{[s,t]})$, the map $(x,y)\mapsto \Polyb_{(s,y),(t,x)}(A)$ is Borel measurable. Let $\sE = \{A \in \sB(\sC_{[s,t]})$ such that $(x,y) \mapsto \Polyb_{(t,x),(s,y)}(A)$ is Borel measurable$\}$. 
It is easily checked that $\sE$ is a $\lambda$ system. 
By the Portmanteau theorem \cite[Theorem 3.3.1]{Eth-Kur-86} and Theorem \ref{prop:polyexist}, for open sets $O \in \sB(\sC_{[s,t]})$ the map $(x,y) \mapsto \Polyb_{(t,x),(s,y)}(O)$ is lower semicontinuous and therefore $O \in \sE$. 
By the $\pi-\lambda$ theorem \cite[Appendix Theorem 4.2]{Eth-Kur-86}, $ \sE =\sB(\sC_{[s,t]}).$ 

Verifying the axioms of a probability measure follows immediately from the fact that each $\Polyb_{(t,x),(s,y)}$ is a probability measure and the Fubini-Tonelli theorem \cite[Theorem 8.8(a)]{Rud-87}. The formula for the density follows immediately from the definition of $\Polyb_{(s;\mu),(t;\zeta)}$ in \eqref{eq:m2mdef}  and considering expectations of Borel functions of the form $f(X_s,X_{t_1},\dots,X_{t_n},X_t)$.

Part \eqref{prop:m2mboch:cont} is an immediate consequence of \eqref{eq:p2phold} in Theorem \ref{prop:polyexist} and the definition of $\Polyb_{(s;\mu),(t;\zeta)}$ in \eqref{eq:m2mdef}

Part \eqref{prop:m2mbasic:Markov} follows from Theorem \ref{prop:polyexist} \eqref{prop:polyexist:Markov} and \eqref{eq:m2mdef}.

To verify part \eqref{prop:m2mboch:IC}, first recall that $(\delta_y,\zeta),(\mu,\delta_x) \in \ICMM(0) \subset \ICMM(4)$. The result follows immediately from \eqref{eq:p2pIC} in Theorem \ref{prop:polyexist} and \eqref{eq:m2mdef}. 


To see Part \eqref{prop:m2mboch:TV}, note that, for example, for any $A \in \sB(\sC_{[s,t]})$
\begin{align*}
&|\Polyb_{(s;\mu_1),(t;\zeta_1)}-\Polyb_{(s;\mu_1),(t;\zeta_2)}(A)|=\\ 
&\bigg|\frac{ \int_{\bbR^2}\Sheb(t,x\viiva s,y)\Polyb_{(s,y),(t,x)}(A)\zeta_1(dx)\mu_1(dy) }{\Sheb(t;\zeta_1|s;\mu_1)} - \frac{ \int_{\bbR^2}\Sheb(t,x\viiva s,y)\Polyb_{(s,y),(t,x)}(A)\zeta_2(dx)\mu_1(dy) }{\Sheb(t;\zeta_2|s;\mu_1)}\bigg| \\
&\leq \bigg|\frac{ \int_{\bbR^2}\Sheb(t,x\viiva s,y)\Polyb_{(s,y),(t,x)}(A)\zeta_1(dx)\mu_1(dy) -  \int_{\bbR^2}\Sheb(t,x\viiva s,y)\Polyb_{(s,y),(t,x)}(A)\zeta_2(dx)\mu_1(dy)}{\Sheb(t;\zeta_1|s;\mu_1)} \bigg| \\
&+ \bigg|\int_{\bbR^2}\Sheb(t,x\viiva s,y)\Polyb_{(s,y),(t,x)}(A)\zeta_2(dx)\mu_1(dy)(\Sheb(t;\zeta_1|s;\mu_1)^{-1} - \Sheb(t;\zeta_2|s;\mu_1)^{-1})\bigg|\\
&\leq \frac{\Sheb(t;|\zeta_2-\zeta_1||s;\mu_1)}{\Sheb(t;\zeta_1|s;\mu_1)} + \Sheb(t;\zeta_2|s;\mu_1)|\Sheb(t;\zeta_1|s;\mu_1)^{-1}-\Sheb(t;\zeta_2|s;\mu_1)^{-1}|
\end{align*}
where in the last step, we bounded the absolute value of an integral (of an arbitrary measurable set) against $\zeta_1-\zeta_2$ by the integral against $|\zeta_1-\zeta_2|$ and used $\Polyb_{(s,y),(t,x)}(A)\leq1$. A similar computation gives
\begin{align*}
&|\Polyb_{(s;\mu_1),(t;\zeta_2)}(A)-\Polyb_{(s;\mu_2),(t;\zeta_2)}(A)|\\
&\leq\frac{\Sheb(t;\zeta_2|s;|\mu_2-\mu_1|)}{\Sheb(t;\zeta_2|s;\mu_1)} + \Sheb(t;\zeta_2|s;\mu_2)|\Sheb(t;\zeta_2|s;\mu_1)^{-1}-\Sheb(t;\zeta_2|s;\mu_2)^{-1}|.
\end{align*}
The result follows from the triangle inequality and the fact that for any measure $m$, $\|m\|_{TV} \leq 2 \sup_{A} |m(A)|$.
\end{proof}

\section{Regularity of solutions and polymers}\label{sec:reg}
Having constructed the measure-to-measure polymers, we now prove the following technical result, which describes convergence properties of partition functions and quenched polymer measures.

\begin{proposition}\label{prop:conv}
There is an event $\Omega_0$ with $\bbP(\Omega_0)=1$ so that the following holds for all $\omega \in \Omega_0$, all $s<t$, and all $\beta \in \bbR$.
\begin{enumerate}[label={\rm(\roman*)}, ref={\rm\roman*}]   \itemsep=3pt 
\item\label{prop:conv:part}  Let $\beta_n \to \beta$, $s_n < t_n$ for all $n$, $s_n \to s$, and $t_n \to t$. Suppose that $(\mu_n,\zeta_n) \in \ICMM(4)$ is a sequence of measures satisfying for some $a < \frac{1}{2(t-s)}$
\begin{align*}
\sup_n\int_{\bbR^2}(1+|x|^4+|y|^4) e^{-a(x-y)^2}\zeta_n(dx)\mu_n(dy) < \infty
\end{align*}
and that there exist positive Borel measures $\zeta,\mu$ so that $\zeta_n \to \zeta$ vaguely and $\mu_n \to \mu$ vaguely. Then $\She_{\beta_n}(t_n;\zeta_n|s_n;\mu_n) \to \Sheb(t;\zeta|s;\mu)$.
\item\label{prop:conv:weak} Suppose that for $p>24$ and for $n \in \bbN$,  a sequence of measures $(\mu_n,\zeta_n)\in \ICMM(p)$ satisfy for some $a\leq \frac{1}{2(t-s)}$
\begin{align*}
\sup_{n} \int_{\bbR}\int_{\bbR} (1+|x|^p+|y|^p) e^{-a(x-y)^2} \zeta_n(dx)\mu_n(dy) < \infty
\end{align*}
and there exists $(\mu,\zeta)\in \ICMM(p)$ so that $\zeta_n \to \zeta$  and $\mu_n \to \mu$ vaguely. Let $\beta_n \to \beta$. Then $\Poly^{\beta_n}_{(s;\mu_n),(t;\zeta_n)}$ converges weakly to $\Polyb_{(s;\mu),(t;\zeta)}$ in $\sM_1(\sC_{[s,t]})$. 
\item \label{prop:conv:TV} If $\mu_n,\zeta_n \in \ICM$ are sequences of measures satisfying for some $a\leq\frac{1}{2(t-s)}$
\begin{align*}
\sup_n\int_{\bbR}(1+|x|^4)e^{-ax^2}\zeta_n(dx)< \infty \quad\text{and}\quad  \sup_n\int_{\bbR}(1+|y|^4)e^{-ay^2}\mu_n(dy)< \infty
\end{align*}
and there exist $\zeta,\mu \in \ICM$ so that the total variation measures $|\mu_n - \mu|$ and $|\zeta_n-\zeta|$ converge vaguely to zero, then for all $x,y \in \bbR$,
\begin{align*}
\|\Polyb_{(s;\mu_n),(t,x)}-\Polyb_{(s;\mu),(t,x)}\|_{TV} \to 0 \qquad\text{ and }\qquad \|\Polyb_{(s,y),(t;\zeta_n)} - \Polyb_{(s,y),(t;\zeta)}\|_{TV}\to 0.
\end{align*}
\end{enumerate}
\end{proposition}
\begin{proof}
To see \eqref{prop:conv:part}, notice that vague convergence of the factors implies vague convergence of the product measure: $\mu_n\otimes \zeta_n \to \mu\otimes\zeta$. In particular, there exists $K>0$ so that $\mu_n\otimes \zeta_n([-K,K]^2) \to\mu\otimes \zeta ([-K,K]^2)>0$. Without loss of generality, we assume that the pre-limit terms are all non-zero as well. Let $T>1$ and $B>1$ be such that for all $n$, $-T \leq s_n < t_n \leq T$ and $-B \leq \beta_n \leq B$. By considering $n$ sufficiently large, assume that $\frac{1}{2(t_n-s_n)}>a$ for all $n$. By hypothesis and Corollary \ref{cor:growth}, there exists $C=C(T,B,a,\omega)$ so that we have for all $n$,
\begin{align}
&\sup_n\int_{\bbR^2} \She_{\beta_n}(t_n,x\viiva s_n,y)\mu_n(dy) \zeta_n(dx) \label{eq:tvtight}\\
&\qquad\qquad \leq C\sup_n \int_{\bbR^2}(1+|x|^4+|y|^4)e^{-a(x-y)^2}\mu_n(dy)\zeta_n(dx)< \infty,\notag\\
&\sup_n \int_{(\bbR \backslash [-M,M])^2} \She_{\beta_n}(t_n,x\viiva s_n,y)\mu_n(dy)\zeta_n(dx)\notag\\ 
&\qquad \qquad \leq \frac{C}{1+|M|^4}\sup_n \int_{\bbR^2}(1+|x|^4+|y|^4)e^{-a(x-y)^2}\mu_n(dy)\zeta_n(dx).\notag
\end{align}
The estimates above show that $\She_{\beta_n}(t_n,x\viiva s_n,y)\zeta_n(dx)\mu_n(dy)$ is tight and bounded in total variation norm and so every subsequence has a weakly convergent subsequence by Prohorov's theorem \cite[Theorem 8.6.2]{Bog-07}. Consider test functions $\varphi,\psi \in\sC_c(\bbR)$. Then because the convergence $\She_{\beta_n}(t_n,x\viiva s_n,y) \to \Sheb(t,x\viiva s,y)$ is uniform on the support of $\varphi(x)\psi(y)$ and everything is continuous in $(x,y)$, along any weakly convergent subsequence, we have
\begin{align*}
\int_{\bbR^2} \varphi(x)\psi(y) \She_{\beta_n}(t_n,x\viiva s_n,y)\mu_n(dy)\zeta_n(dx) \to \int_{\bbR^2} \varphi(x)\psi(y) \Sheb(t,x\viiva s,y)\mu(dy)\zeta(dx).
\end{align*}
Thus $\She_{\beta_n}(t_n,x\viiva s_n,y)\zeta_n(dx)\mu_n(dy) \to \Sheb(t,x\viiva s,y)\zeta(dx)\mu(dy)$ weakly. Weak convergence implies $\int_{\bbR^2}\Sheb(t,x\viiva s,y)\zeta_n(dx)\mu_n(dy)$ $\to \int_{\bbR^2}\Sheb(t,x\viiva s,y)\zeta(dx)\mu(dy)$ and so part \eqref{prop:conv:part} follows.

Turning to part \eqref{prop:conv:weak}, notice that by Proposition \ref{prop:rncont}, for any $B,K>1$ \[\min_{(x,y,\beta) \in [-K,K]^2\times[-B,B]} \{\Sheb(t,x\viiva s,y)\}>0.\] 
Take $B$ sufficiently large that $\beta_n \in [-B,B]$ for all $n$. We have
\begin{align*}
&\Poly^{\beta_n}_{(s;\zeta_n),(t;\mu_n)}(X_s \in [-M,M]^c) = \frac{\int_{\R}\int_{\bbR\backslash[-M,M]}\Sheb(t,x\viiva s,y)\zeta_n(dx)\mu_n(dy)}{\int_{\bbR^2} \Sheb(t,x\viiva s,y)\zeta_n(dx)\mu_n(dy)}\\
&\qquad\leq \frac{C \sup_n \int_{\bbR^2}(1+|x|^p+|y|^p)\heat(t-s,x-y)\zeta_n(dx)\mu_n(dy)}{(1+|M|^p)\zeta_n\otimes\mu_n([-K,K]^2) \min\limits_{(x,y,\beta) \in [-K,K]^2\times[-B,B]} \{\Sheb(t,x\viiva s,y)\}}.
\end{align*}
Therefore, the one-point distribution of $X_s$ under $\Poly^{\beta_n}_{(s;\zeta_n),(t;\mu_n)}$ is tight.

Now, pick $\eta \in (0,1/2)$ and $\epsilon \in (0,1)$ sufficiently small that $\frac{4+2\epsilon}{1-2\eta}\in (0,p-4).$ For such $\eta$,  \eqref{eq:m2mdef}, \eqref{eq:PolyHolder}, and Corollary \ref{cor:growth}, imply that there exists $C=C(s,t,B,\eta,\epsilon,\omega)>0$ so that for $\beta \in [-B,B]$ and $-T \leq s <t \leq T$,
\begin{align*}
\oE^{\Polyb}_{(s;\mu),(t;\zeta)}[|X|_{\sC^{\eta}_{[s,t]}}] \leq 
\int_{\bbR}\int_{\bbR}\frac{C(1+|x|^{\frac{4+2\epsilon}{1-2\eta} + 20}+|y|^{\frac{4+2\epsilon}{1-2\eta} + 20} )\heat(t-s,x-y)}{\Sheb(t;\zeta|s;\mu) }\zeta(dx)\mu(dy).
\end{align*}
Therefore
\begin{align*}
&\oE_{(s;\mu_n),(t;\zeta_n)}^{\Poly^{\beta_n}}[|X_{\aabullet}|_{\sC_{[s,t]}^\eta}]\leq\\ 
&C \frac{\sup_n \int_{\bbR^2}(1+|x|^p+|y|^p)\heat(t-s,x-y)\mu_n(dy)\zeta_n(dx)}{\min_{(x,y)\in[-K,K]}\{(1+|x|^4+|y|^4)^{-1} \heat(t-s,x-y) \}\zeta_n\otimes\mu_n([-K,K]^2)}.
\end{align*}
By  \cite[Theorem 8.2]{Bil-99}, it follows that $\{\Poly^{\beta_n}_{(s;\zeta_n),(t;\mu_n)}: n \in \bbN\}$ is tight in $\sM_1(\sC_{[s,t]})$ and so it suffices to show vague convergence of the finite dimensional marginals. Take $\varphi_0,\dots,\varphi_{k+1}\in \sC_c(\bbR,\bbR)$. Then for $s = t_0 < t_1 < \dots < t_k < t = t_{k+1}$, we have
\begin{align*}
&\oE_{(s;\mu_n),(t;\zeta_n)}^{\Poly^{\beta_n}}\bigg[\varphi_0(X_s)\prod_{i=1}^k\varphi_i(X_{t_i})\varphi_{k+1}(X_t)\bigg] \\
&= \frac{\int_{\bbR} \int_{\bbR}\varphi_0(y) \oE^{\Poly^{\beta_n}}_{(s,y),(t,x)}[\prod_{i=1}^k\varphi_i(X_{t_i})] \varphi_{k+1}(x) \She_{\beta_n}(t,x\viiva s,y) \mu_n(dy)\zeta_n(dx)}{\She_{\beta_n}(t;\zeta_n|s;\mu_n)}.
\end{align*}
The denominator was shown to converge to $\Sheb(t;\zeta|s;\mu)$ in part \eqref{prop:conv:part}.The convergence of continuous functions (in $(x,y)$) \[\oE^{\Poly^{\beta_n}}_{(s,y),(t,x)}\bigg[\prod_{i=1}^k\varphi_i(X_{t_i})\bigg] \to \oE^{\Poly^{\beta}}_{(s,y),(t,x)}\bigg[\prod_{i=1}^k\varphi_i(X_{t_i})\bigg]\] is uniform on the (compact) support of $\varphi_0(x)\varphi_{k+1}(y)$ by Theorem \ref{prop:polyexist} \eqref{prop:polyexist:cmeas}. Similarly, the convergence of the continuous functions (in $(x,y)$) $\She_{\beta_n}(t,x\viiva s,y) \to \Sheb(t,x\viiva s,y)$ is uniform on the support of $\varphi_0(x)\varphi_{k+1}(y)$. We may then conclude from vague convergence of $\mu_n\otimes\zeta_n \to \mu \otimes \zeta$ that
\begin{align*}
\oE_{(s;\mu_n),(t;\zeta_n)}^{\Poly^{\beta_n}}\bigg[\varphi(X_s)\prod_{i=1}^k\varphi_i(X_{t_i})\varphi(X_t)\bigg] \to \oE_{(s;\mu),(t;\zeta)}^{\Polyb}\bigg[\varphi(X_s)\prod_{i=1}^k\varphi_i(X_{t_i})\varphi(X_t)\bigg].
\end{align*}
The result follows.

To see that \eqref{prop:conv:TV} holds, we appeal to Theorem \ref{prop:m2mbasic} \eqref{prop:m2mboch:TV} and part \eqref{prop:conv:part} of this result.
\end{proof}

With the previous result in hand, we turn to the proof of Theorem \ref{thm:jcont}.
\begin{proof}[Proof of Theorem \ref{thm:jcont}]
We begin with part \eqref{thm:jcont:MM}. Take $\mu_n \to \mu$ in the metric on $\ICM$ defined in equation \eqref{eq:ICMm} and $\beta_n\to\beta, s_n \to s,$ $t_n \to t$, and $s_n \leq t_n$ for all $n$. Let $f \in \sC(\R,\R_+)$ be any function for which there exist $a,A>0$ such that $0 \leq f(z) \leq A e^{-az^2}$ for all $z \in \bbR$. The claim is that with the convention in \eqref{eq:mbdydef},
\[\int_{\R}f(x)\She_{\beta_n}(t_n, dx \viiva s_n; \mu_n) \to  \int_{\R} f(x) \Sheb(t, dx \viiva s; \mu).\]
We first consider the case where $s<t$.  Then, by hypothesis, we have for any $b>0$
\begin{align*}
\int_{\bbR^2}e^{-b(x-y)^2}f(y)dy\mu_n(dx) &\leq A\int_{\bbR^2}e^{-b(x-y)^2}e^{-a y^2}dy\mu_n(dx) \\
&= A' \int_{\R} e^{-c x^2}\mu_n(dx) \to A' \int_{\R}e^{-cx^2}\mu(dx).
\end{align*}
where $A'=A'(A,a,b)$ and $c=c(a,b)$ are explicit and come from standard computations involving Gaussian kernels. It follows that the integral on the left is bounded as a function of $n$.  The hypotheses of Proposition \ref{prop:conv}\eqref{prop:conv:part} are satisfied and so
\begin{align*}
\She_{\beta_n}(t_n;f | s_n;\mu_n) \to \Sheb(t;f | s;\mu).
\end{align*} 
Next, suppose that $s=t$. If for all sufficiently large $n$, $s_n = t = t_n$, then there is nothing to be shown, so we may assume without loss of generality that for all sufficiently large $n$, $t_n - s_n>0$ and $t_n-s_n \searrow 0$. This case follows from Theorem \ref{prop:IC}\eqref{prop:IC:musat}.

Next, we turn to part \eqref{thm:jcont:MC}. Again take $\mu_n \to \mu$ in $\ICM$ and  $\beta_n\to\beta, s_n \to s, t_n \to t$, and $y_n \to y$, where $s_n < t_n$ for all $n$ and $s<t$. We just showed that the integrals against $f \in \sC(\R,\R_+)$ for which there exist $a,A>0$ such that $0 \leq f(z) \leq A e^{-az^2}$ for all $z \in \bbR$ converge. To show convergence in $\CICM$, because the limit is strictly positive, it then suffices to show that for any $m \in \bbN$,
\begin{align*}
\lim_{n\to\infty}\sup_{y \in [-m,m]}\bigg|\She_{\beta_n}(t_n, y|s_n;\mu_n)- \Sheb(t,y|s;\mu)\bigg| = 0.
\end{align*}
If not, then for any $\epsilon>0$, there would exist $y_n \in [-m,m]$ such that $|\She_{\beta_n}(t_n, y_n|s_n;\mu_n)- \Sheb(t,y_n|s;\mu)|>\epsilon$ for all $n$. Passing to a subsequence, we may assume that $y_n \to y \in [-m,m]$. By continuity of $y \mapsto\Sheb(t,y|s;\mu)$, this would imply that for all sufficiently large $n$, $|\She_{\beta_n}(t_n, y_n|s_n;\mu_n)- \Sheb(t,y|s;\mu)|>\epsilon/2$, which contradicts  Proposition \ref{prop:conv}\eqref{prop:conv:part}. 

Finally, we turn to part \eqref{thm:jcont:CC}.  Take $f_n \to f$ in the metric on $\CICM$ defined in equation \eqref{eq:CICMm} and $\beta_n\to\beta, s_n \to s,$ $t_n \to t$, and $s_n \leq t_n$ for all $n$. Then because uniform convergence implies vague convergence of the represented measures, we have $f_n(x) dx\to f(x)dx$ in the topology of $\ICM$. Part \eqref{thm:jcont:MC} then implies the result if $s<t$. Consider now the case $s=t$. We may assume that $t_n-s_n>0$ for all sufficiently large $n$, else there is otherwise nothing to prove. In this case, convergence of the integrals appearing in the metric $d_{\CICM}$ in \eqref{eq:CICMm} follows from part \eqref{thm:jcont:MM}. As $f \in \CICM$ is strictly positive, it remains to show that for $m\in\N$,
\[
\lim_{n\to\infty}\max_{-m \leq x \leq m}\bigg|\She_{\beta_n}(t_n,x\viiva s_n; f_n) - f(x)\bigg| = 0.
\]
This follows from Theorem \ref{prop:IC}\eqref{prop:IC:fcont}.
\end{proof}

Next, we show the corresponding continuity results for the quenched polymer measures, recorded as Theorem \ref{thm:polyreg}.
\begin{proof}[Proof of Theorem \ref{thm:polyreg}]
Part \eqref{thm:polyreg:weak} is an immediate consequence of Proposition \ref{prop:conv}\eqref{prop:conv:weak}. Similarly, part \eqref{thm:polyreg:TV} follows from Proposition \ref{prop:conv}\eqref{prop:conv:TV}.

Turning to \eqref{thm:polyreg:SF}, take $s<r<t$ and $f \in \sB_b(\bbR)$ as in the statement and let $x,y \in \bbR$ be given. We have
\begin{align*}
\oE_{(s,y),(t,x)}^{\Polyb}[f(X_r)] &= \frac{\int_{\bbR}f(z)\rnSheb(t,x|r,z)\rnSheb(r,z\viiva s,y) \heat(t-r,x-z)\heat(r-s,z-y)dz }{\Sheb(t,x\viiva s,y)}.
\end{align*}
Continuity in $(s,y,t,x,\beta)$ now follows from continuity of the denominator and the integrand,  Corollary \ref{cor:growth}, and the dominated convergence theorem, whenever $s < r < t$. 

Next, we consider part \eqref{thm:polyreg:vbr} $f\in \sC_c(\R,\R)$. Take sequences $s_n \to r$ with $s_n \leq r$, $y_n \to y$, and $\beta_n\to\beta$. Passing to subsequences, we handle the cases of $s_n = r$ for all $n$ and $s_n < r$ for all $n$ separately. If $s_n = r$ for all $n$, then by Theorem \ref{prop:polyexist}, $\E^{\Poly^{\beta_n}}_{(r,y_n),(t,x)}[f(X_r)]= f(y_n)$ which converges to $f(y)$ by continuity. Now, consider the case of $s_n<r$ with $s_n \to r$.
\begin{align*}
\oE_{(s_n,y_n),(t,x)}^{\Poly^{\beta_n}}[f(X_r)] &= \frac{\int_{\bbR}f(z)\rnShe_{\beta_n}(t,x|r,z)\rnShe_{\beta_n}(r,z|s_n,y_n) \heat(t-r,x-z)\heat(r-s_n,z-y_n)dz }{\She_{\beta_n}(t,x\viiva s_n,y_n)}.
\end{align*}
The denominator converges to $\Sheb(t,x\viiva s,y)$ by continuity. Convergence of the integral follows from observing that $\rnShe_{\beta_n}(r,z|s_n,y_n) \to 1$ uniformly over $z\in \supp f$ and $\heat(r-s_n,z-y_n)dz \to \delta(y)$ vaguely. Continuity of $(t,x,\beta) \mapsto \oE^{\Polyb}_{(s,y),(t,x)}[f(X_r)]$ is similar.

Now, take $\zeta \in \ICM$. We have
\begin{align*}
\oE_{(s,y),(t;\zeta)}^{\Polyb}[f(X_r)] &= \frac{\int_{\bbR} \rnSheb(t,x\viiva s,y)\heat(t-s,x-y)\oE_{(s,y),(t,x)}^{\Polyb}[f(X_r)]\zeta(dx)}{\Sheb(t;\zeta\viiva s,y)}.
\end{align*}
Continuity of the partition function $\Sheb(t;\zeta\viiva s,y)$ in $(t,y,s,\beta)$ follows from Theorem \ref{prop:IC}. Continuity of $(t,y,s,\beta) \mapsto \oE_{(s,y),(t,x)}^{\Polyb}[f(X_r)] \rnSheb(t,x\viiva s,y)\heat(t-s,x-y)$, Corollary \ref{cor:growth}, and the dominated convergence theorem give continuity of the integral, which holds by the result shown above for the point-to-point measures. Continuity of $(s,t,x,\beta) \mapsto \oE_{(s;\mu),(t,x)}^{\Polyb}[f(X_r)]$ is similar. If $f \in \sC_c(\R,\R)$, continuity of $(s,y,\beta) \mapsto \oE_{(s,y),(t;\zeta)}^{\Polyb}[f(X_r)]$ and $(t,x,\beta) \mapsto$ $\oE_{(s;\mu),(t,x)}^{\Polyb}[f(X_r)]$ on $\{(s,y,\beta) \in \bbR^3 : s \leq r\}$ and $\{(t,x,\beta) \in \bbR^3 : t \geq r\}$ follow from the same argument.
\end{proof}

We now turn to the Karlin-McGregor formula, recorded as Proposition \ref{prop:KMG} below.   Before stating the result, we introduce and recall some notation. We denote the Weyl chamber in $\R^n$ by $\bbW_n = \{(x_1,\dots,x_n) : x_1 < \dots < x_n\}$. For $(x_1,\dots,x_n),(y_1,\dots,y_n)\in\bbW_n$, $s,t,x,y \in \bbR$ with $s<t$, $\mu,\zeta \in \ICM$, and $B_1,\dots, B_n \in \sB(\sC_{[s,t]})$, and denoting the coordinate random variables by $(X^{1},\dots, X^{n})$, we set
\begin{align*}
\Polyb_{(s,y_1,\dots,y_n),(t;\zeta)}(X^{i} \in B_i, 1 \leq i \leq n) &= \prod_{i=1}^n \Polyb_{(s,y_i),(t;\zeta)}(X \in B_i), \qquad \text{ and } \\
\Polyb_{(s;\mu),(t,x_1,\dots,x_n)}(X^{i} \in B_i, 1 \leq i \leq n) &= \prod_{i=1}^n \Polyb_{(s;\mu),(t,x_i)}(X \in B_i).
\end{align*}
That is, $\Polyb_{(s,y_1, \dots , y_n),(t;\zeta)}$ is the law of $n$ independent point-to-measure polymers. The polymer paths start from the points $y_1,\dots, y_n$ at time $s$ and run until time $t$, where they share a boundary condition given by the measure $\zeta$. We view this as a measure on the product space $\sC([s,t],\bbR)^n$. There is a similar interpretation for $\Polyb_{(s;\mu),(t,x_1,\dots,x_n)}$. As before, we replace $(s;\mu)$ with $(s,y)$ if $\mu = \delta_y$, with similar notation for $(t,x)$. Call $\cfil_{s:t}^{(n)} =\sigma(X^{1}_u, \dots X^{n}_u : s \leq u \leq t)$ the associated natural filtration. For each $s<t$ and each $n \in \bbN$, introduce the $\cfil_{s:t}^{(n)}$-stopping time
\begin{align*}
\tau_{s:t}^{(n)} &= \inf\{ u \in [s,t] \text{ such that there exist } i,j \in [n] \text{ with } i \neq j \text{ and }X^i_u = X^j_u\}.
\end{align*}
We say that $B_1\times\dots\times B_n \in \sB(\bbR^n)$ is coordinate-wise ordered if $x \in B_i$ and $y \in B_j$ with $i < j$ implies $x < y$.
\begin{proposition}\label{prop:KMG}
There exists an event $\Omega_0$ with $\bbP(\Omega_0)=1$ so that on $\Omega_0$, the following hold:
\begin{enumerate} [label={\rm(\roman*)}, ref={\rm\roman*}]   \itemsep=3pt 
\item \label{prop:KMG:det}
For all   $s<t$, all $r \in (s,t)$, all $\beta \in \bbR$, all $B_1\times\cdots\times B_n \in \sB(\bbR^n)$ which are coordinate-wise ordered, all $\mu,\zeta \in \ICM$, and all $(x_1,\dots,x_n),(y_1,\dots,y_n) \in \bbW_n$,
\begin{align*}
&\Polyb_{(s,y_1,\dots,y_n),(t;\zeta)}(X^{1}_r \in B_1, \dots, X^{n}_r \in B_n, \tau_{s:t}^{(n)}>r)\\ 
&= \int\limits_{B_1 \times \dots \times B_n} \det_{1 \leq i,j \leq n}\bigg[\frac{\Sheb(t;\zeta|r,z_j)\Sheb(r,z_j\viiva s,y_i)}{\Sheb(t;\zeta\viiva s,y_i)} \bigg] dz_{1:n}
\end{align*}
and
\begin{align*}
&\Polyb_{(s;\mu),(t,x_1,\dots,x_n)}(X^{1}_r \in B_1, \dots, X^{n}_r \in B_n, \tau_{s:t}^{(n)}>r)\\ 
&= \int\limits_{B_1 \times \dots \times B_n} \det_{1\leq i,j \leq n}\bigg[\frac{\Sheb(t,x_i|r,z_j)\Sheb(r,z_j|s;\mu)}{\Sheb(t,x_i|s;\mu)} \bigg]dz_{1:n}
\end{align*}
\item\label{prop:KMG:pos} For all $s<t$, all $r \in (s,t)$, all $\beta \in \bbR$, all $(x_1,\dots x_n),(y_1,\dots,y_n),(z_1,\dots,z_n)\in\bbW_n$ and all $\mu,\zeta\in\ICM$, 
\begin{align*}
&\det_{1 \leq i,j \leq n}\bigg[\frac{\Sheb(t;\zeta|r,z_j)\Sheb(r,z_j\viiva s,y_i)}{\Sheb(t;\zeta\viiva s,y_i)} \bigg]> 0 \ \text{ and } \  \det_{1 \leq i,j \leq n}\bigg[\frac{\Sheb(t,x_i|r,z_j)\Sheb(r,z_j|s;\mu)}{\Sheb(t,x_i|s;\mu)} \bigg]  > 0.
\end{align*}
\item\label{prop:KMG:pos2}  For all $s<t$, all $\beta \in \bbR$, and all  $(y_1,\dots,y_n),(x_1,\dots,x_n)\in\bbW_n$,
\begin{align*}
\det_{1 \leq i,j \leq n}\big[\Sheb(t,x_j\viiva s,y_i) \big] > 0.
\end{align*} 
\end{enumerate}
\end{proposition}
Note that Theorem \ref{thm:strpos} is Proposition \ref{prop:KMG}\eqref{prop:KMG:pos2}. The form of the first part of the proof is adapted from an argument due to Varadhan, which is sketched in the discrete case in Exercise 4.3.5 of \cite{Hou-etal-09}.
\begin{proof}[Proof of Proposition \ref{prop:KMG}]
Consider $\beta \in \bbR$ $(y_1, y_2. \dots, y_n)\in\bbW_n$, $\zeta \in \ICM$, and $\varphi_1,\dots \varphi_n \in \sC_c(\R,\R)$ where the sets $B_i = \supp \varphi_i$ are coordinate-wise ordered as in the statement. In particular, note that this condition enforces that $\supp \varphi_i \cap \supp \varphi_j = \emptyset$ if $i \neq j$. Call $\oE^{\Polyb}_{(s,y_1,\dots,y_n),(t;\zeta)}$ the expectation under $\Polyb_{(s,y_1,\dots,y_n),(t;\zeta)}$. Denote by $S_n$ the group of permutations on $[n]=\{1,\dots,n\}$ and fix $\sigma \in S_n$. We consider the $\cfil_{s:u}^{(n)}$ martingale defined by
\begin{align*}
M^{\sigma}(u) = \oE^{\Polyb}_{(s,y_1,\dots,y_n),(t;\zeta)}\bigg[\prod_{i=1}^n \varphi_{i}(X_r^{\sigma(i)})\bigg| \cfil_{s:u}^{(n)} \bigg].
\end{align*}
Independence of the coordinate random variables and the Markov property of $\Polyb_{(s,y_i),(t;\zeta)}(\aabullet)$ imply that for $u \leq r$, $\Polyb_{(s,y_1,\dots,y_n),(t;\zeta)}$ almost surely,
\be\label{eq:KMGmgsig}\begin{aligned} 
M^{\sigma}(u) &= \prod_{i=1}^n \oE^{\Polyb}_{(u,X^{\sigma(i)}_u),(t;\zeta)}\big[\varphi_{i}(X_r)\big]  \\&= \int_{\bbR^n} \prod_{i=1}^n \varphi_{i}(z_{i}) \prod_{i=1}^n \frac{\Sheb(t;\zeta|r,z_{i})\Sheb(r,z_{i}|u,X^{\sigma(i)}_u)}{\Sheb(t;\zeta|u,X^{\sigma(i)}_u)}dz_{1:n}.
\end{aligned}\ee
Note that the path continuity of $X_u$ under $\Polyb_{(s,y_i),(t;\zeta)}$ combined with Theorem \ref{thm:polyreg}\eqref{thm:polyreg:vbr} implies that $u \mapsto M_u^{\sigma}$ is then a bounded and continuous $\cfil_{s:u}^{(n)}$ martingale on $[s,r]$ under $\Polyb_{(s,y_1,\dots,y_n),(t;\zeta)}$. Therefore, all of these properties also hold for
\begin{align}
M(u) &= \sum_{\sigma \in S_n} (-1)^{\sgn(\sigma)} M^{\sigma}(u) = \det_{1 \leq i,j \leq n}\big[\oE^{\Polyb}_{(u,X_u^i),(t;\zeta)}[\varphi_j(X_r)]\big]\label{eq:KMGmg1}\\
&= \int_{\bbR^n} \prod_{i=1}^n \varphi_{i}(z_{i}) \det\bigg[\frac{\Sheb(t;\zeta|r,z_{i})\Sheb(r,z_{i}|u,X^{j}_u)}{\Sheb(t;\zeta|u,X^{j}_u)} \bigg]_{i.j=1}^ndz_{1:n}. \label{eq:KMGmg2}
\end{align}
By the optional stopping theorem,
\begin{align*}
M(s) &= \oE^{\Polyb}_{(s,y_1,\dots,y_n),(t;\zeta)}[M(\tau_{s:t}^{(n)} \wedge r)] = \oE^{\Polyb}_{(s,y_1,\dots,y_n),(t;\zeta)}[M(\tau_{s:t}^{(n)}) \ind_{\{\tau_{s:t}^{(n)} \leq r\}} + M(r) \ind_{\{\tau_{s:t}^{(n)} > r\}}]
\end{align*}
Note that on the event $\{\tau_{s:t}^{(n)} \leq r\}, M(\tau_{s:t}^{(n)}) = 0$ because two rows in the matrix in the determinant in \eqref{eq:KMGmg1} are equal. On the other hand, because the $y_i$ are ordered and the supports of the test functions $\varphi_i$ are also ordered, if $\sigma$ is not the identity, then path continuity forces
\begin{align*}
M^{\sigma}(r) = \prod_{i=1}^n \oE^{\Polyb}_{(r,X^{\sigma(i)}_r),(t;\zeta)}\big[\varphi_{i}(X_r)\big] \ind_{\{\tau_{s:t}^{(n)}>r\}} = 0, \qquad \Polyb_{(s,y_1,\dots,y_n),(t;\zeta)}\text{-almost surely}.
\end{align*}
When $\sigma$ is the identity, by Theorem \ref{prop:polyexist}, $M^{\sigma}(r) = \prod_{i=1}^n \varphi_i(X_r)$ almost surely under $\Polyb_{(s,y_1,\dots,y_n),(t;\zeta)}$. Consequently, we have
\begin{align*}
&\int_{\bbR^n} \prod_{i=1}^n \varphi_{i}(z_{i}) \det_{1 \leq i,j \leq n}\bigg[\frac{\Sheb(t;\zeta|r,z_{i})\Sheb(r,z_{i} \viiva s,y_j)}{\Sheb(t;\zeta\viiva s,y_j)} \bigg]dz_{1:n} \\
&\qquad\qquad= \oE^{\Polyb}_{(s,y_1,\dots,y_n),(t;\zeta)}\bigg[ \prod_{i=1}^n \varphi_i(X_r^i)\ind_{\{\tau_{s:t}^{(n)}>r\}}\bigg].
\end{align*}
A standard monotone class theorem \cite[Appendix Theorem 4.3]{Eth-Kur-86} argument then implies that the previous identity holds with $\prod_{i=1}^n \varphi_i$ replaced by $f \in \sB_b(\bbW_n)$. In particular, if $B_1\times\dots\times B_n$ is coordinate-wise ordered, then
\begin{align*}
&\int\limits_{B_1\times\dots\times B_n} \det_{1\leq i,j \leq n}\bigg[\frac{\Sheb(t;\zeta|r,z_{i})\Sheb(r,z_{i}\viiva s,y_j)}{\Sheb(t;\zeta\viiva s,y_j)} \bigg]dz_{1:n}= \Polyb_{(s,y_1,\dots,y_n),(t;\zeta)}(X^i_r \in B_i, i \in [n],\tau_{s:t}^{(n)}>r).
\end{align*}
The result for $\Polyb_{(s;\mu),(t,x_1,\dots,x_n)}$ is similar.  This completes the proof of part \eqref{prop:KMG:det}.

\smallskip 

Turning to the proof of parts \eqref{prop:KMG:pos} and \eqref{prop:KMG:pos2}, note that non-negativity of the determinant in part \eqref{prop:KMG:pos} follows immediately from part \eqref{prop:KMG:det}. The claim is that this inequality is everywhere strict. Note that by multilinearity of the determinant, we have
\begin{align}
\det_{1 \leq i,j \leq n}\bigg[\frac{\Sheb(t;\zeta|r,z_j)\Sheb(r,z_j\viiva s,y_i)}{\Sheb(t;\zeta\viiva s,y_i)} \bigg]&= \frac{\prod_{j=1}^n \Sheb(t;\zeta|r,z_j)}{\prod_{i=1}^n\Sheb(t;\zeta\viiva s,y_i) }\det_{1\leq i,j \leq n}\bigg[\Sheb(r,z_j\viiva s,y_i)\bigg]\label{eq:detid}
\end{align}
with a similar identity for the other term in part \eqref{prop:KMG:pos}. By strict positivity of the first term on the right-hand side of this equality, we see that parts \eqref{prop:KMG:pos} and \eqref{prop:KMG:pos2} are equivalent.

Fix $s <t$, $r \in (s,t)$, $(y_1,\dots,y_n)\in\bbW_n$, $\zeta \in \ICM$, and $\delta \in (0,1)$. Let $B_1,\dots,B_n \in \sB(\R)$ be bounded Borel sets of positive Lebesgue measure with the property that if $x \in B_i$ and $y \in B_j$ with $i<j$, then $y-x > \delta$. Our first claim is that for any sets satisfying these conditions,
\begin{align*}
\Polyb_{(s,y_1,\dots,y_n),(t;\zeta)}(X^i_r \in B_i \text{ for } i \in [n],\tau_{s:t}^{(n)}>r) >0.
\end{align*}
For each $i$, let $[a_i,b_i] = \bigcap_{B_i \subset [a,b]} [a,b]$ be the smallest closed interval containing $B_i$. By hypothesis, we have for $i \in [n-1]$, $b_i + \delta \leq a_{i+1}$. Call $c_i = (a_i+b_i)/2$.  In order to show that this event has positive probability, we first show that with positive probability, for some  $v$  slightly smaller than $r$,  each of the $n$ paths remains in a narrow symmetric cylinder of radius $\epsilon$ around the straight line segment connecting $(s,y_i)$ and $(v,c_i)$ for the entire time interval $[s,v]$. These cylinders will be chosen to be narrow enough that the paths cannot intersect before time $v$ as long as they remain in the cylinders. Then the path is required to end in $B_i$ at time $r$, without exiting the interval $(a_i-\delta/4,b_i+\delta/4)$ on $[v,r]$. See Figure \ref{fig:polypf} for an illustration.

\begin{figure}
\includegraphics{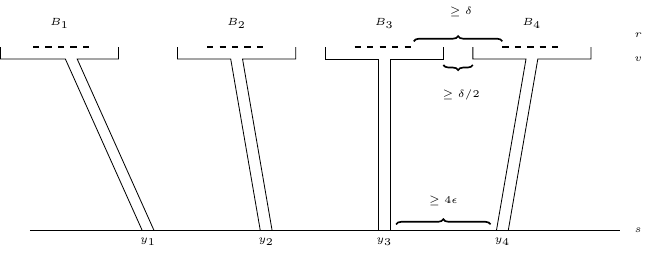}
\caption{\small The path $X^i$ starts at $y_i$ at time $s$ and reaches $B_i$ (dashed, at level $r$) at time $r$, while remaining inside the tube connecting $y_i$ and $B_i$ (drawn in solid lines) between times $s$ and $r$. Before time $v$, the tube is a symmetric cylinder with diameter $4\epsilon$ around the straight line segment connecting $y_i$ and $c_i$. The tube expands at a time $v$ (which is slightly smaller than $r$) in order to allow the path enough space to reach any point in $B_i$, while still not allowing for intersections with any other path.}\label{fig:polypf}
\end{figure}

Recall that for $a<b$, we defined the process $\tilde{X}$ on $\sC([a,b],\bbR)$ by
\begin{align*}
\widetilde{X}_u &= X_u -\bigg(\frac{b-u}{b-a}X_a + \frac{u-a}{b-a}X_b\bigg),
\end{align*}
Note that the definition of $\tilde{X}$ depends on the space on which it is defined through $a$ and $b$, which we suppress. For $v \in (s,r)$ denote by $\ell_{i,v}(u)$ the straight line segment connecting $(s,y_i)$ with $(v,c_i)$, where $c_i = (a_i+b_i)/2$, i.e., for $u \in [s,v]$,
\[\ell_{i,v}(u) = \frac{v-u}{v-s}\cdot y_i + \frac{u-s}{v-s} \cdot \frac{a_i+b_i}{2}.\]

Because $\zeta \in \ICM$, there exists $K>0$ so that $\zeta[-K,K]>0$. Without loss of generality, we may assume that $K$ is sufficiently large that $a_1-1,b_n+1, y_1-1,y_n+1 \in [-K,K]$.  By Lemma \ref{lem:Qbridge} , there exists $C=C(t,s,K,\omega)>0$ so that for all $x,y \in [-K,K]$ and all $u,v \in [s,t]$ with $u<v$,
\begin{align*}
\oE_{(u,y)(v,x)}^{\Polyb}[|\widetilde{X}|_{\sC^{1/4}_{[u,v]}}] < C.
\end{align*}

Let $\epsilon\in (0,\delta/4)$ be sufficiently small that for $i \in [n-1]$, $y_{i+1}-y_i > 8\epsilon$ and for $i \in [n]$, $b_i-a_i>8\epsilon$. In particular, $(c_i-\epsilon, c_i + \epsilon) \subset [a_i,b_i]$ for all $i \in [n]$. Take $m \in \bbN$ sufficiently large that $C(r-s)^{1/4}/m^{1/4} < \epsilon^2$. Let $u_{j} = s + (r-s)\frac{j}{m},$ $j \in \{0,\dots,m\}$,  be a uniform partition of $[s,r]$ with mesh size $m^{-1}(r-s)$. In particular, we have for all $x,y \in [-K,K]$ and all $j \in [m]$,
\begin{align*}
\Polyb_{(u_{j-1},x),(u_j,y)}\bigg(\sup_{u,v \in [u_{j-1},u_j]}|\widetilde{X}_u - \widetilde{X}_v| \geq \epsilon\bigg) \leq \Polyb_{(u_{j-1},x),(u_j,y)}\bigg(|\widetilde{X}|_{\sC_{[u_{j-1},u_j]}^{1/4}} \geq \frac{m^{1/4}\epsilon}{(r-s)^{1/4}}\bigg) \leq \epsilon,
\end{align*}
by Markov's inequality and the choice of $m$. The next part of the argument is illustrated in Figures \ref{fig:polypf2} and \ref{fig:polypf3}. Call
\begin{align*}
E_i &= \bigg\{\sup_{u \in [s,u_{m-1}]} |X_{u}^i - \ell_{i,u_{m-1}}(u)| < 2\epsilon,\, X_r^i \in B_i, \,\forall u \in [u_{m-1},r]: X_u^i \in \bigg(a_i-\frac{\delta}{4},b_i+\frac{\delta}{4}\bigg)\bigg\}.
\end{align*}
Momentarily viewing $E_i$ as an event on $\sC_{[s,t]}$ instead of the product space, for any $x \in [-K,K]$ and $i \in [n]$, we have
\begin{align*}
&\Polyb_{(s,y_i),(t,x)}(E_i) 
\geq \Polyb_{(s,y_i),(t,x)}\bigg( X_r \in B_i; \, \forall j \in [m-1]: |X_{u_{j}} - \ell_{i,u_{m-1}}(u_{j})| < \epsilon;\\
&\qquad\qquad\qquad\qquad\qquad\qquad\qquad\qquad\qquad  \forall j \in [m]: \sup_{u,v \in [u_{j-1},u_j]}|\widetilde{X}_u - \widetilde{X}_v| < \epsilon\bigg) \\
&= \oE_{(s,y_i),(t,x)}^{\Polyb}\bigg[ \prod_{j=1}^{m-1} \Polyb_{(u_{j-1},X_{u_{j-1}}),(u_{j},X_{u_{j}})}\bigg(\sup_{u,v \in [u_{j-1},u_j]}|\widetilde{X}_u - \widetilde{X}_v| < \epsilon\bigg)\ind_{\{|X_{u_{j}} - \ell_{i}(u_{j})| < \epsilon\}}\one_{\{X_r^i \in B_i\}}  \bigg] \\
&\geq (1-\epsilon)^m\tspb \Polyb_{(s,y_i),(t,x)}\bigg(X_r^i \in B_i; \, \forall j \in [m-1]: |X_{u_{j}} - \ell_{i}(u_{j})| < \epsilon\bigg),
\end{align*} 
where in the second-to-last step, we have repeatedly applied the Markov property and in the last step that $\epsilon < \delta/4$. Finally, by Lemma \ref{lem:rndbd}, there is a constant $C'=C'(T,m,\omega)$ so that
\begin{align*}
&\Polyb_{(s,y_i),(t,x)}\bigg(X_r^i \in B_i; \, \forall j \in [m-1]: |X_{u_{j}} - \ell_{i}(u_{j})| < \epsilon\bigg) \\
&\qquad
\geq \frac{C'}{(1+2K)^m} \inf_{-K \leq x \leq K}\bfP^{BB}_{(s,y_i),(t,x)}\bigg[X_r^i \in B_i;\,  \forall j \in [m]: |X_{u_{j}} - \ell_{i}(u_{j})| < \epsilon\bigg]>0.
\end{align*}
\begin{figure}[t]
\includegraphics{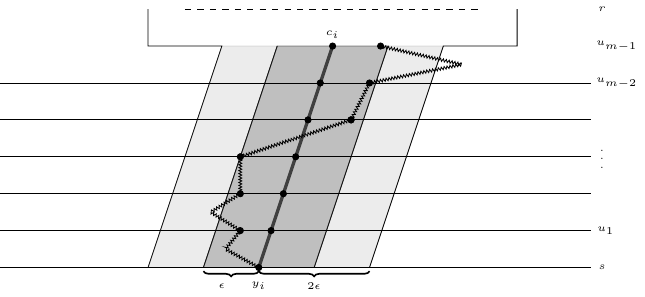}
\caption{\small On $E_i$, the $i^{th}$ path is required to lie within $\epsilon$ of $\ell_{i,u_{m-1}}(u_j)$ at the times $\{u_j : j \in [m-1]\}$. Between these times, the path is required to remain inside the larger cylinder of radius $2\epsilon$ (the union of the light and dark grey regions) around $\ell_{i,v}(\cdot)$ (thick). The points $\ell_{i,u_{m-1}}(u_j)$ and some admissible values of $X_{u_j}^i$ are marked as bullets. An inadmissible path between these values (due to exiting the cylinder between $u_{m-2}$ and $u_{m-1}$)  is drawn as a zigzag. In order to exit in this way, the path between $X_{u_{j-1}}^i$ and $X_{u_{j}}^i$ deviates from the straight line between those points by more than $\epsilon$.}\label{fig:polypf2}
\end{figure}
\begin{figure}[h]
\includegraphics{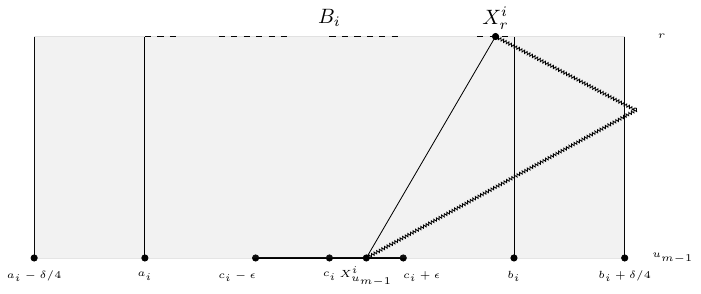}
\caption{\small By Lemma \ref{lem:rndbd}, $X_{u_{m-1}}^i \in (c_i - \epsilon, c_i + \epsilon)$ (solid at level $u_{m-1}$)  and $X_r^i \in B_i$ (dashed, at level $r$) with positive probability. If $\epsilon$ is such that $(c_i - \epsilon, c_i + \epsilon) \subset [a_i,b_i]$, then in order to exit the interval $[a_i-\delta/4,b_i+\delta/4]$, the path must deviate by more than $\delta/4$ from the straight line connecting $X_{u_{m-1}}^i$ and $X_r^i$ on the time interval $[u_{m-1},r]$.}\label{fig:polypf3}
\end{figure}
The strict positivity of the infimum follows because the probability being minimized is positive for each $x$ and continuous in $x$ and the region over which we take the infimum is compact. Continuity can be seen either by applying dominated convergence to the integral form of the probability or abstractly as a consequence of the strong Feller property of Brownian bridge. 
We have
\begin{align*}
&\Polyb_{(s,y_i),(t;\zeta)}(E_i) =  \frac1{\Sheb(t;\zeta\viiva s,y_i)}{ \int_{\bbR}\Sheb(t,x\viiva s,y_i)\Polyb_{(s,y),(t,x)}(E_i) \zeta(dx)}\\[5pt]
&\qquad\geq    \frac{C'\zeta[-K,K]}{(1+2K)^m\Sheb(t;\zeta\viiva s,y_i)}   \\[3pt] 
&\qquad\qquad\quad \times \;  \inf_{x\in[-K,K]}\Big\{\Sheb(t,x\viiva s,y_i) \bfP^{BB}_{(s,y_i),(t,x) }\big[X_r^i \in B_i\,; \forall j \in [m]:  |X_{u_{j}} - \ell_{i}(u_{j})| < \epsilon\big]\Big\}. 
\end{align*}
Now, notice that by construction, we have $|y_i - y_j| > 8 \epsilon$ and $|c_i-c_j|>\delta+ 4\epsilon > 8\epsilon$. Therefore $|\ell_{i,u_{m-1}}(u) - \ell_{j,u_{m-1}}(u)|>8\epsilon$ for all $u \in [s,u_{m-1}]$ and so the tubes in Figure \ref{fig:polypf}, which the event $E_i$ forces the path to remain in, do not intersect up to time $u_{m-1}$. By hypothesis, we also have $a_{i} - b_{i-1} \geq \delta$ and so the tubes also do not intersect on $[u_{m-1},r]$. Consequently, by independence,
\begin{align*}
&\Polyb_{(s,y_1,\dots,y_n),(t;\zeta)}(X^i_r \in B_i, i \in [n],\tau_{s:t}^{(n)}>r)\\ 
&\qquad\geq \Polyb_{(s,y_1,\dots,y_n),(t;\zeta)}(E_1,\dots,E_n)= \prod_{i=1}^n \Polyb_{(s,y_i),(t;\zeta)}\big(E_i\big)>0.
\end{align*}
Now, consider $(a_1,\dots,a_n),(b_1,\dots,b_n)\in\bbW_n$ with $a_{i-1}<b_{i-1}<a_{i}$ for $i \in \{2,\dots,n\}$ and call $B_i = (a_i,b_i)$. The previous result implies that
\begin{align*}
\bigg\{(z_1,\dots,z_n) \in \bbW_n : \det_{1\leq i,j \leq n}\bigg[\frac{\Sheb(t;\zeta|r,z_{i})\Sheb(r,z_{i}\viiva s,y_j)}{\Sheb(t;\zeta\viiva s,y_j)} \bigg] = 0 \bigg\} \bigcap\, (B_1\times\dots,\times B_n)
\end{align*}
has Lebesgue measure zero, with a similar result for the other determinant in \eqref{prop:KMG:pos}. By \eqref{eq:detid}, it follows that for all $(x_1,\dots,x_n)$, $(y_1,\dots,y_n)\in\bbW_n$, the Lebesgue measure of the sets
\begin{align*}
&\bigg\{(z_1,\dots,z_n)\in\bbW_n : \det_{1\leq i,j \leq n}\big[\Sheb(r,z_j\viiva s,y_i)\big] = 0 \bigg\} \text{ and }\\
&\bigg\{(z_1,\dots,z_n)\in\bbW_n : \det_{1\leq i,j \leq n}\big[\Sheb(t,x_j|r,z_i)\big] = 0 \bigg\}
\end{align*}
are both zero. 

For a permutation $\sigma \in S_n$, set  
$\bbW_n^\sigma=\{z_{1:n}\in\R^n:  z_{\sigma(1)}<\dotsm<z_{\sigma(n)}\}$. By the Chapman-Kolmogorov identity in Lemma \ref{lem:ChaKol} and the Cauchy-Binet-Andr\'eief  identity, \cite[Lemma 3.2.3]{And-Gui-Zei-10}), for any $(y_1,\dots,y_n)$, $(x_1,\dots,x_n)\in\bbW_n$ and any $s < r < t$,
\begin{align*}
&\det_{1\leq i,j \leq n} \bigl[ \Sheb(t,y_j\viiva s, x_i)\bigr]
=   \det_{1 \leq i,j \leq n}\biggl[   \int_\R \Sheb(t,y_j\viiva r,z) \tspb  \Sheb(r,z\viiva s, x_i)\tspb dz \biggr] \\[6pt] 
&=  \frac1{n!} \int_{\R^n}   \det_{1\leq i,j \leq n} \bigl[ \Sheb(t,y_j\viiva r, z_i)\bigr]   \tspb \det_{1\leq i,j \leq n} \bigl[ \Sheb(r, z_i\viiva s, x_j)\bigr]  \tspb dz_{1:n}   \\[6pt] 
&=  \frac1{n!} \sum_{\sigma\in S_n}  \int_{\bbW_n^\sigma}   \det_{1\leq i,j \leq n} \bigl[ \Sheb(t,y_j\viiva r, z_i)\bigr]  \tspb \det_{1\leq i,j \leq n} \bigl[ \Sheb(r, z_i\viiva s, x_j)\bigr]  \tspb dz_{1:n}   \\[6pt] 
&=  \frac1{n!} \sum_{\sigma\in S_n}  \int_{\bbW_n^\sigma}   \det_{1 \leq i,j \leq n} \bigl[ \Sheb(t,y_j\viiva r, z_{\sigma(i)})\bigr]   \tspb \det \bigl[ \Sheb(r, z_{\sigma(i)}\viiva s, x_j)\bigr]  \tspb dz_{1:n}    \\[6pt] 
&=  \frac1{n!} \sum_{\sigma\in S_n}  \int_{\bbW_n}   \det_{1 \leq i,j \leq n} \bigl[ \Sheb(t,y_j\viiva r, z_{i})\bigr]   \tspb \det_{1\leq i,j \leq n} \bigl[ \Sheb(r, z_{i}\viiva s, x_j)\bigr] \tspb dz_{1:n}   \\[6pt] 
& =   \int_{\bbW_n}   \det_{1\leq i,j \leq n} \bigl[ \Sheb(t,y_j\viiva r, z_{i})\bigr]   \tspb \det_{1\leq i,j \leq n} \bigl[ \Sheb(r, z_{i}\viiva s, x_j)\bigr] \tspb dz_{1:n} >0,
\end{align*}
by the previous observation.
\end{proof}

\begin{proof}[Proof of Proposition \ref{prop:stochmon}]
It suffices to check stochastic monotonicity for the finite dimensional distributions. By Proposition \ref{prop:KMG} \eqref{prop:KMG:pos}, for all $s<r<t$, all $\beta \in\bbR$, all $\zeta \in \ICM$, all $y_1<y_2$ and all $z_1<z_2$, we have
\begin{align*}
&\frac{\Sheb(t;\zeta|r,z_1)\Sheb(r,z_1\viiva s,y_1)}{\Sheb(t;\zeta\viiva s,y_1)} \frac{\Sheb(t;\zeta|r,z_2)\Sheb(r,z_2\viiva s,y_2)}{\Sheb(t;\zeta\viiva s,y_2)} \\
&\qquad\qquad> \frac{\Sheb(t;\zeta|r,z_1)\Sheb(r,z_1\viiva s,y_2)}{\Sheb(t;\zeta\viiva s,y_2)}\frac{\Sheb(t;\zeta|r,z_2)\Sheb(r,z_2\viiva s,y_1)}{\Sheb(t;\zeta\viiva s,y_1)}
\end{align*}
For any $a \in \bbR$, integrating both sides of this expression on $(-\infty,a)$ with respect to $z_1$ and on $(a,\infty)$ with respect to $z_2$, we have
\begin{align*}
\Polyb_{(s,y_1),(t;\zeta)}\big(X_r < a\big)\Polyb_{(s,y_2),(t;\zeta)}\big(X_r \geq a\big) > \Polyb_{(s,y_2),(t;\zeta)}(X_r < a) \Polyb_{(s,y_1),(t;\zeta)}(X_r\geq a)
\end{align*}
All of the probabilities above here are strictly positive, so we may re-write this as
\begin{align*}
\frac{\Polyb_{(s,y_2),(t;\zeta)}(X_r \geq a) }{1- \Polyb_{(s,y_2),(t;\zeta)}(X_r \geq a)} > \frac{\Polyb_{(s,y_1),(t;\zeta)}(X_r \geq a)}{1-\Polyb_{(s,y_1),(t;\zeta)}(X_r \geq a)},
\end{align*}
which holds if and only if
\begin{align*}
\Polyb_{(s,y_2),(t;\zeta)}(X_r \geq a)>\Polyb_{(s,y_1),(t;\zeta)}(X_r \geq a).
\end{align*}
We prove stochastic monotonicity of the $n$-point distributions by induction. Assume that for all $s<t$, all $\zeta \in \ICM$, all $y_1 < y_2$, all $a_{1:n}\in\R^n$, and all $r_{1:n}\in\bbR^n$ satisfying $s < r_1 < \dots < r_n < t$, we have
\begin{align*}
\Polyb_{(s,y_2),(t;\zeta)}(X_{r_1} \geq a_1,\dots,X_{r_n}>a_{n}) >\Polyb_{(s,y_1),(t;\zeta)}(X_{r_1} \geq a_1,\dots,X_{r_n}\geq a_{n})
\end{align*}
Now, fix $a_{1:n+1}\in\bbR^{n+1}$ and $r_{1:n+1}$ with $r_1 < \dots < r_n < r_{n+1}$. The induction hypothesis implies that for each $y \in \bbR$,
\begin{align*}
z \mapsto \Polyb_{(s,y),(r_{n+1},z)}(X_{r_1} \geq a_1, \dots, X_{r_n} \geq a_n)\one_{\{z \geq a_{n+1}\}}
\end{align*} 
is non-decreasing. We have
\begin{align*}
&\Polyb_{(s,y_2),(t;\zeta)}(X_{r_1} \geq a_1,\dots,X_{r_n}\geq a_{n}, X_{r_{n+1}}\geq a_{n+1}) \\
&= \oE_{(s,y_2),(t;\zeta)}^{\Polyb}[\Polyb_{(s,y_2),(r_{n+1},X_{r_{n+1}})}(X_{r_1} \geq a_1, \dots, X_{r_n} \geq a_n)\ind_{\{X_{r_{n+1}}\geq a_{n+1} \}} ] \\
&\geq \oE_{(s,y_2),(t;\zeta)}^{\Polyb}[\Polyb_{(s,y_1),(r_{n+1},X_{r_{n+1}})}(X_{r_1} \geq a_1, \dots, X_{r_n} \geq a_n)\ind_{\{X_{r_{n+1}}\geq a_{n+1} \}} ] \\
&\geq \oE_{(s,y_1),(t;\zeta)}^{\Polyb}[\Polyb_{(s,y_1),(r_{n+1},X_{r_{n+1}})}(X_{r_1} \geq a_1, \dots, X_{r_n} \geq a_n)\ind_{\{X_{r_{n+1}}\geq a_{n+1} \}} ] \\
&= \Polyb_{(s,y_1),(t;\zeta)}(X_{r_1} \geq a_1,\dots,X_{r_n}\geq a_{n}, X_{r_{n+1}}\geq a_{n+1}),
\end{align*}
where in the first inequality, we apply the induction hypothesis and in the second, we applied the base case of the induction with $r=r_{n+1}$. This implies the first inequality in \eqref{eq:stochmon:1}. The second is similar.
\end{proof}

\appendix

\section{Mild solutions and uniqueness}\label{app:mild}
In this section we discuss some details and partially survey the literature concerning existence and uniqueness of mild solutions to \eqref{eq:SHEf} with possibly random initial conditions. We then show that 
the superposition formulation of $\Sheb(t,x\viiva s;\mu)$ is, up to indistinguishability, the usual mild solution to \eqref{eq:SHEf}.  

Let $X$ be a random variable taking values in the space of positive Borel measures satisfying $\bbP(X \in \ICM)=1$ and which is independent of the white noise after some initial time $s$, which we will typically take to be zero. For such $s$, define $\filt{W,X}_{s,t} =  \wedge_{a < s \leq t < b} \tspc\sigma(\filt{W,0}_{a,b},X)$ to be the augmentation of the natural filtration of the white noise enlarged by $\sigma(X)$. Setting $W_0(g) = 0$ for all $g \in L^2(\bbR)$ and $W_t(g) - W_s(g) = W(f_{s,t,g})$, where $f_{s,t,g}(r,x) = \ind_{[s,t]}(r)g(x)$  and $-\infty < s \leq t < \infty$, it is straightforward to check that $W_t(\aabullet)$ defines an orthogonal martingale measure for $t \geq 0$ in the sense of \cite{Wal-86} with respect to either $\fil_{0,t}$ or $\filt{W,X}_{0,t}$. See the discussion in Chapter 2 of \cite{Wal-86}.

For each $s \in \bbR$, the mild formulation of \eqref{eq:SHEf} seeks fixed points $U$ to the Duhamel equation
\begin{align}
U(t,x) &= \int_{\bbR} \heat(t,x-z) X(dz) + \beta \int_s^t \int_{\bbR} \heat(t-r,x-z) U(r,z)W(dz\,dr)  \label{eq:genmild}
\end{align}
which take values in an appropriate class of functions on $\{(t,x)  \in \bbR^2: t > s\}$, where the stochastic integral is understood in the sense of Walsh \cite{Wal-86}. 

Various hypotheses for existence and conditions for uniqueness of solutions have appeared in the literature. For non-random initial data, the minimal assumption that has been studied is that of \cite{Che-Dal-14,Che-Dal-15}, who assume that
\begin{align}
X \text{ is a non-random measure in }\ICM  .  \label{eq:CDcond}
\end{align}
The first paper to allow for random initial data was \cite{Ber-Gia-97}, who assume that there exists a random variable $X_0 \in \sC(\bbR,\bbR)$ so that for each $p>0$, there exists $a_p>0$ for which
\begin{align}
X(dx) = X_0(x)dx \text{ and } \sup_{r \in \bbR} e^{-a_p|r|}\bbE[|X_0(r)|^p] <\infty.\label{eq:BGcond}
\end{align}
The first paper we are aware of to systematically study mild solutions of \eqref{eq:genmild} was  \cite{Ber-Can-95}. The results in \cite{Ber-Can-95} are stated and proven only under the assumption that for all $t>0$,
\begin{align}
X\text{ is non-random and }\sup_{r \in (0,t]} \sup_{x \in \bbR} \sqrt{r}\bigg(\int_{\bbR} \heat(r,x-z)X(dz)\bigg)^2<\infty.  \label{eq:BCcond1}
\end{align}
We begin by recalling this original result, 
appealing to translation invariance of the model to extend these results from the case of $s=0$ to $s \in \bbR$.
\begin{theorem}{\rm(\cite[Theorem 3.1]{Ber-Can-95})\label{thm:appexist:BC}}
Under Condition \eqref{eq:BCcond1}, for each $s \in \R$ there exists an $(\fil_{s,t} : s \leq t)$ adapted solution $U \in \sC((s,\infty)\times\bbR,\bbR)$ to \eqref{eq:genmild} satisfying  for all $T>s$,
\begin{align}
\sup_{t \in (s,T]} \sup_{x \in \bbR} \int_s^t \int_s^r \int_{\bbR} \int_{\bbR} \heat(t-r,x-y)^2\heat(r-v,y-z)^2\bbE[U(v,z)^2]dzdydvdr <\infty.\label{cond:BC}
\end{align}
Moreover, under \eqref{eq:BCcond1}, if $U$ and $V$ are any two solutions satisfying \eqref{eq:genmild} and \eqref{cond:BC}, then $U$ and $V$ are indistinguishable.
\end{theorem}

The most general existence and uniqueness result for non-random initial data is the following result, which comes from combining results in \cite{Che-Dal-14} and \cite{Che-Dal-15}.
\begin{theorem}{\rm(\cite[Theorem 2.4]{Che-Dal-15} with \cite[Theorem 3.1]{Che-Dal-14})} \label{thm:appexist:CD}
Under Condition \eqref{eq:CDcond}, for each $s \in \R$, there exists an $(\fil_{s,t} : s \leq t)$ adapted solution $U \in \sC((s,\infty)\times\bbR,\bbR)$ to \eqref{eq:genmild} satisfying 
\begin{enumerate} [label={\rm(\roman*)}, ref={\rm\roman*}]   \itemsep=3pt 
\item\label{cond:CD1} For all $t>s$ and $x \in \bbR$, 
\[\int_s^t \int_{\bbR} \heat(t-r,x-z) \bbE[U(r,z)^2]dzdr <\infty\]
\item\label{cond:CD2}  For all $t>s$,
\begin{align*}
&\lim_{\substack{(u,v) \to (t,x) \\ u > s}} \bbE\bigg[\bigg( \int_s^u \int_{\bbR} \heat(u-r,v-z) U(r,z)W(dz\,dr) \\
&\qquad\qquad\qquad- \int_s^t \int_{\bbR} \heat(t-r,x-z) U(r,z)W(dz\,dr) \bigg)^2\bigg] = 0
\end{align*}
\end{enumerate}
Moreover, if $U$ and $V$ are any two solutions satisfying these properties, then $U$ and $V$ are indistinguishable.
\end{theorem}

Both the results in \cite{Che-Dal-14,Che-Dal-15} and \cite{Ber-Can-95} apply to show the existence and uniqueness of a solution to \eqref{eq:Greens} for fixed initial conditions $s,y$, which we record below. 
The claim about the solution being represented by \eqref{eq:SHEchaos} is sketched in Section 3.2 of \cite{Alb-Kha-Qua-14-jsp}. 
A pedagogical proof appears in the lecture notes \cite[Theorem 2.2]{Cor-18}.
\begin{lemma}\label{lem:fixGreen}
For each $s,y \in \bbR$, there exists a unique (up to indistinguishability) $(\filt{W,X}_{s,t} : s \leq t)$-adapted process $\Sheb(\aabullet,\aabullet\viiva s,y)$ taking values in $\sC((s,\infty),\bbR)$ which satisfies the mild equation \eqref{eq:Greens} and the conditions in either Theorem \ref{thm:appexist:BC} or  \ref{thm:appexist:CD}. Moreover, this process is a modification of the process defined for fixed $(s,y,t,x,\beta) \in \varsets\times\bbR$ by \eqref{eq:SHEchaos}.
\end{lemma}

We next turn to the assumption in \eqref{eq:BGcond}, which allows for a class of random initial data which is rich enough to include the exponential of Brownian Motion with drift, which are increment-stationary for KPZ \cite{Fun-Qua-15}. 
\begin{theorem}{\rm\cite[Theorem 3.1]{Ber-Gia-97}\label{thm:appexist:BG}}
 Under Condition \eqref{eq:BGcond}, for each $s \in \R$, there exists an $(\filt{W,X}_{s,t} : s \leq t)$ adapted process $U(t,x)$ taking values in $\sC([s,\infty),\bbR)$ which satisfies the mild equation \eqref{eq:genmild}. For each $T>s$, there exists a constant $a>0$ so that resulting process satisfies
\begin{align}
\sup_{t\in [s,T]} \sup_{x\in\bbR} e^{-a|x|}\bbE[|U(t,x)|^2] <\infty. \label{eq:BGgrowth}
\end{align}
Moreover, if $U$ and $V$ are two such continuous and adapted solutions satisfying \eqref{eq:genmild} and \eqref{eq:BGgrowth}, then $\bbP(U(t,x) = V(t,x) \text{ for all }t>s, x\in \bbR)=1.$
\end{theorem}

We now show that $\Sheb(t,x\viiva s;X)$ defined through \eqref{eq:superpos} agrees with the mild solution in \eqref{eq:genmild} in the previous results.

\begin{lemma}\label{lem:uniq}
Assume that $X$ satisfies either \eqref{eq:CDcond} or \eqref{eq:BGcond}.
For each $s \in \R$, up to indistinguishability, $\Sheb(\aabullet,\aabullet|s;X)$ (as defined by \eqref{eq:superpos}) is the unique continuous and adapted solution to \eqref{eq:genmild} in Theorems \ref{thm:appexist:CD} and \ref{thm:appexist:BG} respectively.
\begin{proof}
Because mild solutions are formulated for fixed initial conditions and fixed $\beta$, scaling and translation invariance implies that it is without loss of generality to consider the case of $\beta=1$ and $s =0$. By construction, $\She_1(t,x|0;X)$ is adapted and continuous, so we just need to check that it solves the mild equation and the moment conditions of the two theorems. To check that $\She_1(t,x|0;X)$ satisfies \eqref{eq:genmild}, we will apply the stochastic Fubini theorem; see \cite[Theorem 4.33]{DaP-Zab-14} or \cite[Theorem 2.6]{Wal-86}. This result shows that whenever

\begin{align*}
&\int_{0}^t\int_{\bbR} \int_{\bbR}  \heat(t-r,x-z)^2\bbE[\She_1(r,z|0,y)^2] X(dy) dzdr< \infty, \text{ if }\eqref{eq:CDcond} \text{ holds or} \\
&\int_{0}^t\int_{\bbR} \int_{\bbR}  \heat(t-r,x-z)^2\bbE[\She_1(r,z|0,y)^2] \bbE[X_0(y)^2] dy dzdr < \infty \text{ if }\eqref{eq:BGcond} \text{ holds,}
\end{align*}
we have
\begin{align*}
\She_1(t,x|0;X) &= \int_{\bbR}\She_1(t,x|0,y)X(dy) \\
&= \int_{\bbR}\heat(t,x-y)X(dy) + \int_0^t \int_{\bbR}  \heat(t-r,x-z) \int_{\bbR} \She_1(r,z|0,y)X(dy) W(dzdr) \\
&=\int_{\bbR}\heat(t,x-y)X(dy) + \int_0^t \int_{\bbR}\heat(t-r,x-z) \She_1(r,z|0;X)W(dzdr).
\end{align*}
Recall that we have $\bbE[\She_1(r,z|0,y)^2] \leq \Mom{2}{t} \heat(t,z-y)^2$. Considering that $\bbE[X_0(y)^2] dy \in \ICM$, it suffices to consider the first integral under \eqref{eq:CDcond}. By Lemma \ref{lem:int1},
\begin{align*}
\int_{\bbR}\int_{0}^t \int_{\bbR}  \heat(t-r,x-z)^2\heat(r,z-y)^2 dzdr X(dy) &= \frac{\sqrt{\pi t}}{2}\int_{\bbR} \heat(t,x-y)
X(dy),
\end{align*}
which is finite by hypothesis. It remains to check the moment hypotheses. Call $V(t,x) = \She_1(t,x|0,X)$, so that $V(t,x)$ solves \eqref{eq:genmild}. Under \eqref{eq:BGcond}, by Cauchy-Schwarz applied twice
\begin{align*}
&\bbE[V(t,x)^2] = \int_{\bbR}\int_{\bbR}\bbE\bigg[\She_1(t,x|0,z)\She_1(t,x|0,w)\bigg]\bbE[X_0(w)X_0(z)] dzdw \\
&\leq C \Mom{2}{t}\int_{\bbR}\int_{\bbR} \heat(t,x-z)\heat(t,x-w)e^{\frac{a_2}{2}(|w|+|z|)}dwdz \\
&\leq C \Mom{2}{t}\bigg(\int_{\bbR}\heat(t,x-z)(e^{\frac{a_2}{2}z} + e^{-\frac{a_2}{2} z})dz\bigg)^2 =C  \Mom{2}{t} \bigg(e^{\frac{a_2}{2}x + \frac{a_2^2}{8}t} +e^{-\frac{a_2}{2}x + \frac{a_2^2}{8}t}\bigg)^2,
\end{align*} 
where $C$ is the value of the supremum appearing in \eqref{eq:BGcond}. This verifies \eqref{eq:BGgrowth}. 

Next, we turn to showing conditions \eqref{cond:CD1} and \eqref{cond:CD2} in Theorem \ref{thm:appexist:CD} under \eqref{eq:BCcond1}. We have
\begin{align*}
&\int_0^t \int_{\bbR} \heat(t-r,x-z) \bbE[V(r,z)^2]dzdr \\
&\leq \Mom{2}{t}\int_0^{t} \int_{\bbR}  \heat(t-r,x-z)\bigg( \int_{\bbR}\heat(r,z-y)X(dy)\bigg)^2dzdr \\
&=\int_{\bbR}\int_{\bbR} \int_0^t \frac{ \Mom{2}{t}}{(2\pi)^{3/2}r\sqrt{t-r}} \int_{\bbR}e^{-\bigg[\frac{(x-z)^2}{2(t-r)}+\frac{(z-v)^2+(z-w)^2}{2r}\bigg]}dzdr X(dv)X(dw)\\
&=\int_{\bbR}\int_{\bbR} \int_0^t \frac{\Mom{2}{t}}{(2\pi)^{3/2}r\sqrt{t-r}} \int_{\bbR}e^{-\frac{z^2}{2}\big(\frac{1}{t-r}+\frac{2}{r}\big)}e^{z\big(\frac{x}{t-r} + \frac{v+w}{r}\big)}dz e^{-\big(\frac{x^2}{2(t-r)} + \frac{v^2+w^2}{2r}\big)}dr X(dv)X(dw)\\
&= \int_{\bbR}\int_{\bbR} \int_0^t \frac{\Mom{2}{t}}{(2\pi)r\sqrt{t-r}\sqrt{\frac{1}{t-r} + \frac{2}{r}}} e^{\frac{\big(\frac{x}{t-r} + \frac{v+w}{r}\big)^2}{2\big(\frac{1}{t-r} + \frac{2}{r}\big)} -\big(\frac{x^2}{2(t-r)} + \frac{v^2+w^2}{2r}\big)}dr X(dv)X(dw) \\
&=\int_{\bbR}\int_{\bbR} \int_0^t\frac{\Mom{2}{t}}{2\pi\sqrt{r(r+2(t-r))}}e^{-\frac{x^2}{2t-r} + \frac{x(v+w)}{2t-r} - \frac{(t-r)(v-w)^2 + r (v^2+w^2)}{2r(2t-r)} } drX(dv)X(dw)\\
&\leq e^{-\frac{x^2}{t}} \int_0^t\frac{\Mom{2}{t}}{2\pi\sqrt{r(r+2(t-r))}}dr \bigg(\int_{\bbR} e^{\frac{2|x| |v|- v^2}{2t}}X(dv)\bigg)^2 < \infty.
\end{align*}
In the last step, we bounded $|x(v+w)| \leq |x|(|v|+|w|)$, $-(t-r)(v-w)^2 \leq 0$, and then $t \leq 2t-r$. \eqref{cond:CD1} follows.

Still considering the case of \eqref{eq:CDcond}, we now verify condition \eqref{cond:CD2}. For simplicity, we take the case of $u=t+h$, $h>0$, $h \to 0$. The case of $u=t-h$ follows similarly.
\begin{align*}
&\bbE\bigg[\bigg( \int_0^{u} \int_{\bbR} \heat(u-r,v-z) V(r,z)W(dz\,dr) - \int_0^t \int_{\bbR} \heat(t-r,x-z) V(r,z)W(dz\,dr) \bigg)^2\bigg]\\
&= \bbE\bigg[\bigg( \int_0^{t} \int_{\bbR} \big(\heat(t+h-r,v-z) - \heat(t-r,x-z)\big) V(r,z)W(dz\,dr) \bigg)^2\bigg]\\
&\qquad+\bbE\bigg[\bigg( \int_t^{t+h} \int_{\bbR} \heat(t+h-r,v-z) V(r,z)W(dz\,dr)\bigg)^2\bigg] \\
&\leq \Mom{2}{2t}\int_{\bbR} \int_{\bbR} \big(\heat(t+h-r,v-z) - \heat(t-r,x-z)\big)^2 \bigg(\int_{\bbR}\heat(r,z-w)X(dw)\bigg)^2\one_{(0,t)}(r)dz\,dr   \\
&\qquad\qquad+ \Mom{2}{2t} \int_{\bbR}\int_{\bbR}\heat(t+h-r,v-z)^2 \bigg(\int_{\bbR}\heat(r,z-w)X(dw)\bigg)^2\one_{(t,t+h)}(r)dz\,dr
\end{align*}
Note that the integrands in both integrals above converge to zero pointwise. We bound
\begin{align}
&\big(\heat(t+h-r,v-z) - \heat(t-r,x-z)\big)^2 \bigg(\int_{\bbR}\heat(r,z-w)X(dw)\bigg)^2\one_{(0,t)}(r)\notag\\
& \leq 2 \heat(t+h-r,v-z)^2  \bigg(\int_{\bbR}\heat(r,z-w)X(dw)\bigg)^2\one_{(0,t)}(r)\label{eq:gdc1}\\
&+ 2 \heat(t-r,x-z)^2 \bigg(\int_{\bbR}\heat(r,z-w)X(dw)\bigg)^2\one_{(0,t)}(r) \label{eq:gdc2}
\end{align}
By the generalized dominated convergence theorem \cite[Theorem 4.17]{Roy-88}, to show that the first integral converges to zero, it suffices to show that
\begin{align*}
&\lim_{\substack{h \searrow 0 \\ v\to x}} \int_{\bbR}\int_{\bbR}\heat(t+h-r,v-z)^2  \bigg(\int_{\bbR}\heat(r,z-w)X(dw)\bigg)^2\one_{(0,t)}(r) dz\,dr \\
&= \int_{\bbR}\int_{\bbR}\heat(t-r,x-z)^2  \bigg(\int_{\bbR}\heat(r,z-w)X(dw)\bigg)^2\one_{(0,t)}(r) dz\,dr < \infty
\end{align*}
Note that there are two claims here: the  integral in \eqref{eq:gdc2} does not depend on $h$ or $v$ and is the limit of the integral in \eqref{eq:gdc1}  as $h \searrow 0$ and $v \to x$. We must show that the integral in \eqref{eq:gdc2} is finite and that the integral in \eqref{eq:gdc1} converges to it. Arguing as above, we compute for $h \geq 0$ (now, including the case of \eqref{eq:gdc2}),
\begin{align*}
&\int_{\bbR}\int_{\bbR}\heat(t+h-r,v-z)^2  \bigg(\int_{\bbR}\heat(r,z-w)X(dw)\bigg)^2\one_{(0,t)}(r) dz\,dr \\
&= \int_{\bbR}\int_{\bbR} \int_0^t \frac{(2\pi)^{-2}}{(t+h-r)r}\int_{\bbR} e^{-\frac{z^2}{2}\big(\frac{2}{t+h-r}+\frac{2}{r}\big)}e^{z\big(\frac{2v}{t+h-r} + \frac{y+w}{r}\big)}dz e^{-\big(\frac{2v^2}{t+h-r} + \frac{y^2+w^2}{2r}\big)}dr\, X(dy)\,X(dw)\\
&= \int_{\bbR}\int_{\bbR} \int_0^t \frac{(2\pi)^{-{3/2}}}{(t+h-r)r\sqrt{\frac{2}{t+h-r} + \frac{2}{r}}} e^{\frac{\big(\frac{2v}{t+h-r} + \frac{y+w}{r}\big)^2}{2\big(\frac{2}{t+h-r}+ \frac{2}{r}\big)} } e^{-\big(\frac{2v^2}{t+h-r} + \frac{y^2+w^2}{2r}\big)}dr\, X(dy)\,X(dw)\\
&= \int_{\bbR}\int_{\bbR} \int_0^t \frac{(2\pi)^{-{3/2}}}{\sqrt{2(t+h)r(t+h-r)} } e^{-v^2 \big(\frac{2}{t+h} + \frac{r}{(t+h)(t+h-r)}\big) + \frac{v(y+w)}{t+h}-\frac{(w-y)^2}{4r}-\frac{(w+y)^2}{4(t+h)}}dr\, X(dy)\,X(dw)
\end{align*}
The integrand is continuous in $h$ and $v$, so to show that this integral converges, we may use the ordinary dominated convergence theorem. For any $h \in [0,t]$ and $v \in [-K,K]$, $r \in (0,t)$ and $y,w \in \bbR$, we have
\begin{align*}
& \frac{(2\pi)^{-{3/2}}}{\sqrt{2(t+h)r(t+h-r)} } e^{-v^2 \big(\frac{2}{t+h} + \frac{r}{(t+h)(t+h-r)}\big) + \frac{v(y+w)}{t+h}-\frac{(y-w)^2(t+h-r)}{4r(t+h)} - \frac{y^2+w^2}{4(t+h)}} \\
&\leq  \frac{1}{\sqrt{t\cdot r\cdot(t-r)} } e^{\frac{|K|(|y|+|w|)}{t}- \frac{y^2+w^2}{8t}},
\end{align*}
and 
\begin{align*}
\int_{\bbR}\int_{\bbR}\int_0^t \frac{1}{\sqrt{t\cdot r\cdot(t-r)} } e^{\frac{|K|(|y|+|w|)}{t}- \frac{y^2+w^2}{8t}} dr\,X(dy)\,X(dw) < \infty.
\end{align*}
The result follows. It remains to show that
\begin{align*}
\lim_{\substack{h \searrow 0\\ v \to x}}\int_{\bbR}\int_{\bbR}\heat(t+h-r,v-z)^2 \bigg(\int_{\bbR}\heat(r,z-w)X(dw)\bigg)^2\one_{(t,t+h)}(r)dz\,dr = 0,
\end{align*}
which follows from essentially the same estimates as in the previous case. \eqref{cond:CD2} now follows.
\end{proof}
\end{lemma}

\section{Continuity of stochastic processes}\label{app:KC}

This appendix presents a version of the Kolmogorov-Chentsov theorem  sufficient for our purposes. For $d \geq 2$, we consider a process $X$ with  values in a  complete separable  metric space $(S, \varrho)$ and 
indexed by   $d-2$ copies of the  unit interval $[0,1]$ and one copy of the time-ordered unit triangle $\bfT_\delta = \{(s^1,s^2) : 0 \leq s^1 \leq s^1+\delta\le  s^2 \leq 1\}$ with a gap $\delta\in[0,1)$.  The case $\delta=0$  corresponds to enforcing only the ordering $s^1\le  s^2$.  Generic points of $[0,1]^{d-2} \times \bfT_\delta$ are denoted by $r=(r^1,\dots,r^{d-2},r^{d-1},r^{d})$, with superscripts for coordinates. The last two coordinates satisfy $r^{d-1} \leq r^{d}-\delta$. Subscripts are reserved for indexing sequences  in $[0,1]^{d-2}\times \bfT_\delta$.

 For $n\in\Z_+$ let 
\[  D_{\delta,n}=\bigl\{ (k^1,\dotsc, k^d)2^{-n}\in [0,1]^{d-2} \times \bfT_\delta:   k^1,\dotsc, k^d \in \lzb0,2^n\rzb
 \bigr\} \]
and then   $D_\delta=\bigcup_{n\in\Z_+} D_{\delta,n}$, the set 
of {\sl dyadic rational points}  in $[0,1]^{d-2} \times \bfT_\delta$.
   
\begin{theorem}\label{thm:KC}   Fix $d\ge 2$ and $\delta\in[0,1)$  as above and let $(S, \varrho)$ be a complete separable metric space. 

\smallskip 

{\rm (a)}   
Suppose $\{X_{r}:r\in D_\delta\}$ is an $S$-valued  stochastic 
process defined on a complete probability space
$(\Omega, \cF, \bbP)$ 
 with the following property: there exist constants
$\constkc<\infty$ 
and $\alpha_{1},\dotsc, \alpha_d,  \nu>0$ such that 
\be
\bbE\bigl[ \varrho(X_s,X_r)^\nu\,\bigr]\le \constkc\sum_{i=1}^d \abs{s^i-r^i}^{d+\alpha_i}
\quad
\text{for all $r,s\in D_\delta$.}
\label{kc-hyp1}
\ee
Then there exists an $S$-valued process $\{Y_s: s\in[0,1]^{d-2}\times \bfT_\delta\}$ on $(\Omega,\sF,\bbP)$ such that 
the path $s\mapsto Y_s(\w)$ is continuous for each $\w\in\Omega$ and 
$P\{Y_s=X_s\}=1$  for each $s\in D_\delta$.   Furthermore,  for all choices of 
  $\aconstkc_i\in(0,\alpha_i/\nu)$ for $i\in[d]$,
\be \label{kolm-c-4}
E\biggl[ \tspb\sup_{r\ne s \; \in [0,1]^{d-2}\times \bfT_\delta} \,  \biggl\lvert \frac{\varrho(Y_s(\omega),Y_r(\omega))}{\sum_{i=1}^d\abs{s^i-r^i}^{\aconstkc_i}} \biggr\rvert^\nu \,\biggr] \le \constkc  \left(\sum_{i=1}^d  \frac{2^{\aconstkc_i+1}}{(1- 2^{\aconstkc_i-\alpha_i/\nu}) (1-2^{-\alpha_i/\nu})}\right)^\nu<\infty.  
\ee
\medskip 

{\rm (b)}  Suppose $\{X_s: s\in [0,1]^{d-2}\times \bfT_\delta\}$ is an $S$-valued  stochastic 
process that satisfies the moment bound \eqref{kc-hyp1} for all $r,s\in[0,1]^{d-2}\times\bfT_\delta$.
Then the process $Y$ of part {\rm(a)} is  a version of  $X$.  If $X$ is almost surely continuous to begin with, then $P\{\tspa Y_s=X_s\, \forall s\in [0,1]^{d-2}\times\bfT_\delta\tspa \}=1$.  
\end{theorem} 
The proof is a standard chaining argument, which we omit. See, for example, Section 2.4 of \cite{Kho-02} for a similar proof.

\medskip

\section{Computations}\label{app:comp}
The following sequence of lemmas present elementary computations that  go into our H\"older regularity estimates. We omit the details of standard identities.

\begin{lemma}\label{lem:int2}
For $t>0$ and $x \in \bbR$,
\begin{align*}
\int_{\bbR} \frac{\heat(t-r,x-z )^2\heat(r,z)^2}{\heat^2(t,x)} dz = \frac{\sqrt{t}}{2\sqrt{\pi(t-r)r}} \ind_{(0,t)}(r)
\end{align*}
\end{lemma}

\begin{lemma}\label{lem:int1}
For $0 < t$ and $x \in \bbR$,
\begin{align*}
\int_{0}^t  \int_{\bbR} \frac{\heat(t-r,x-z )^2\heat(r,z)^2}{\heat(t,x)^2} dz\,dr &=  \frac{\sqrt{t\pi}}{2}
\end{align*}
\end{lemma}

\begin{lemma}\label{lem:int3}
For $t,h>0$ and $x \in \bbR$,
\begin{align*}
\int_0^t  \int_{\bbR} \bigg[ \frac{\heat(t+h-r,x-z )^2\heat(r,z)^2}{\heat^2(t+h,x)} \bigg] dz\,dr=\frac{\sqrt{t+h}}{2\sqrt{\pi}} \bigg[\arcsin\left(1-\frac{2h}{t+h}\right) + \frac{\pi}{2} \bigg] 
\end{align*}
\end{lemma}

\begin{lemma}\label{lem:int4bd}
For $t,h>0$ and $x \in \bbR$,
\begin{align*}
\int_t^{t+h}  \int_{\bbR}  \bigg[ \frac{\heat(t+h-r,x-z )^2\heat(r,z)^2}{\heat^2(t+h,x)} \bigg] dz\,dr &= \frac{\sqrt{t+h}}{2\sqrt{\pi}} \bigg[ \frac{\pi}{2} - \arcsin\left( 1-\frac{2h}{t+h} \right) \bigg] \leq 4 \sqrt{h}.
\end{align*}
\end{lemma}
\begin{proof}
By Lemma \ref{lem:int2},
\begin{align*}
&\int_t^{t+h} \int_{\bbR} \frac{\heat(t+h-r,x-z )^2\heat(r,z)^2}{\heat^2(t+h,x)} dz\,dr = \int_t^{t+h}  \frac{\sqrt{t+h}}{2\sqrt{\pi}\sqrt{(t+h-r)r}} dr\\
&= \frac{1}{2\sqrt{\pi}} \int_t^{t+h}dr \frac{\sqrt{t+h}}{ \sqrt{(t+h-r)r}} = \frac{\sqrt{t+h}}{2\sqrt{\pi}} \int_{\frac{t}{t+h}}^1 du\frac{1}{\sqrt{u(1-u)}} \\
&= \frac{\sqrt{t+h}}{\sqrt{\pi}} \int_{\frac{t}{t+h}}^1 du \frac{1}{\sqrt{1- (2u-1)^2}} = \frac{\sqrt{t+h}}{2\sqrt{\pi}} \int_{\frac{t-h}{t+h} }^1 dv \frac{1}{\sqrt{1- v^2}} \\
&= \frac{\sqrt{t+h}}{2\sqrt{\pi}} \bigg[ \frac{\pi}{2} - \arcsin\left(\frac{t-h}{t+h} \right) \bigg]
\end{align*}
For the inequality, we observe that for $x \in (0,2)$, $\frac{\pi}{2} - \arcsin\left(1-x\right) < 8 \sqrt{x}$. To see this, write $\frac{\pi}{2} - \arcsin\left(1-x\right) = \arccos(1-x)=2\arcsin(\sqrt{x/2})$, where the last equality is the double angle formula. Then use strict convexity of $\arcsin$ on $[0,1]$ $\arcsin(0)=0$, and $\arcsin(1) = \pi/2$ to see that $2\arcsin(\sqrt{x/2}) < 2(\pi/2)\sqrt{x/2}<8\sqrt{x}$ for $x\in(0,2)$. 

Writing $(t-h)/(t+h) = 1-2h/(t+h)$, we have
\begin{align*}
\frac{\sqrt{t+h}}{2\sqrt{\pi}} \bigg[ \frac{\pi}{2} - \arcsin\left(1 - \frac{2h}{t+h} \right) \bigg] \leq  \frac{4\sqrt{2 h}}{\sqrt{\pi}} \leq 4 \sqrt{h}
\end{align*}
\end{proof}

\begin{lemma}\label{lem:xybd}
For $0  < t$ and $x,y \in \bbR$,
\begin{align*}
 \frac{\sqrt{\pi t}}{2} - \frac{|x-y|}{2}\leq  \int_0^t  \int_{\bbR}  \bigg[ \frac{\heat(t-r,y-z )\heat(r,z)}{\heat(t,y)} \frac{\heat(t-r,x-z) \heat(r,z)}{\heat(t,x)} \bigg] dr dz \leq  \frac{\sqrt{\pi t}}{2}
 \end{align*}
\end{lemma}
\begin{proof}
Notice that for $r \in (0,t)$,
\begin{align*}
\frac{(y-z)^2 + (z-x)^2}{2(t-r)} + \frac{z^2}{r} &= \left(\frac{t}{r(t-r)}\right) \left(z - \frac{r(x+y)}{2t}\right)^2 + \frac{x^2}{2t} +  \frac{y^2}{2t} + \frac{r(x-y)^2}{4t(t-r)}.
\end{align*}
Therefore, after changing variables, in the second equality,
\begin{align*}
& \int_{\bbR} \bigg[ \frac{\heat(t-r,y-z )\heat(r,z)}{ \heat(t,y)} \frac{\heat(t-r,x-z) \heat(r,z)}{\heat(t,x)} \bigg]dz \\
&= \frac{1}{(2\pi)} \frac{t}{r(t-r)}  \int_{\bbR} e^{ \frac{x^2+y^2}{2t} - \bigg[\frac{(y-z)^2 + (z-x)^2}{2(t-r)} + \frac{z^2}{r}\bigg] } dz \\
&= \frac{1}{(2\pi)} \frac{t}{r(t-r)} e^{- \frac{(x-y)^2r}{4t(t-r)}}  \int_{\bbR} e^{ - \frac{t}{r(t-r)}z^2 }dz  = \frac{1}{2\sqrt{\pi}}\sqrt{\frac{t}{r(t-r)}} e^{- \frac{(x-y)^2r}{4t(t-r)}} .
\end{align*}
To compute the $dr$ integral, we now substitute $s = \frac{r}{t(t-r)}$, which satisfies $r = \frac{st^2}{1+st}$. 
\begin{align*}
dr=\frac{t^2}{(1+st)^2}ds, \qquad \frac{t}{r(t-r)} = \frac{(1+st)^2}{st^2}, \qquad dr \sqrt{\frac{t}{r(t-r)}} = ds \frac{t}{\sqrt{s}(1+st)}
\end{align*}
\begin{align*}
\int_0^t   \frac{1}{2\sqrt{\pi}}\sqrt{\frac{t}{r(t-r)}} e^{- \frac{(x-y)^2r}{4t(t-r)}} dr&=  \frac{t}{2\sqrt{\pi}} \int_0^\infty  \frac{1}{\sqrt{s}(1+st)} e^{-\frac{(x-y)^2}{4}s} ds \\
\end{align*}
For $\alpha \geq 0$, call
\begin{align*}
I(\alpha) &= \frac{t}{2\sqrt{\pi}} \int_0^\infty ds \frac{1}{\sqrt{s}(1+st)} e^{- \alpha s}.
\end{align*}
We have $I(0) = \frac{\sqrt{\pi t}}{2}$ and
\begin{align*}
I'(\alpha) &= \frac{-t}{2\sqrt{\pi}}\int_0^\infty ds \frac{\sqrt{s}}{1+st} e^{- \alpha s}
\end{align*}
Substitute $u= \alpha s$ so that 
\begin{align*}
\frac{du}{\alpha} = ds, \qquad \frac{\sqrt{s}}{1+st} = \frac{\sqrt{\alpha} \sqrt{u}}{\alpha + ut}, \qquad ds \frac{\sqrt{s}}{1+st} = du\frac{1}{\sqrt{\alpha}} \frac{\sqrt{u}}{\alpha+ u t}.
\end{align*}
Then
\begin{align*}
I'(\alpha) &= \frac{-t}{2\sqrt{\pi}}\int_0^\infty du\frac{1}{\sqrt{\alpha}} \frac{\sqrt{u}}{\alpha+ u t}e^{- u} \geq \frac{-1}{2\sqrt{\pi}} \frac{1}{\sqrt{\alpha}}\int_0^\infty du \frac{1}{\sqrt{u}}e^{- u} = -\frac{1}{2\sqrt{\alpha}}
\end{align*}
It follows that
\begin{align*}
\frac{t}{2\sqrt{\pi}} \int_0^\infty ds \frac{1}{\sqrt{s}(1+st)} e^{-\frac{(x-y)^2}{4}s}  = I\left(\frac{(x-y)^2}{4}\right) &\geq \frac{\sqrt{\pi t}}{2} - \frac{|x-y|}{2}. \qedhere
\end{align*}
\end{proof}
The next result follows from the previous result by expanding out the square.
\begin{corollary}
For $0 < t$ and $x,y \in \bbR$,
\begin{align*}
\int_0^t  \int_{\bbR} \bigg[ \frac{\heat(t-r,y-z )\heat(r,z)}{\heat(t,y)}  -  \frac{\heat(t-r,x-z) \heat(r,z)}{\heat(t,x)} \bigg]^2dz dr \leq  |x-y|.
\end{align*}
\end{corollary}
Next, we turn to the case where the heat kernels have different time coordinates, but the same space coordinate. We begin by computing the space integral of the cross-term that will appear when we expand out the square:
\begin{lemma}\label{lem:diffhspace}
For $0 < t$, $h>0$, and $x \in \bbR$,
\begin{align*}
&\int_{\bbR} \bigg[ \frac{\heat(t+h-r,x-z )\heat(r,z)}{\heat(t+h,x)} \frac{\heat(t-r,x-z) \heat(r,z)}{\heat(t,x)} \bigg] dz =  \\
& \frac{\sqrt{t(t+h)}}{\sqrt{2\pi r((t+h)(t-r) + t(t+h-r))}} \bigg[ e^{- \frac{x^2}{2t}\frac{h^2 r}{(t+h)((t+h)(t-r) + t(t+h-r)))} }\bigg]\ind_{(0,t)}(r)
\end{align*}
\end{lemma}
\begin{proof}
Write
\begin{align*}
&\int_{\bbR} \bigg[ \frac{\heat(t+h-r,x-z )\heat(r,z)}{\heat(t+h,x)} \frac{\heat(t-r,x-z) \heat(r,z)}{\heat(t,x)} \bigg] dz \\
&=  \frac{1}{(2\pi)} \sqrt{\frac{t+h}{(t+h-r)r}} \sqrt{\frac{t}{(t-r)r}} \int_{\bbR}  \bigg[ e^{\frac{x^2}{2(t+h)} + \frac{x^2}{2t} - \big(\frac{(x-z)^2}{2(t+h-r)} + \frac{(x-z)^2}{2(t-r)} + \frac{z^2}{r}  \big)}\bigg]dz\,\ind_{(0,t)}(r)
\end{align*}
We note that
\begin{align*}
&\bigg[\frac{1}{2(t+h-r)} + \frac{1}{2(t-r)} \bigg](x-z)^2 + \frac{1}{r} z^2 \\
&= \left( \frac{1}{2(t-r)} + \frac{1}{2(t+h-r)} + \frac{1}{r}\right)\left(z - \frac{\left(\frac{1}{2(t-r)} + \frac{1}{2(t+h-r)} \right) x}{ \frac{1}{2(t-r)} + \frac{1}{2(t+h-r)} + \frac{1}{r}}\right)^2 + \frac{ \left(\frac{1}{2(t-r)} + \frac{1}{2(t+h-r)}\right)\frac{1}{r}} {\frac{1}{2(t-r)} + \frac{1}{2(t+h-r)} + \frac{1}{r}} x^2.
\end{align*}
We have
\begin{align*}
&\frac{1}{2(t+h-r)} +  \frac{1}{2(t-r)} + \frac{1}{r}  
= \frac{(t+h)(t-r) + t(t+h-r)}{2(t+h-r)(t-r)r},\\
\end{align*}

It follows that 
\begin{align*}
&\int_{\bbR} \bigg[ \frac{\heat(t+h-r,x-z )\heat(r,z)}{\heat(t+h,x)} \frac{\heat(t-r,x-z) \heat(r,z)}{\heat(t,x)} \bigg] dz \\
&= \frac{1}{2\pi} \sqrt{\frac{t+h}{(t+h-r)r}} \sqrt{\frac{t}{(t-r)r}}  \sqrt{\frac{2\pi (t+h-r)(t-r)r}{(t+h)(t-r) + t(t+h-r)}} \times \\
&\qquad\qquad\qquad\bigg[ e^{\frac{x^2}{2(t+h)} + \frac{x^2}{2t} -   \frac{ \left(\frac{1}{2(t-r)} + \frac{1}{2(t+h-r)}\right)\frac{1}{r}} {\frac{1}{2(t-r)} + \frac{1}{2(t+h-r)} + \frac{1}{r}} x^2 }\bigg] \ind_{(0,t)}(r). 
\end{align*}
The remainder of the claim is tedious but easy algebra.
\end{proof}
The next result is the point where our results become suboptimal. A more refined analysis is likely possible to improve this estimate if one is interested in optimal H\"older regularity at the boundary, but it suffices for our purposes.
\begin{lemma}\label{lem:withgap}
For $\delta < 1$, $0 < \delta \leq t \leq t + h \leq T$ with $T>1$, and $x\in \bbR$,
\begin{align*}
&\frac{\sqrt{\pi t}}{2} - \frac{T}{\delta \sqrt{\pi}} \sqrt{h} - h \frac{\sqrt{\pi} T^{3/2}}{8\delta^3}x^2 \\
&\qquad \qquad\leq \int_0^t \int_{\bbR}\frac{\heat(t+h-r,x-z )\heat(r,z)}{\heat(t+h,x)} \frac{\heat(t-r,x-z) \heat(r,z)}{\heat(t,x)} dz dr \leq \frac{\sqrt{\pi t}}{2}.
\end{align*}
\end{lemma}
\begin{proof}
Applying Lemma \ref{lem:diffhspace}, we define for $t,x$ as in the statement and $h \in [0,T-t]$,
\begin{align*}
I(h) &= \int_0^t \int_{\bbR}  \bigg[\frac{\heat(t+h-r,x-z )\heat(r,z)}{\heat(t+h,x)} \frac{\heat(t-r,x-z) \heat(r,z)}{\heat(t,x)}\bigg] dz dr  \\
&=  \frac{\sqrt{t}}{\sqrt{2\pi}} \int_0^t \frac{\sqrt{t+h}}{\sqrt{r((t+h)(t-r) + t(t+h-r))}} \bigg[ e^{- \frac{x^2}{2t}\frac{h^2 r}{(t+h)((t+h)(t-r) + t(t+h-r))) } }\bigg] dr
\end{align*}
Without loss of generality, take $T>t$. Notice that for all $h>0$,
\begin{align*}
I(h) \leq I(0) =  \frac{\sqrt{t}}{2\sqrt{\pi}} \int_0^t  \frac{1}{\sqrt{r(t-r)}} dr = \frac{\sqrt{\pi t}}{2}.
\end{align*}
To see this, notice that the term in the exponential is negative if $h>0$ and that
\begin{align*}
\frac{\frac{d}{dh} \sqrt{\frac{t+h}{r ((t+h)(t-r) + t(t+h-r))}}}{\sqrt{\frac{t+h}{r((t+h)(t-r) + t(t+h-r))}}} = - \frac{rt}{2(t+h)((t+h)(t-r) + t(t+h-r) ) } <0
\end{align*}

Notice that $0 < r < t$ implies that $(t+h)((t+h)(t-r) + t(t+h-r))) \geq t(t+h)h \geq \delta^2 h$ Therefore,
\begin{align*}
I(h) \geq  \frac{1}{2\sqrt{\pi}} \int_0^{ t}  \frac{\sqrt{t+h}}{ \sqrt{r(t+h-r)}} e^{-\frac{x^2 h}{2\delta^3}  r} dr = J(h).
\end{align*}
By dominated convergence, $J(0) = I(0) = \frac{\sqrt{\pi t}}{2}$. Differentiating under the integral, for $h>0$,
\begin{align*}
J'(h) &=  -  \frac{1}{2\sqrt{\pi}}  \int_0^{t}  \bigg[\frac{r}{2(t+h)(t+h-r)} + \frac{x^2 r}{2\delta^3} \bigg] \sqrt{\frac{t+h}{r(t+h-r)}} e^{-\frac{x^2 h}{2\delta^3}  r} dr\\
&\geq-  \frac{1}{2\sqrt{\pi}}  \int_0^{t} \bigg[\frac{r}{2 \delta(t+h-r)} + \frac{x^2 r}{2\delta^3} \bigg] \sqrt{\frac{T}{r(t+h-r)}}dr.
\end{align*}
We have
\begin{align*}
\int_0^tdr \frac{\sqrt{r}}{(t+h-r)^{3/2}} &= 2\sqrt{\frac{t}{h}} - 2 \arcsin\left(\sqrt{\frac{t}{t+h}}\right),\\
 \int_0^t dr \sqrt{\frac{r}{t+h-r}} &= - \sqrt{ht} + (t+h) \arcsin \left( \sqrt{\frac{t}{t+h}}\right)
\end{align*}
Recognizing that $\sqrt{ht}>0$, $\arcsin(\sqrt{\frac{t}{t+h}}) >0$, and $0 < (t+h)\arcsin(\sqrt{\frac{t}{t+h}}) \leq \frac{\pi }{2} T$, we see that
\begin{align*}
J'(h) \geq - \frac{1}{2\sqrt{\pi}}\bigg[ \frac{T}{\delta} \frac{1}{\sqrt{h}}  + \frac{x^2}{4 \delta^3} T^{3/2} \pi \bigg]
\end{align*}
and therefore
\begin{align*}
J(h) &\geq \frac{\sqrt{\pi t}}{2} - \frac{T}{\delta \sqrt{\pi}} \sqrt{h} - h \frac{\sqrt{\pi} T^{3/2}}{8\delta^3}x^2. \qedhere
\end{align*}
\end{proof}

\begin{lemma}\label{lem:near0}
For $h \in [0,1]$, $\alpha \in (0,1]$, $t \in [0,h^\alpha]$, and $x \in \bbR$,
\begin{align*}
\int_0^t \int_{\bbR} \left( \frac{\heat(t+h-r,x-z)\heat(r,z)}{\heat(t+h,x)} -  \frac{\heat(t-r,x-z) \heat(r,z)}{\heat(t,x)}\right)^{2}dz dr \leq 10 h^{\alpha/2}
\end{align*}
\end{lemma}
\begin{proof}
By Lemmas \ref{lem:int1} and \ref{lem:int3},
\begin{align*}
&\int_0^t  \int_{\bbR} \left( \frac{\heat(t+h-r,x-z)\heat(r,z)}{\heat(t+h,x)} -  \frac{\heat(t-r,x-z) \heat(r,z)}{\heat(t,x)}\right)^{2} dz dr \\
&=\int_0^t \int_{\bbR}  \frac{\heat(t+h-r,x-z)^2\heat(r,z)^2}{\heat(t+h,x)^2} + \int_0^t dr \int_{\bbR}dz\, \frac{\heat(t-r,x-z)^2 \heat(r,z)^2}{\heat(t,x)^2}dz dr\\
&- 2 \int_0^t  \int_{\bbR}\frac{\heat(t+h-r,x-z)\heat(r,z)}{\heat(t+h,x)}\frac{\heat(t-r,x-z) \heat(r,z)}{\heat(t,x)}dz dr ,\\
&\leq\int_0^t  \int_{\bbR} \frac{\heat(t+h-r,x-z)^2\heat(r,z)^2}{\heat(t+h,x)^2} + \int_0^t dr \int_{\bbR}dz\, \frac{\heat(t-r,x-z)^2 \heat(r,z)^2}{\heat(t,x)^2} dz dr\\
&= \frac{\sqrt{t+h}}{2\sqrt{\pi}}\left[\arcsin\left(1-\frac{2h}{t+h}\right) + \frac{\pi}{2}\right] +\frac{\sqrt{t\pi}}{2} \leq 10 h^{\alpha/2}.
\end{align*}
In the last step, we bounded $\arcsin(\cdot) \leq \pi/2$, $\sqrt{t+h} \leq \sqrt{2}h^{\alpha/2}$, and used a crude bound on the numerical prefactors.
\end{proof}
Finally, we combine our estimates to obtain the last bound needed for our H\"older estimates.
\begin{proposition}\label{prop:nogap}
For $T,K>1$, $t \in [0,T]$, $x \in [-K,K]$, and $h \in [0,1]$,
\begin{align*}
\int_0^t  \int_{\bbR}\left( \frac{\heat(t+h-r,x-z)\heat(r,z)}{\heat(t+h,x)} -  \frac{\heat(t-r,x-z) \heat(r,z)}{\heat(t,x)}\right)^{2} dz dr   \leq 10T^{3/2}K^2 h^{1/7}.
\end{align*}
If, in addition, for $\delta>0$, we have $t,t+h \in [\delta,T]$, then
\begin{align*}
\int_0^t \int_{\bbR} \left( \frac{\heat(t+h-r,x-z)\heat(r,z)}{\heat(t+h,x)} -  \frac{\heat(t-r,x-z) \heat(r,z)}{\heat(t,x)}\right)^{2} dz dr \leq   \frac{10}{\delta^{3}} T^{3/2}K^2\sqrt{h}.
\end{align*}
\end{proposition}
\begin{proof}
For $h=0$, there is nothing to show, so take $h \in (0,1]$. The first claim holds by Lemma \ref{lem:near0} if $t \in [0,h^{2/7}].$  By Lemmas \ref{lem:int1}, \ref{lem:int3}, and \ref{lem:withgap} (with $\delta = h^{2/7}$), we have for $t \in [h^{2/7},T]$,
\begin{align*}
&\int_0^t  \int_{\bbR} \left( \frac{\heat(t+h-r,x-z)\heat(r,z)}{\heat(t+h,x)} -  \frac{\heat(t-r,x-z) \heat(r,z)}{\heat(t,x)}\right)^{2}dz dr \\
&=\int_0^t \int_{\bbR}\frac{\heat(t+h-r,x-z)^2\heat(r,z)^2}{\heat(t+h,x)^2} dz  dr  + \int_0^t \int_{\bbR} \frac{\heat(t-r,x-z)^2 \heat(r,z)^2}{\heat(t,x)^2} dz  dr \\
&- 2 \int_0^t \int_{\bbR} \frac{\heat(t+h-r,x-z)\heat(r,z)}{\heat(t+h,x)}\frac{\heat(t-r,x-z) \heat(r,z)}{\heat(t,x)}dz  dr  \\
&\leq \frac{\sqrt{t+h}}{2\sqrt{\pi}}\left[\arcsin\left(1-\frac{2h}{t+h}\right) + \frac{\pi}{2}\right] + \frac{\sqrt{t\pi}}{2} - \sqrt{\pi t} + \frac{T}{\sqrt{\pi}}h^{3/14} + \frac{\sqrt{\pi}}{8}T^{3/2}K^2 h^{1/7} \\
&\leq \frac{\sqrt{\pi h}}{2} + \frac{T}{\sqrt{\pi}}h^{3/14} + \frac{\sqrt{\pi}}{8}T^{3/2}K^2 h^{1/7} \leq 10 T^{3/2}K^2 h^{1/7}.
\end{align*}
In the last step, we bounded $\arcsin(\cdot) \leq \pi/2$ and $\sqrt{t+h} \leq \sqrt{t}+\sqrt{h}$. 

If, instead, $t,t+h \in [\delta,T]$ for some fixed $\delta>0$, the same argument gives
\begin{align*}
&\int_0^t  \int_{\bbR} \left( \frac{\heat(t+h-r,x-z)\heat(r,z)}{\heat(t+h,x)} -  \frac{\heat(t-r,x-z) \heat(r,z)}{\heat(t,x)}\right)^{2}dz dr \\
&\leq \frac{\sqrt{t+h}}{2\sqrt{\pi}}\left[\arcsin\left(1-\frac{2h}{t+h}\right) + \frac{\pi}{2}\right] + \frac{\sqrt{t\pi}}{2} - \sqrt{\pi t} + 2\frac{T}{\delta \sqrt{\pi}} \sqrt{h} + 2 h \frac{\sqrt{\pi} T^{3/2}}{8\delta^3}K^2 \\
&\leq  \frac{10}{\delta^{3}} T^{3/2}K^2\sqrt{h}. \qedhere
\end{align*}
\end{proof}

\section{Notation, terminology, and conventions}\label{app:top}
\subsubsection*{Constants in proofs}
Constants are typically denoted $C,C',C'',\dots$ or $c,c',c'',\dots$. In the statements of results, constants reset between results. Within proofs, constants reset for the proof of each claim of a result, unless otherwise indicated.

\subsubsection*{Notation}
The integers are $\bbZ$, the non-negative integers are $\bbZ_+=\{0,1,2,\dots\}$, the natural numbers are $\bbN=\{1,2,\dots\}$, the real numbers in $d$ dimensions are $\bbR^d$, the rational numbers are $\bbQ^d$, and the dyadic rationals are $\bbD^d = \{(\frac{k_1}{2^{n_1}}, \dots, \frac{k_d}{2^{n_d} }) : k_1,\dots,k_d,n_1,\dots,n_d \in \bbZ\}$. The standard coordinate basis vectors in $\bbR^d$ are denoted by $\evec_i, i = 1,2,\dots,d$. For $n \in \bbN$, we denote $[n] = \{1,\dots,n\}$. We denote tuples with subscripts. For $m<n$ with $m,n\in \bbN$, $x_{m:n} = (x_m,x_{m+1},\dots,x_n)$ and $x_{n:m} = (x_n, x_{n-1},\dots,x_m)$. The Weyl chamber in $\R^n$ is denoted by $\bbW_n = \{(x_1,\dots,x_n)\in\R^n : x_1 < \dots < x_n\}$. The maximum of two real numbers $a,b\in\R$ is sometimes denoted by $a \vee b$ and the minimum is sometimes denoted by $a \wedge b$.

For $T,K,\delta>0$, we have the following domains of our various fields of solutions:
\begin{align*}
\varset &= \{(s,y,t,x) \in \bbR^4 : s \leq t\}, \qquad \varsets = \{(s,y,t,x)\in \bbR^4 : s<t\}, \\
\varsetth &= \{(s,t,x) \in \bbR^3 : s<t\}\\ 
\varsett{T}{K} &= \{(s,y,t,x) \in \varset: -T \leq s,t \leq T, -K \leq x,y \leq K\}, \\
\varsetsg{T}{K}{\delta} &= \{(s,y,t,x) \in \varsets:  -T \leq s,t \leq T, -K \leq x,y \leq K, t-s\geq\delta\},\\
\varsetthg{T}{K}{\delta} &= \{(s,t,x) \in \varsetth:  -T \leq s,t \leq T, -K \leq x \leq K, t-s\geq\delta\}.
\end{align*}
\subsubsection*{H\"{o}lder seminorms on functions}
Given $\sK \subset \bbR^d$ and $\alpha \in (0,1]$ and $f \in \sC(\sK,\bbR)$ the $\alpha$-H\"older semi-norm is defined by
\be\label{app-H7}   
|f|_{\sC^\alpha(\sK)} = \sup_{\substack{x_{1:d}, y_{1:d} \,\in \,\sK \\ x_{1:d}\neq y_{1:d}}} \frac{|f(x_1,\dots,x_d) -f(y_1,\dots,y_d)|}{\sum_{i=1}^d|x_i-y_i|^\alpha}
\ee
We define time-space H\"older semi-norms for $\sK \subset \bbR^4$, $f \in \sC(\sK,\bbR)$, and $\alpha,\nu \in (0,1]$ by
\begin{align*}
|f|_{\sC^{\alpha,\nu}(\sK)} &= \sup_{\substack{(t_1,x_1,s_1,y_1)\neq (t_2,x_2,s_2,y_2)\\ (t_i,x_i,s_i,y_i) \in \sK, i\, \in \{1,2\}} } \frac{|f(t_1,x_1,s_1,y_1) - f(t_2,x_2,s_2,y_2)|}{|t_1-t_2|^{\alpha} + |s_1-s_2|^{\alpha} + |x_1-x_2|^\nu + |y_1-y_2|^\nu},
\end{align*}
and, similarly, time-space-inverse temperature H\"older semi-norms for $\sK \subset \bbR^5$, $f \in \sC(\sK,\bbR)$, and $\alpha,\nu,\gamma \in (0,1]$ by
\begin{align*}
|f|_{\sC^{\alpha,\nu,\gamma}(\sK)} &= \sup_{\substack{(t_1,x_1,s_1,y_1,\beta_1)\neq \\
(t_2,x_2,s_2,y_2,\beta_2)\\ (t_i,x_i,s_i,y_i,\beta_i) \in \sK, i\, \in \{1,2\}} } \frac{|f(t_1,x_1,s_1,y_1,\beta_1) - f(t_2,x_2,s_2,y_2,\beta_2)|}{|t_1-t_2|^{\alpha} + |s_1-s_2|^{\alpha} + |x_1-x_2|^\nu + |y_1-y_2|^\nu + |\beta_1-\beta_2|^\gamma}.
\end{align*}

\subsubsection*{Topological conventions and notation}
In this section $X$ denotes an arbitrary metric space. $\sC(X,\bbR)$ is the space of continuous functions from $X$ to $\bbR$  with the topology of uniform convergence on compact sets. $\sC_b(X,\bbR)$ is the space of bounded continuous functions is equipped with the supremum norm. 

$\sB(X)$ is the Borel $\sigma$-algebra of $X$. The bounded Borel measurable functions on $X$ are denoted by $\sB_b(X)$. A measure $\mu$ on $(X,\sB(X))$ is said to be positive if $\mu(B) \in [0,\infty]$ for all $B \in \sB(X)$ and signed if $\mu(B) \in \bbR$ for all $B \in \sB(X)$. The zero measure $\zeromeas$ assigns measure $0$ to all $B \in \sB(X)$. A measure is non-zero if it is not the zero measure.  

We denote by $\sC_c(X,\bbR)$ the space of compactly supported continuous functions equipped with the supremum norm. 
 $\sC_c(X,\bbR_+)$ is the space of such functions which are also non-negative.
The positive and locally finite Borel measures on $\bbR^d$ are $\sM_+(\bbR^d,\sB(\bbR))$. We say that $\mu_n \in \sM_+(\bbR^d,\sB(\bbR))$ converges to $\mu \in \sM_+(\bbR^d,\sB(\bbR))$ vaguely if $\int_{\bbR^d} \varphi(x) \mu_n(dx) \to \int_{\bbR^d} \varphi(x)\mu(x)$ for all $\varphi \in \sC_c(\bbR^d,\bbR)$. 
$\sM_1(X)$ is the space of probability measures on $X$. When restricting attention to finite positive measures, we typically use the weak topology, where the test functions come from $\sC_b(\bbR^d,\bbR)$. See \cite[Definition 8.1.2]{Bog-07}.

The space $\ICM\cup\{\zeromeas\}$ defined in equation \eqref{eq:ICM} admits a natural Polish topology, which we metrize as follows. Let $\{\varphi_j : j \in \bbR\} \in \sC_c(\R,\R_+)$ be a countable dense subset of $\sC_c(\R,\R_+)$. Define for $\zeta,\eta\in\ICM$,
\begin{align}
d_{\ICM}(\zeta,\eta) &= \sum_{j=1}^\infty 2^{-j} \biggl( 1\wedge \bigg\{\biggl\lvert\int_{\R}\varphi_j d\zeta - \int_{\R}\varphi_j d\eta \biggr\rvert \bigg\} \biggr) \label{eq:ICMm} \\
&\qquad +  \sum_{m=1}^\infty  2^{-m} \biggl( 1\wedge \biggl\lvert\int_{\R} e^{-\tspb\frac1my^2} \zeta(dy)  - \int_{\R}e^{-\tspb\frac1my^2} \eta(dy)  \biggr\rvert \biggr).  \notag
\end{align}
\begin{lemma}\label{lem:ICM}
$(\ICM\cup\{\zeromeas\},d_{\ICM})$ is a complete separable metric space.
\end{lemma}
\begin{proof}
Note that the sum over $j$ in the definition of $d_{\ICM}$ metrizes the vague topology on $\sM_+(\R)$. Let $\{\zeta_n\}$ be a Cauchy sequence in $(\ICM,d_{\ICM})$. Then there is a vague limit $\zeta_n\to\zeta$.  By completeness of $\R$, there exist $\{a_m : m \in \bbN \}$ with 
\[   a_m = \lim_{n\to\infty}   \int_{\R} e^{-\tspb\frac1my^2} \zeta_n(dy) \qquad \text{ and }\qquad A_m = \sup_n    \int_{\R} e^{-\tspb\frac1my^2} \zeta_n(dy)  <\infty. \]  
To conclude completeness, we need to show that  
\be\label{gg62}  a_m=  \int_{\R} e^{-\tspb\frac1my^2} \zeta(dy). \ee

Fix $\sC_c(\R,\R_+)$-functions  $\{\psi_k: k\in\N\}$ that satisfy $\ind_{[-k,k]} \le\psi_k\le \ind_{[-k-1,k+1]}$.  Vague convergence implies that 
\begin{align*}
a_m \ge  \lim_{n\to\infty}   \int_{\R} e^{-\tspb\frac1my^2} \tspb \psi_k(y) \tspb \zeta_n(dy) 
=   \int_{\R} e^{-\tspb\frac1my^2} \tspb \psi_k(y) \tspb \zeta(dy)   
\nearrow  \int_{\R} e^{-\tspb\frac1my^2}  \zeta(dy)  \quad\text{ as } k\nearrow\infty. 
\end{align*}  

Now, let $m<\ell$ and $k\in\N$. 
\be\label{gg67} \begin{aligned}
\int_{[k,\infty)} e^{-\tspb\frac1my^2}   \tspb \zeta_n(dy)  &= \int_{[k,\infty)} e^{ y^2(\frac1\ell-\frac1m)}  e^{-\tspb\frac1\ell y^2}   \tspb \zeta_n(dy)   \le e^{ k^2(\frac1\ell-\frac1m)}   \int_{[k,\infty)}  e^{-\tspb\frac1\ell y^2}   \tspb \zeta_n(dy) \\
&   \le e^{ k^2(\frac1\ell-\frac1m)}  A_\ell. 
\end{aligned}\ee   

Fix $m\in\N$. 
 Since $ \int_{\R} e^{-\tspb\frac1my^2} \zeta(dy)<\infty$, 
\be\label{gg71}   \lim_{k\to\infty}   \int_{\R} e^{-\tspb\frac1my^2} \tspb \bigl(1- \psi_k(y)\bigr)  \tspb \zeta(dy)  =0. \ee 
 Take $\ell=2m$ in \eqref{gg67} and let $k\in\N$.   Let   $o_k(1)$ denote a quantity that depends on $(k,n)$ and  vanishes  as $n\to\infty$ when $k$ is fixed. 
 \begin{align*}
&\biggl\lvert   \int_{\R} e^{-\tspb\frac1my^2} \zeta(dy) -  a_m \biggr\rvert  
\le   \biggl\lvert   \int_{\R} e^{-\tspb\frac1my^2}  \tspb \psi_k(y) \tspb   \zeta(dy) -   \int_{\R} e^{-\tspb\frac1my^2} \tspb \psi_k(y) \tspb \zeta_n(dy)   \biggr\rvert  \\
&\qquad  \qquad 
  +  \int_{\R} e^{-\tspb\frac1my^2} \tspb \bigl(1- \psi_k(y)\bigr)  \tspb \zeta(dy) + \int_{[k,\infty)} e^{-\tspb\frac1my^2}   \tspb \zeta_n(dy)      + \biggl\lvert     \int_{\R} e^{-\tspb\frac1my^2}  \tspb \zeta_n(dy)  -  a_m \biggr\rvert    \\
  &\qquad  \le o_k(1) +    \int_{\R} e^{-\tspb\frac1my^2} \tspb \bigl(1- \psi_k(y)\bigr)  \tspb \zeta(dy)   +  e^{ -\tspb\frac1{2m} k^2}  A_{2m} +o(1) .  
   \end{align*} 
First keep $k$ fixed and  let $n\to\infty$ to remove $o_k(1)+o(1)$. Then let $k\to\infty$.   \eqref{gg62}  has been verified.  It remains to show separability. 

We claim that measures of the form $\sum_{i=1}^n a_i \delta_{b_i}$ where $a_i \in \bbQ\cap(0,\infty)$ and $b_i \in \bbQ$ are dense. It suffices to show that for each $M,J\in \bbN$, each $\epsilon \in(0,1)$, and each $\zeta \in \ICM$, there exists $n \in \bbN$ and $a_{1:n},b_{1:n}$ as above so that for $\eta = \sum_{i=1}^n a_i \delta_{b_i}$ and all $j \in [J]$ and $m \in [M]$,
\begin{align*}
\biggl\lvert\int_{\R}\varphi_j d\zeta - \int_{\R}\varphi_j d\eta \biggr\rvert<\epsilon \qquad\text{ and } \qquad \biggl\lvert\int_{\R} e^{-\tspb\frac1my^2} \zeta(dy)  - \int_{\R}e^{-\tspb\frac1my^2} \eta(dy)  \biggr\rvert<\epsilon.
\end{align*}
Fix $K>0$ so that for all $m \in [M]$, 
\begin{align*}
\int_{\R\backslash[-K,K]} e^{-\tspb\frac1my^2} \zeta(dy)  <\epsilon/2
\end{align*}
and $\supp \varphi_j \subset[-K,K]$. The result now follows from density of measures of the form $\sum_{i=1}^n a_i \delta_{b_i}$ in the space of finite positive measures on $[-K,K]$.
\end{proof}

We also introduce a metric on the space of strictly positive continuous functions representing measures in $\ICM$, which we denoted by $\CICM$ in equation \eqref{eq:CICM} above: 
\begin{align*}
\CICM = \biggl\{f \in \sC(\R,(0,\infty)) : \forall a>0, \int_{\bbR} e^{-ax^2}f(x)dx <\infty \biggr\}.
\end{align*}
We equip this space with the metric defined for $f,g \in \CICM$ by
\begin{align}
d_{\CICM}(f,g) &= \sum_{m=1}^\infty 2^{-m} \bigg(1\wedge \sup_{-m \leq x \leq m}\bigg[|f(x)-g(x)| + \bigg|\frac{1}{f(x)} -\frac{1}{g(x)}\bigg|\bigg]\bigg) \label{eq:CICMm} \\
&+\sum_{m=1}^\infty  2^{-m} \biggl( 1\wedge \biggl\lvert\int_{\R} e^{-\tspb\frac1my^2} f(y)dy  - \int_{\R}e^{-\tspb\frac1my^2} g(y)dy \biggr\rvert \biggr) \notag\end{align}
We have the following.
\begin{lemma}
$(\CICM,d_{\CICM})$ is a complete separable metric space.
\end{lemma}
\begin{proof}
The first term in the metric ensures that Cauchy sequences in $(\CICM,d_{\CICM})$ are Cauchy under the supremum norm on any interval $[-m,m]$ for $m \in \bbN$, hence a locally uniform limit taking values in $\sC(\R,\R_+)$ exists.  The second term ensures that the locally uniform limit point is strictly positive.  Locally uniform convergence implies vague convergence of the represented measures. The last term in the metric then ensures that these measures are Cauchy in $\ICM$ and so, by Lemma \ref{lem:ICM}, the integrals against Gaussian kernels converge as well. Therefore $
(\CICM,d_{\CICM})$ is complete. Separability can be seen by taking a dense subset of compactly supported continuous functions in $\sC(\R,\R_+)$ and then adding small rational positive $\epsilon$ times Gaussian kernels to each term to make the functions strictly positive.
\end{proof}

\subsubsection*{Total Variation of (Formal) Signed Measures}
Given a signed measure $\mu$ on $(X,\sB(X))$, the total variation measure $|\mu|$ is given by the sum of the positive and negative parts in its Jordan-Hahn decomposition. See \cite[Definition 3.1.4]{Bog-07}. The total variation norm is $\|\mu\|_{TV} = |\mu|(X)$, which satisfies $\|\mu\|_{TV}\leq 2\sup\{|\mu(A)| : A \in \sB(X)\} \leq 2 \|\mu\|_{TV}$.  Given two finite positive Borel measures $\mu,\zeta$ on $\bbR^d$, we denote by $|\mu-\zeta|$ the total variation measure assigned to the signed measure given by their difference. Given two locally finite positive Borel measures $\mu,\zeta$ on $\bbR^d$, the difference $\mu-\zeta$ may not define a signed measure, but we \textit{define} the positive measure corresponding to the total variation of the difference by $|\mu-\zeta|(A) = \lim_{M\to\infty}|\mu|_{[-M,M]^d} - \zeta|_{[-M,M]^d}|(A) = \lim_{M\to\infty}|\mu - \zeta|(A \cap [-M,M]^d)$, where the measures $\mu|_{[-M,M]^d}$ and $\zeta|_{[-M,M]^d}$ are the finite measures obtained by restricting $\mu$ and $\zeta$ to $[-M,M]^d$. The limit exists by monotonicity.

\subsubsection*{Stochastic Processes}
We call $\sC = \sC(\bbR,\bbR)$ and denote by $X= (X_t : t \in \bbR)$ the canonical process on $(\sC,\sB(\sC)).$ For $-\infty \leq s < t \leq \infty$,  $\cfil_{s:t} = \sigma(X_u : s < u < t)$ and $\cfil_{t:t}=\sigma(X_t)$. For $-\infty < s \leq t < \infty$, the spaces $\sC_{[s,t]}=\sC([s,t],\bbR)$ of real-valued continuous functions on $[s,t]$, equipped with the uniform topology and Borel $\sigma$-algebra $\sB(\sC_{[s,t]})$, are naturally embedded into $(\sC,\sB(\sC))$ by restriction. We abuse notation and at times continue to use the notation $X$ to denote $X|_{[s,t]} = (X_u : s \leq u \leq t)$. 

If $\sA$ is set and $F$ and $G$ are stochastic processes on a complete probability space $(\Omega,\sF,\bbP)$ indexed by $\sA$, then we say that $F$ and $G$ are \textit{modifications} of one another if for all $\alpha \in \sA$,
\begin{align*}
\bbP(F(\alpha) = G(\alpha))=1. 
\end{align*}
We say that $F$ and $G$ are \textit{indistinguishable} if
\begin{align*}
\bbP(\text{for all }\alpha \in \sA, F(\alpha) = G(\alpha))=1.
\end{align*}

\subsubsection*{Stochastic Ordering}
We say that a function $F$:\,$\sC(\R,\R) \to\bbR$ (resp.~$F$:\,$\sC([s,t],\R) \to \R$) is increasing if $F(X) \leq F(Y)$ whenever $X_t \leq Y_t$ for all $t$.  Given two probability measures $\bfP$ and $\bfQ$ on $\sC(\R,\R)$ (resp.~$\sC([s,t],\R)$), $\bfQ$ stochastically dominates $\bfP$, denoted $\bfP \std \bfQ$, if $\int F(X) \bfP(dX) \leq \int F(X) \bfQ(dX)$ for all increasing $F \in \sC_b(\sC(\R,\R),\R)$ (resp.~$F \in \sC_b(\sC([s,t],\R),\R)$).


\section*{Statements and declarations}
\subsection*{Data availability statement} Data sharing is not applicable to this article as no datasets were generated or analyzed during the current study.

\subsection*{Conflicts of interest/Competing interests} The authors have no conflicts of interest to declare that are relevant to the content of this article.

\small

\bibliographystyle{plain}
\bibliography{PAM}
\end{document}